\documentclass[11pt,reqno]{amsart}
\usepackage[a4paper,margin=30mm,headheight=14pt]{geometry}
\usepackage{amsmath,amssymb,amsthm,mathtools,mathrsfs,booktabs,array}
\usepackage{xcolor,enumitem,etoolbox,microtype,needspace}
\usepackage[T1]{fontenc}
\usepackage{lmodern}
\usepackage{hyperref}
\usepackage{aliascnt}
\usepackage[nameinlink,capitalize,noabbrev]{cleveref}
\hypersetup{
  colorlinks=true,
  linkcolor=blue!45!black,
  citecolor=blue!45!black,
  urlcolor=blue!45!black,
  pdftitle={Sharp noncommutative Hilbert and Riesz inequalities: classification and spectral corrections},
  pdfauthor={Zhouyu Long},
  pdfsubject={Sharp noncommutative Hilbert and Riesz bounds, classification and spectral corrections},
  pdfkeywords={Hilbert transform, Schatten class, Volterra operator, Riesz-Titchmarsh transform, Riesz projection, subdiagonal algebra, spectral correction}
}
\numberwithin{equation}{section}

\newtheorem{theorem}{Theorem}[section]
\theoremstyle{plain}
\newaliascnt{lemma}{theorem}
\newtheorem{lemma}[lemma]{Lemma}
\aliascntresetthe{lemma}
\theoremstyle{plain}
\newaliascnt{proposition}{theorem}
\newtheorem{proposition}[proposition]{Proposition}
\aliascntresetthe{proposition}
\theoremstyle{plain}
\newaliascnt{corollary}{theorem}
\newtheorem{corollary}[corollary]{Corollary}
\aliascntresetthe{corollary}
\theoremstyle{definition}
\newaliascnt{definition}{theorem}
\newtheorem{definition}[definition]{Definition}
\aliascntresetthe{definition}
\theoremstyle{remark}
\newaliascnt{remark}{theorem}
\newtheorem{remark}[remark]{Remark}
\aliascntresetthe{remark}
\theoremstyle{remark}
\newaliascnt{example}{theorem}
\newtheorem{example}[example]{Example}
\aliascntresetthe{example}
\theoremstyle{plain}
\crefname{theorem}{Theorem}{Theorems}
\Crefname{theorem}{Theorem}{Theorems}
\crefname{lemma}{Lemma}{Lemmas}
\Crefname{lemma}{Lemma}{Lemmas}
\crefname{proposition}{Proposition}{Propositions}
\Crefname{proposition}{Proposition}{Propositions}
\crefname{corollary}{Corollary}{Corollaries}
\Crefname{corollary}{Corollary}{Corollaries}
\crefname{definition}{Definition}{Definitions}
\Crefname{definition}{Definition}{Definitions}
\crefname{remark}{Remark}{Remarks}
\Crefname{remark}{Remark}{Remarks}
\crefname{example}{Example}{Examples}
\Crefname{example}{Example}{Examples}
\newcommand{\R}{\mathbb R}\newcommand{\C}{\mathbb C}\newcommand{\Z}{\mathbb Z}
\newcommand{\T}{\mathbb T}\newcommand{\D}{\mathbb D}
\newcommand{\tr}{\operatorname{tr}}\newcommand{\spec}{\operatorname{spec}}
\newcommand{\supp}{\operatorname{supp}}\newcommand{\Rea}{\operatorname{Re}}\newcommand{\Ima}{\operatorname{Im}}
\newcommand{\Sp}{\mathcal S}\newcommand{\dd}{\,\mathrm d}\newcommand{\Hs}{\mathcal H}
\newcommand{\Euler}{\mathcal E}\newcommand{\EulerQ}{\mathcal Q}
\newcommand{\av}{\mathop{\mathrm{av}}\nolimits_\theta}
\newcommand{\diag}{\operatorname{diag}}
\newcommand{\norm}[1]{\left\lVert#1\right\rVert}

\newcommand{\Lp}{\mathcal L_p}
\newcommand{\HSch}{\mathscr H_0}
\newcommand{\id}{\operatorname{id}}

\newcommand{\M}{\mathcal M}
\newcommand{\B}{\mathcal B}
\newcommand{\Id}{I}
\newcommand{\Ds}{\mathcal D_s}
\title[Sharp noncommutative Hilbert and Riesz inequalities]{Sharp noncommutative Hilbert and Riesz inequalities:\\classification and spectral corrections}
\author[Zhouyu Long]{\bfseries Zhouyu Long$^{*}$}
\thanks{$^{*}$Department of Mathematical Sciences, Tsinghua University,
Beijing 100084, China.\newline
E-mail: \href{mailto:longzy25@mails.tsinghua.edu.cn}{\nolinkurl{longzy25@mails.tsinghua.edu.cn}}}

\makeatletter
\patchcmd{\@setauthors}{\centering\footnotesize}{\centering\large}{}%
  {\ClassError{amsart}{Unable to set author font}{Check the amsart author layout.}}
\patchcmd{\@setauthors}{\MakeUppercase{\authors}}{\authors}{}%
  {\ClassError{amsart}{Unable to preserve author case}{Check the amsart author layout.}}
\def\author@andify{\nxandlist{\unskip,\quad}{\unskip,\quad}{\unskip,\quad}}
\let\addresses\@empty
\makeatother
\date{}

\begin{document}
\begin{abstract}
For every $1<p<\infty$, with $p^*=\max\{p,p/(p-1)\}$,
we determine the sharp Hilbert norm $C_p=\cot(\pi/(2p^*))$
on $L^p(\R;\Sp_p)$ and the sharp uniform bound $D_p=\csc(\pi/p)$
for retained and strict Riesz projections of maximal subdiagonal
algebras in finite tracial von Neumann algebras.
The Hilbert value is also the common Fourier--Schur completely
bounded norm.
For finite matrices, a nonnegative logarithmic determinant defect
connects Hein\"avaara's tracial joint spectral measure to eigenvalues.
The resulting transfer identity classifies finite signed real ridge
inequalities on nilpotents and identifies their least additive
eigenvalue corrections. A dual theorem characterizes weak-* limits
of weighted angular measures of finite nilpotents.
In general finite tracial algebras, separate positive-kernel and
sharp determinant-covariance comparisons yield the Riesz bounds.
Consequences include the sharp Volterra constant, Riesz--Titchmarsh
and shifted-transform formulas, joint real-pencil norms, and exact
triangular and operator-valued Fourier Riesz norms.
\end{abstract}
\maketitle

\section{Introduction}\label{sec:intro}
Throughout, $1<p<\infty$, $p'=p/(p-1)$, and
$p^*=\max\{p,p'\}$. We determine the noncommutative sharp constants
\[
 C_p=\cot\frac{\pi}{2p^*},\qquad D_p=\csc\frac{\pi}{p}.
\]
The Hilbert transform on $L^p(\R;\Sp_p)$ and the ordered Hilbert
Schur multiplier have norm $C_p$. Every maximal subdiagonal algebra
in a finite tracial von Neumann algebra has retained and strict
Riesz projections bounded by $D_p$, optimally over this class.
The first three theorems state the analytic-distance bound, the
determinant-covariance estimate, and the classification with least
spectral correction.

The scalar Hilbert value is due to Pichorides \cite{Pichorides},
and the real-pencil values to Hollenbeck--Kalton--Verbitsky
\cite{HKV}. Neuwirth--Ricard identified the Hilbert Fourier--Schur
and completely bounded norms for every exponent
\cite[Theorem~6.2]{NR}, and computed their value for powers of two
and their conjugate exponents. We determine that value throughout
the exponent range. The shifted-transform formula includes the
Riesz--Titchmarsh conjecture at the half shift and Laeng's full
Conjecture~5.7 \cite{Laeng}; it extends the even-integer and
conjugate-exponent cases of Ba\~nuelos--Kwa\'snicki \cite{BKRT}.
The Volterra consequence gives the conjectured sharp comparison
of the self-adjoint parts of compact quasinilpotent operators,
recorded in \cite{MS}; see also the classical work
\cite{Matsaev61,GK68,GK70}.

Hollenbeck--Verbitsky determined the scalar Riesz value $D_p$
\cite{rz:HV}. Randrianantoanina proved boundedness for finite maximal
subdiagonal algebras \cite{rz:Rand}. Neuwirth--Ricard identified the
triangular Schatten norm with the completely bounded Fourier Riesz
norm and computed its value at $p=4$ \cite[Proposition~6.3]{NR};
duality gives $p=4/3$, besides the elementary case $p=2$.
Here the sharp uniform upper bound is proved directly in arbitrary
finite tracial algebras.

The proofs use two spectral comparisons with exact nonnegative
remainders. For finite matrices, our logarithmic determinant defect
compares eigenvalues with Hein\"avaara's positive tracial joint
spectral measure \cite{Hein}. The ridge transform and its
one-variable differential inverse are already present in
\cite[Remark~2, p.~40]{HeinThesis}; the additional ingredient here
is the nonnegative eigenvalue defect and its optimal spectral
correction. For finite tracial algebras, logarithmic rearrangement
and determinant covariance give the analytic witnesses needed for
the Riesz theorem. Remark~\ref{rem:mergedscope} specifies the
distinct scopes of these comparisons.

\subsection{Sharp Riesz bounds in finite tracial algebras}
For the subdiagonal results, $(\M,\tau)$ is a finite von Neumann algebra with a faithful normal tracial state and $\tau(\Id)=1$. For $1\leq p<\infty$, $L^p(\M)$ is the completion of $\M$ for
$\|x\|_p=(\tau|x|^p)^{1/p}$, identified with the corresponding affiliated measurable operators. We use the trace pairing $\langle x,y\rangle=\tau(y^*x)$, H\"older's inequality, cyclicity for integrable products, and $L^p$--$L^{p'}$ duality, where $p'=p/(p-1)$. Since the trace is normalized, $\|x\|_q\leq\|x\|_p$ when $q\leq p$.

A \emph{subdiagonal algebra} is a weak-star closed unital algebra $\B\subset\M$ for which $\mathcal D=\B\cap\B^*$ admits a trace-preserving normal conditional expectation $\Phi:\M\to\mathcal D$, multiplicative on $\B$, and $\B+\B^*$ is weak-star dense in $\M$. It is \emph{maximal} if it is maximal among the subdiagonal algebras for the same expectation. Write $\B_0=\B\cap\ker\Phi$ and
\[
 H^p=\overline{\B}^{\,L^p},\qquad H_0^p=\overline{\B_0}^{\,L^p},\qquad
 \mathcal C=\B+\B_0^*.
\]
On this bounded core the retained-diagonal projection is $P(U+Y)=U$ and the strict projection is $P_0=P-\Phi$. Their well-definedness and core density are proved below.

\Needspace{9\baselineskip}
\begin{theorem}\label{rz:allp}
For every maximal subdiagonal algebra in $(\M,\tau)$ and every
$1<p<\infty$,
\begin{align}
 \|U+Y\|_p&\ge\sin\frac\pi p\,\|U\|_p,
 &U&\in H^p,\quad Y\in(H_0^p)^*,\notag\\
 \|U_0+Y_0\|_p&\ge\sin\frac\pi p\,\|U_0\|_p,
 &U_0&\in H_0^p,\quad Y_0\in(H^p)^*.
 \label{rz:distance}
\end{align}
Here $(H^p)^*=\{x^*:x\in H^p\}$ and similarly for $H_0^p$.
Consequently $P$ and $P_0$ extend to bounded operators on $L^p(\M)$,
and
\begin{equation}\label{rz:Pbound}
 \quad\|P\|_{L^p\to L^p}\leq\csc\frac\pi p,
 \qquad \|P_0\|_{L^p\to L^p}\leq\csc\frac\pi p.\quad
\end{equation}
The distance coefficient and both projection bounds are optimal
uniformly over this class. Equality is not asserted for every
individual algebra.
\end{theorem}

\paragraph{The new upper bound.}
The scalar sharp value supplies a lower bound; the known
Fourier--Schur transference relations do not supply the missing
all-exponent upper bound. We prove that upper bound in the finite
algebra itself. Determinant covariance handles $2<p<4$;
positive-kernel comparison and independently established analytic
regularization handle $p>4$. Continuity, duality, and a subdiagonal
matrix construction complete the retained and strict bounds.
Section~\ref{rz:applicationssection} then gives the exact uniform
triangular norm and the exact Fourier Riesz norm with any nonzero
finite coefficient algebra
(Corollaries~\ref{rz:triangularnorm} and~\ref{rz:Fourier_exact}).

The trace inequality below supplies the middle-exponent estimate.
Here $\operatorname{Det}_\tau$ is the Fuglede--Kadison determinant,
with its definition and structural properties recalled in
Appendix~\ref{rz:prelim}.

\Needspace{9\baselineskip}
\begin{theorem}\label{rz:detcovariance}
Let $(\M,\tau)$ be any finite von Neumann algebra with a faithful
normal trace satisfying $\tau(\Id)=1$, and let $N\in\M$ satisfy
$\operatorname{Det}_\tau(\Id+\zeta N)\ge1$ for every $\zeta\in\mathbb C$.
Then $\tau N=0$, and for every bounded invertible $D>0$ and $0<\beta<1$,
\begin{equation}\label{rz:weightedineq}
 |\tau D^\beta N|^2\le
 \sin^2\frac{\pi\beta}{2(1+\beta)}\,
 \tau D^{1+\beta}\,\tau(N^*D^{\beta-1}N).
\end{equation}
When $v=\tau(N^*D^{\beta-1}N)>0$, the deficit has the nonnegative
entropy and determinant residual \eqref{rz:covariance_certificate}.
For each fixed $\beta$, the squared
coefficient is optimal uniformly over normalized finite tracial algebras;
this is proved after the circle construction in Section~\ref{rz:completion}.
\end{theorem}

For $N\in\B_0$ or $N\in\B_0^*$, the determinant hypothesis follows
from the general Jensen inequality
(Corollary~\ref{rz:weighted}). With $\beta=(p-2)/2$, the squared
coefficient in \eqref{rz:weightedineq} is $\cos^2(\pi/p)$.
Analytic factorization turns this covariance estimate into the
$L^2$ angle bound used for $2<p<4$. The proof and residual are given
in Section~\ref{rz:sec:spectral}; sharpness follows from the scalar
circle construction in Section~\ref{rz:completion}.

\subsection{Matrix classification and optimal spectral correction}
The underlying questions are: which finite signed real ridge
inequalities hold for all nilpotent matrices, and what is their
least additive eigenvalue correction for arbitrary matrices?
Theorem~\ref{thm:classification} answers both through the least
homogeneous superharmonic majorant, extending the scalar
analytic-test/majorant principle of \cite[Theorem~2.1]{HKV}.

All Hilbert spaces are complex; traces are unnormalized unless
specified otherwise. Let $m_2$ denote planar Lebesgue measure and
set $\Delta=\partial_x^2+\partial_y^2$.
We write $\|\cdot\|_p$ for the Schatten $p$-norm; unadorned
matrix norms are operator norms. A compact quasinilpotent operator
is called a Volterra operator. For $T=A+iB$, set
$A=\Rea T=(T+T^*)/2$ and $B=\Ima T=(T-T^*)/(2i)$.
For real coefficients $c_j,a_j,b_j$, define the finite signed
real ridge sum and its trace functional by

\begin{align}
 F(x+iy)&=\sum_{j=1}^m c_j|a_jx+b_jy|^p,\label{eq:F}\\
 \mathcal F_p(A,B)&=\sum_{j=1}^m c_j\tr|a_jA+b_jB|^p,
 \qquad A=A^*,\quad B=B^*.\label{eq:traceF}
\end{align}
Let $\mathcal U_p$ be the class of real $p$-homogeneous functions
$U\in C^1(\C)\cap W^{2,q}_{\mathrm{loc}}(\C)$ for some $q>1$ such
that $\Delta U\le0$ distributionally.  Thus $U(rz)=r^pU(z)$ for
$r\ge0$.

\begin{theorem}\label{thm:classification}
For $F$ in \eqref{eq:F}, the following are equivalent.
\begin{enumerate}
\item[(N)] $\mathcal F_p(\Rea N,\Ima N)\le0$ for every finite
nilpotent complex matrix $N$.
\item[(P)] For every complex polynomial $q$ with $q(0)=0$,
\begin{equation}\label{eq:disccondition}
 \int_0^{2\pi}F(q(e^{it}))\frac{\dd t}{2\pi}\le0.
\end{equation}
\item[(S)] There exists $U\in\mathcal U_p$ such that $U\ge F$ on
$\C$.
\end{enumerate}
When these conditions hold there is a least superharmonic majorant
$U_*$ of $F$.  It is $p$-homogeneous, even, and belongs to
$\mathcal U_p$.  Writing
$U_*(re^{i\theta})=r^pu(\theta)$ and
$F(re^{i\theta})=r^pf(\theta)$, one has
\begin{equation}\label{eq:obstacle}
 u\ge f,\qquad u''+p^2u\le0,\qquad
 u''+p^2u=0\quad\text{on }\{u>f\}.
\end{equation}
It is enough in (N) to test $N=q(S_n)$, where $S_n$ is the nilpotent
shift and $q(0)=0$.

Moreover, for every $z\in\C$,
\begin{align}
 U_*(z)
 &=\sup_{\substack{n\ge1,\ T\in M_n(\C)\\\spec T=\{z\}}}
       \frac1n\mathcal F_p(\Rea T,\Ima T)\notag\\
 &=\sup_{\substack{n\ge1,\ q\ {\rm polynomial}\\q(0)=0}}
       \frac1n\mathcal F_p\bigl(\Rea(zI_n+q(S_n)),
                                  \Ima(zI_n+q(S_n))\bigr).
 \label{eq:optimalcorrection}
\end{align}
Consequently a real-valued function $G$ satisfies
\begin{equation}\label{eq:universalcorrection}
 \mathcal F_p(\Rea T,\Ima T)\le\sum_{k=1}^nG(\lambda_k(T))
 \quad\text{for every finite }T
\end{equation}
if and only if $G\ge U_*$. Eigenvalues are counted with algebraic
multiplicity. The comparison holds for arbitrary real-valued $G$.
\end{theorem}

\paragraph{The spectral mechanism.}
Hein\"avaara's positive tracial joint spectral measure $\mu_{A,B}$
\cite{Hein} represents real directional trace
functionals. Corollary~\ref{thm:certificate} connects it to the
eigenvalues of $T=A+iB$ through a nonnegative potential $\Psi_T$:
for every $U\in\mathcal U_p$ with $U\ge F$,
\[
\begin{aligned}
 \sum_kU(\lambda_k(T))-\mathcal F_p(A,B)
 &=p(p+1)\int_{\C}(U-F)\dd\mu_{A,B}\\
 &\quad+\frac1{2\pi}\int_{\C}\Psi_T(-\Delta U)\dd m_2.
\end{aligned}
\]
Both remainders are finite and nonnegative. This gives the matrix
inequality, while translated Toeplitz tests establish the pointwise
optimal correction \eqref{eq:optimalcorrection}.
The mass of $\Psi_T$ is $\frac\pi2(\|T\|_2^2-\sum_k|\lambda_k(T)|^2)$,
linking Henrici's departure from normality \cite{Henrici} to the
sharp quantitative Weyl comparison in Corollary~\ref{cor:weylquant}.

\paragraph{Angular measures and inverse realization.}
Theorem~\ref{thm:density} gives a converse realization theorem:
a finite positive measure $\nu$ on $\T$ is a weak-* limit of the
weighted angular measures of single finite nilpotent matrices if
and only if
\[
 \nu=(\partial_\theta^2+p^2)\omega,
 \qquad \omega\in C(\T),\quad\omega\ge0.
\]
The matrix dimensions may grow. Separately,
Corollary~\ref{cor:momentcone} derives the finite directional moment
cones and their optimal quotients directly from the classification
by finite-dimensional separation.

At $p=4$, Theorem~\ref{thm:fourmomentdisk} identifies the closed disk
of normalized directional moments. Corollary~\ref{cor:quarticquotients}
gives the optimal quotient in closed form for every signed numerator
and every signed denominator strictly positive on this disk.
The signed classification and joint norm formulas follow by substitution.

\subsection{Operator models and the common constant}\label{subsec:hilbertintro}
Use
$Hf(x)=\pi^{-1}\operatorname{p.v.}\int f(t)/(x-t)\dd t$ and
\begin{equation}\label{eq:Ralpha}
 (R_\alpha V)_\ell=\frac{\sin(\pi\alpha)}{\pi}
       \sum_{m\in\Z}\frac{V_m}{\ell+\alpha-m},\qquad 0<\alpha<1.
\end{equation}
Initially defined on finitely supported sequences, $R_\alpha$
extends boundedly to $\ell^p(\Z;\Sp_p(\mathcal K))$ by
Lemma~\ref{lem:finitewindow}; $\mathcal K=\C$ gives the scalar case.

For a finite family $\mathcal E=(E_j)$ of pairwise orthogonal
projections with $\sum_jE_j=I$ on a finite-dimensional Hilbert
space, define
\begin{equation}\label{eq:flagH}
 D_{\mathcal E}X=\sum_jE_jXE_j,\qquad
 H_{\mathcal E}X=i\sum_{j,k}\operatorname{sgn}(j-k)E_jXE_k.
\end{equation}
Here $D_{\mathcal E}$ is the block-diagonal conditional expectation,
and $(E_j)$ are the successive difference projections of the flag.
We use $\operatorname{sgn}(0)=0$. Suprema over $\mathcal E$ allow all
finite ambient dimensions and block sizes, independently of
$\mathcal K$. Write $\Sp_p^d=\Sp_p(\C^d)$ and define, initially on
finitely supported matrices on $\ell^2(\mathbb N)$,

\begin{equation}\label{eq:standardSchur}
 \HSch(e_{jk})=i\operatorname{sgn}(j-k)e_{jk}.
\end{equation}
The spaces $L^p(\R)$ and $\Sp_p=\Sp_p(\ell^2(\mathbb N))$ carry their
canonical noncommutative $L^p$ operator-space structures. For a linear
map $T:E\to F$ its completely bounded norm means
\begin{equation}\label{eq:cbdefinition}
 \|T\|_{\mathrm{cb},E\to F}
 =\sup_{d\ge1}\|\id_{M_d}\otimes T\|_{M_d[E]\to M_d[F]}.
\end{equation}
Section~\ref{subsec:cbproof} recalls the characterization of this
norm by Schatten $p$-amplifications.

For real pairs $(a_j,b_j)_{j=1}^m$ and weights $w_j\ge0$, set
\begin{align}
 P(x,y)&=\sum_{j=1}^m w_j|a_jx+b_jy|^p,\notag\\
 \Lambda_p(P)^p
 &=\sup_{N\ne0\ {\rm finite\ nilpotent}}
 \frac{\sum_j w_j\tr|a_j\Rea N+b_j\Ima N|^p}
      {\tr|\Rea N|^p}.\label{eq:jointconstant}
\end{align}
We reserve $C_p$ for the scalar Hilbert constant and use $\Lambda_p(P)$
for the joint constant. All joint output spaces below use the $\ell^p$
direct-sum norm. Write
\[
 \mathbf H f=(w_j^{1/p}(a_jf+b_jHf))_{j=1}^m,\qquad
 \mathbf T_{\mathcal E}X
       =(w_j^{1/p}(a_jX+b_jH_{\mathcal E}X))_{j=1}^m.
\]

\begin{theorem}\label{thm:unified}\label{cor:jointschur}
The constant \eqref{eq:jointconstant} is finite and equals the least
$C\ge0$ for which $P-C^p|x|^p$ has a majorant in $\mathcal U_p$.
Such a majorant exists at this optimal value of $C$.
It also satisfies
\begin{equation}\label{eq:jointschur}
 \Lambda_p(P)^p
 =\sup_{\substack{q\ne0\ {\rm polynomial}\\q(0)=0}}
 \frac{\int_0^{2\pi}P(\Rea q(e^{it}),\Ima q(e^{it}))\dd t}
      {\int_0^{2\pi}|\Rea q(e^{it})|^p\dd t}.
\end{equation}
For every nonzero Hilbert space $\mathcal K$,
\begin{equation}\label{eq:jointmodels}
 \Lambda_p(P)
 =\|\mathbf H\|_{L^p(\R)\to\ell_p^m(L^p(\R))}
 =\|\mathbf H\|_{L^p(\R;\Sp_p(\mathcal K))\to
                     \ell_p^m(L^p(\R;\Sp_p(\mathcal K)))}
 =\sup_{\mathcal E}\|\mathbf T_{\mathcal E}\|.
\end{equation}
If $a_j=\cos(\pi\alpha_j)$ and $b_j=\sin(\pi\alpha_j)$,
$0<\alpha_j<1$, the simultaneous transform
$\mathbf R V=(w_j^{1/p}R_{\alpha_j}V)_{j=1}^m$ satisfies
\begin{equation}\label{eq:jointdiscrete}
 \|\mathbf R\|_{\ell^p(\Z)\to\ell_p^m(\ell^p(\Z))}
 =\|\mathbf R\|_{\ell^p(\Z;\Sp_p(\mathcal K))\to
                    \ell_p^m(\ell^p(\Z;\Sp_p(\mathcal K)))}
 =\Lambda_p(P).
\end{equation}
The scalar Hilbert and discrete norms and the uniform Schur norm
are unchanged on replacing them by completely bounded norms for
the canonical noncommutative $L^p$ direct sums.
At $p=2$, $\Lambda_2(P)^2=\sum_jw_j(a_j^2+b_j^2)$.
\end{theorem}

The simultaneous realization uses one finite nilpotent matrix for
all shifts. At $p=4$, the joint norm and the criterion for a common
extremizing sequence are explicit (Corollary~\ref{cor:jointfour});
the same moment disk classifies all signed fourth-power inequalities
(Corollary~\ref{cor:quarticcone}). For a bounded real sequence
$\boldsymbol\kappa$, let $D_{\boldsymbol\kappa}X$ have diagonal
entries $\kappa_jX_{jj}$ and zero off-diagonal entries. Then
Corollary~\ref{cor:diagonalload} gives the exact perturbation formula
\[
 \|\HSch+D_{\boldsymbol\kappa}\|
 =\|\HSch+D_{\boldsymbol\kappa}\|_{\rm cb}
 =\max\{C_p,\|\boldsymbol\kappa\|_\infty\}.
\]

\begin{corollary}
\label{thm:realpencil}\label{cor:HilbertCB}\label{cor:RT}\label{cor:conjectures}
Let $(a,b)$ be a real pair and let $\mathcal K\ne\{0\}$.

\smallskip\noindent\textup{(i) Scalar norm and optimal majorant constant.}
For a single pencil, the common norm and the least admissible
constant in Theorem~\ref{thm:unified} are given by
\begin{equation}\label{eq:Koblique}
 K_p(a,b)^p=
 \max_{0\le t\le\pi}
 \frac{|a\cos t+b\sin t|^p+
 |a\cos(t+\pi/p)+b\sin(t+\pi/p)|^p}
 {|\cos t|^p+|\cos(t+\pi/p)|^p}.
\end{equation}
Its symmetries and coordinate values are
\begin{gather}
 K_p(a,-b)=K_p(a,b),\qquad K_p(ca,cb)=|c|K_p(a,b),\qquad
 K_{p'}(a,b)=K_p(a,b),\label{eq:Kproperties}\\
 K_p(0,1)=C_p=\cot\frac{\pi}{2p^*},\qquad
 K_2(a,b)=\sqrt{a^2+b^2},\qquad K_p(a,0)=|a|.\label{eq:constant}
\end{gather}

\smallskip\noindent\textup{(ii) Schur and completely bounded norms.}
For the standard infinite Hilbert Schur multiplier,
\begin{equation}\label{eq:SchurCBpencil}
 \|aI+b\HSch\|_{\Sp_p\to\Sp_p}
 =\|aI+b\HSch\|_{\mathrm{cb},\Sp_p}
 =\|aI+bH\|_{\mathrm{cb},L^p(\R)}=K_p(a,b).
\end{equation}
In particular, the Hilbert-transform norms on $L^p(\R)$ and
$L^p(\R;\Sp_p)$, and the completely bounded norm on $L^p(\R)$,
all equal $C_p$. The infinite signed Schur
conventions with diagonal value $+1$ or $-1$ have the same operator norm
and completely bounded norm, by Lemma~\ref{lem:diagonalnormalization}.
For the Hilbert transform, the identification of the Schur multiplier
norm with the completely bounded norm is due to \cite[Theorem~6.2]{NR};
the value for every $1<p<\infty$ and the extension to finite
families of real pencils follow here.

\smallskip\noindent\textup{(iii) Shifted discrete transforms.}
For $0<\alpha<1$,
\begin{equation}\label{eq:RalphaNorm}
 \|R_\alpha\|_{\ell^p(\Z;\Sp_p(\mathcal K))}
 =\|R_\alpha\|_{\ell^p(\Z)}
 =K_p(\cos(\pi\alpha),\sin(\pi\alpha)).
\end{equation}
For $t=k+\alpha$ with $k\in\Z$ and $0<\alpha<1$, define
$(R_tV)_\ell=(-1)^k(R_\alpha V)_{\ell+k}$; at integers set
$(R_kV)_\ell=(-1)^kV_{\ell+k}$. Thus the same norm formula holds
with $(\cos(\pi t),\sin(\pi t))$ for every real $t$.
The completely bounded value is the same
by Theorem~\ref{thm:unified}. Formula~\eqref{eq:RalphaNorm} resolves
Laeng's Conjecture~5.7 and, at $\alpha=1/2$, the Riesz--Titchmarsh
conjecture.

\smallskip
The uniform Schur bounds are sharp as suprema over finite dimensions.
\end{corollary}

\subsection{Sharp Volterra bounds}
\begin{corollary}\label{cor:MM}
Let $T=A+iB$ be compact and quasinilpotent.  If $A\in\Sp_p$, then
$B\in\Sp_p$ and, for every real $a,b$,
\begin{equation}\label{eq:obliqueMM}
 \|aA+bB\|_p\le K_p(a,b)\|A\|_p.
\end{equation}
In particular, if $B\in\Sp_p$, then $A\in\Sp_p$ and
\begin{equation}\label{eq:main}
 \|A\|_p\le C_p\|B\|_p.
\end{equation}
The constants are optimal already over finite nilpotent matrices.  In
particular \eqref{eq:main} resolves, for every $1<p<\infty$, the
conjectured sharp constant recorded by Matsaev--Sodin
\cite{MS}; see also \cite{GK68}.
\end{corollary}

The spectral correction determines the optimal comparison constant
for every nonempty prescribed spectral set.

\begin{corollary}\label{cor:spectralsector}
Let $1<p<\infty$ and $\varnothing\ne\Omega\subset\C$, and put
\[
 r(z)=\frac{|\Rea z|}{|\Ima z|}\quad(\Ima z\ne0),\qquad
 r(0)=0,\qquad r(z)=+\infty\quad(z\in\R\setminus\{0\}).
\]
The optimal constant in
$\|\Rea T\|_p\le L\|\Ima T\|_p$ over finite matrices with
$\spec T\subset\Omega$ is
\begin{equation}\label{eq:spectralsetconstant}
 L_\Omega=\max\{C_p,\sup_{z\in\Omega}r(z)\},
\end{equation}
where $L_\Omega=+\infty$ means that no finite bound exists.
If $L_\Omega<\infty$, every compact $T=A+iB\in\Sp_p$ with
$\spec T\subset\Omega$ satisfies
\begin{equation}\label{eq:sectorMatsaev}
 \|A\|_p\le L_\Omega\|B\|_p.
\end{equation}
In particular, for $K\ge C_p$, the bound with constant $K$ holds
for all finite matrices with spectrum in $\Omega$ exactly when
$\Omega\subset\Sigma_{p,K}$, where
\begin{equation}\label{eq:spectralsector}
 \Sigma_{p,K}=\{z:|\Rea z|\le K|\Ima z|\},
 \qquad\Sigma_p=\Sigma_{p,C_p}.
\end{equation}
For $\Omega=\Sigma_p$, sharpness holds already over finite nilpotents.
\end{corollary}

The compact assertion assumes $T\in\Sp_p$ and allows nonnormal operators.
The one-component membership assumption of Corollary~\ref{cor:MM}
is used only in the quasinilpotent setting.

\begin{remark}\label{rem:mergedscope}
The real ridge classification concerns finite matrices and its
Schatten-class consequences. Its extension to arbitrary finite
von Neumann algebras would require a compatible positive ridge
representation and the corresponding integrability and comparison
identities. The Riesz argument establishes its stated bounds in
arbitrary finite tracial algebras by a different spectral comparison;
it does not supply such a representation. In particular, traces of
real directional powers cannot in general be combined as scalar
radial functions to recover $\tau|A+iB|^p$.
The real-pencil norm identities and the Riesz projection identities
therefore remain complementary results. They do not assert a sharp
lifting theorem for arbitrary complex pencils $aI+bH$.
For the subdiagonal projections, sharpness means optimality uniformly
over algebras; for triangular truncation it means a supremum over
matrix sizes. The Fourier Riesz norm equals $D_p$ for every nonzero
finite coefficient algebra.
\end{remark}

The proofs follow two complementary routes. For finite matrices,
Hein\"avaara's tracial joint spectral measure and a nonnegative
logarithmic determinant defect give the exact spectral transfer
identity in Section~\ref{sec:measure}. Analytic-disc envelopes and
Toeplitz tests then yield the classification and least spectral
correction in Section~\ref{sec:toeplitz}. The transfer identity and
classification, together with the finite-dimensional realizations
of Section~\ref{sec:realizations}, give the Hilbert, Volterra, and
shifted-transform bounds. Section~\ref{sec:angular} establishes the
angular inverse realization.

For general finite tracial algebras, the Riesz argument uses the
separate comparisons in Section~\ref{rz:sec:spectral}.
Determinant covariance gives the analytic-distance estimate for
$2<p<4$, while positive-kernel comparison and independently proved
analytic regularization supply the witnesses for $p>4$.
Section~\ref{rz:sec:witnesses} completes the projection theorem by
continuity, duality, and a subdiagonal corner construction.
Section~\ref{sec:joint} collects the compact spectral, moment, and
projection consequences; Section~\ref{sec:boundary} explains the
limits of the two-point criterion. The appendices contain the
coordinate majorant, quantitative comparisons, Hardy-space
structural inputs, analytic regularization, and rotation and
independent-coordinate consequences.

\section{A spectral transfer theorem and its homogeneous consequences}\label{sec:measure}
We use the following form of Hein\"avaara's representation
\cite{Hein}. For Hermitian $A,B\in M_d(\C)$ there is
a positive Borel measure $\mu_{A,B}$, normalized by
$\mu_{A,B}(\{0\})=0$, such that the following identity holds
simultaneously for every $s,t\in\R$ and every measurable real $f$
with $f(0)=0$ and $\int_{-M}^M|f(u)/u|\dd u<\infty$ for each $M>0$:
\begin{equation}
 \int_{\R^2} f(sx+ty)\dd\mu_{A,B}(x,y)
 =\tr\Hs(f)(sA+tB),\qquad
 \Hs(f)(u)=\int_0^1\frac{1-r}{r}f(ru)\dd r.
 \label{eq:measure}
\end{equation}
The measure is unique away from the origin
\cite{Hein}. Complex tests are handled by real and
imaginary parts. The normalization at zero fixes the otherwise
irrelevant mass there.
In particular, for every $q>0$,
\begin{equation}
 q(q+1)\int|sx+ty|^q\dd\mu_{A,B}
 =\tr|sA+tB|^q.
 \label{eq:moments}
\end{equation}
For integer $k\ge1$ the same tracial formula gives
\begin{equation}\label{eq:polymoment}
 \int(sx+ty)^k\dd\mu_{A,B}
 =\frac{\tr(sA+tB)^k}{k(k+1)}.
\end{equation}
Nonnegative one-variable tests show
\begin{equation}
 \supp\mu_{A,B}\subset[-\norm A,\norm A]\times[-\norm B,\norm B].
 \label{eq:musupport}
\end{equation}
For example, choose a nonnegative $f$ vanishing on $[-\norm A,\norm A]$ and strictly positive off this interval, then use $(s,t)=(1,0)$ in \eqref{eq:measure}. The other coordinate is identical.
For every $q>0$, the elementary bound
$|(x,y)|^q\le C_q^{\mathrm{coord}}(|x|^q+|y|^q)$, where
$C_q^{\mathrm{coord}}=\max\{1,2^{q/2-1}\}$, and \eqref{eq:moments} give
\begin{equation}
 \int |(x,y)|^q\dd\mu_{A,B}<\infty.
 \label{eq:radialmoments}
\end{equation}
Thus the tracial measure has compact support and every positive radial
moment is finite, although its total mass can be infinite.

There is also an exact support-radius identity. Put $T=A+iB$,
$\ell_\theta(z)=\Rea(e^{-i\theta}z)$ and
$R_\theta=\sup_{z\in\supp\mu_{A,B}\cup\{0\}}|\ell_\theta(z)|$.
For $0<a<R_\theta$, continuity and the definition of support give
$\mu_{A,B}(E_a)>0$ for $E_a=\{|\ell_\theta|>a\}$, while
$a\mu_{A,B}(E_a)\le\int|\ell_\theta|\dd\mu_{A,B}<\infty$.
For $q\ge1$,
\[
 a^q\mu_{A,B}(E_a)\le\int|\ell_\theta|^q\dd\mu_{A,B}
 \le R_\theta^{q-1}\int|\ell_\theta|\dd\mu_{A,B}.
\]
Taking $q$th roots, then $q\to\infty$ and $a\uparrow R_\theta$,
gives convergence to $R_\theta$ despite the possibly infinite total
mass. The case $R_\theta=0$ is immediate. Equation~\eqref{eq:moments}
therefore gives $R_\theta=\|\Rea(e^{-i\theta}T)\|$, and hence
\begin{equation}\label{eq:tracialradius}
 \sup_{z\in\supp\mu_{A,B}\cup\{0\}}|z|
 =\sup_\theta\|\Rea(e^{-i\theta}T)\|
 =\max_{\|v\|=1}|v^*Tv|.
\end{equation}
The union with $\{0\}$ includes the zero matrix.

The full ridge identity \eqref{eq:measure}, rather than only its absolute
power moments, is used in the Green decomposition below; see also
Example~\ref{ex:scalarmeasure}.

\subsection{A nonnegative logarithmic determinant defect}\label{sec:potential}
Logarithmic determinants and subharmonicity have long been used to
compare real and imaginary parts; see \cite{MS} and the Hilbert-transform
formulation in \cite[Section~2]{MOS}. Here the determinant Jensen
formula gives the explicit nonnegative potential used in the exact
spectral-defect identity.
Fix a finite matrix $T=A+iB$. For $w\notin\spec T$ set $E=T-wI$ and define
\begin{align}
 L_T(w)&=d\log2+\frac1{2\pi}\int_0^{2\pi}
  \log\left|\det\Rea(e^{-i\theta}E)\right|\dd\theta,\label{eq:L}\\
 \Psi_T(w)&=L_T(w)-\log|\det(T-wI)|.\label{eq:Psi}
\end{align}
Values at the finitely many eigenvalues will be irrelevant to area integration.

\begin{lemma}\label{lem:jensen}
For $w\notin\spec T$,
\begin{equation}
 \Psi_T(w)=\sum_{j=1}^d\log^+\left|\lambda_j(E^{-1}E^*)\right|\ge0.
 \label{eq:psipositive}
\end{equation}
The function $\Psi_T$ is continuous off $\spec T$ and belongs to
$L^q(\C)$ for every finite $q\ge1$. Its essential support is contained
in the numerical range:
\begin{equation}\label{eq:numericalrangesupport}
 \supp\Psi_T\subset\overline{W(T)},\qquad
 W(T)=\{v^*Tv:\ v\in\C^d,\ \|v\|=1\}.
\end{equation}
In particular it vanishes for $|w|>\|T\|$. On each bounded set it is bounded above by a constant plus
$\sum_j\log^-|w-\lambda_j(T)|$. Moreover,
\begin{equation}
 \Psi_{rT}(w)=\Psi_T(w/r),\qquad r>0.
 \label{eq:psiscaling}
\end{equation}
\end{lemma}
\begin{proof}
The determinant factorization is
\[
 \det\Rea(e^{-i\theta}E)
 =2^{-d}e^{-id\theta}\det E\,
   \det(I+e^{2i\theta}E^{-1}E^*).
\]
For every $\zeta\in\C$,
\[
 \frac1{2\pi}\int_0^{2\pi}\log|1+\zeta e^{i\theta}|\dd\theta
 =\log^+|\zeta|.
\]
Factoring the characteristic polynomial, rather than diagonalizing the matrix, proves \eqref{eq:psipositive}. Roots on the unit circle cause only integrable logarithmic zeros. Since $E$ is invertible, the determinant polynomial in $e^{2i\theta}$ is not identically zero.

When $w\notin\overline{W(T)}$, $E$ is invertible because
$\spec T\subset W(T)$. Let $E^{-1}E^*v=\lambda v$, with $v\ne0$.
Then $E^*v=\lambda Ev$. Writing $q=v^*Ev$, we have
\[
 |q|\ge\operatorname{dist}(w,W(T))\|v\|^2>0,\qquad
 \overline q=\lambda q.
\]
Thus every eigenvalue has $|\lambda|=1$, and \eqref{eq:psipositive}
is zero. Applying this to each eigenvalue proves the support assertion.
Eigenvalue continuity as an unordered multiset
also proves continuity of $\Psi_T$ off $\spec T$.

On a fixed bounded set, the determinant norm bound in \eqref{eq:L} gives, with a sufficiently large constant $M$,
\[
 0\le\Psi_T(w)\le M-\sum_j\log|w-\lambda_j(T)|
 \le M' +\sum_j\log^-|w-\lambda_j(T)|.
\]
Consequently $\Psi_T\in L^q(\C)$ for every finite $q\ge1$. Scaling in the definition proves \eqref{eq:psiscaling}. The zero matrix is included, with $\Psi_0=0$ almost everywhere.
\end{proof}

\subsection{A compensated ridge decomposition}
Write $v_\theta=(\cos\theta,\sin\theta)$ and
$\av=(2\pi)^{-1}\int_0^{2\pi}\dd\theta$.

\begin{lemma}
\label{lem:ridgegreen}\label{lem:fubini}
Let $f\in C_c^\infty(\R^2)$ be real and satisfy $f(0)=0$. Define
\begin{equation}\label{eq:gtheta}
 g_\theta(t)=\frac1{2\pi}\int_{\R^2}\Delta f(w)
 \bigl(\log|v_\theta\cdot w-t|-\log|v_\theta\cdot w|\bigr)\dd m_2(w).
\end{equation}
Then $g_\theta$ is smooth, $g_\theta(0)=0$, and its derivatives on
bounded intervals are uniformly bounded in $\theta$. Moreover,
\begin{equation}\label{eq:ridgegreen}
 f(z)=\av g_\theta(v_\theta\cdot z).
\end{equation}
For compact $K,Z\subset\R^2$ and $R<\infty$,
\begin{equation}\label{eq:absfubini}
 \sup_{\substack{\theta\in[0,2\pi],\,z\in Z\\|a|\le R}}
 \int_K\left(
 |\log|v_\theta\cdot(w-z)-a||+|\log|v_\theta\cdot w||
 \right)\dd m_2(w)<\infty.
\end{equation}
Thus Fubini's theorem applies to integration in the angular and
planar variables for finite sums with bounded measurable shifts $a=a_j(\theta)$. For each $R_1<\infty$,
\begin{equation}\label{eq:Qglinear}
 |g_\theta(t)|+|t^2g_\theta''(t)+2tg_\theta'(t)|
 \le C_{R_1}|t|,\qquad |t|\le R_1.
\end{equation}
\end{lemma}
\begin{proof}
Choose $R_0$ so every $K-z$, $z\in Z\cup\{0\}$, lies in the disk
of radius $R_0$. Coordinates along $v_\theta$ bound the first
logarithmic integral in \eqref{eq:absfubini} by
\[
 2R_0\int_{-R_0}^{R_0}|\log|s-a||\dd s
 \le2R_0\int_{-R_0-R}^{R_0+R}|\log|t||\dd t<\infty.
\]
The second term is the case $a=0,z=0$; Tonelli gives the stated
Fubini consequence.

The directional integrals
$a_\theta(s)=\int_\R\Delta f(sv_\theta+uv_\theta^\perp)\dd u$
have common compact support and uniform bounds on all derivatives.
After translation, the first term of \eqref{eq:gtheta} is
$(2\pi)^{-1}\int\log|u|a_\theta(u+t)\dd u$.
Differentiation falls on $a_\theta$, proving the uniform smoothness.
The subtraction gives $g_\theta(0)=0$ and therefore
\eqref{eq:Qglinear}.

Finally,
\[
 \av\bigl(\log|v_\theta\cdot(w-z)|-\log|v_\theta\cdot w|\bigr)
 =\log|w-z|-\log|w|
\]
locally integrably in $w$, since the two constants $-\log2$ cancel.
The established Fubini bound and
$\Delta\log|w-z|=2\pi\delta_z$ give
\[
 \av g_\theta(v_\theta\cdot z)
 =\frac1{2\pi}\int\Delta f(w)
       (\log|w-z|-\log|w|)\dd m_2(w)=f(z)-f(0).
\]
Thus the logarithms are paired with smooth area tests before the
tracial ridge identity is used.
\end{proof}

\subsection{The smooth comparison identity}
Put $\Euler=z\cdot\nabla$ and $\EulerQ=\Euler(\Euler+1)$.

\begin{lemma}\label{lem:smoothcomparison}
For a real-valued smooth function $f$ on a neighbourhood of
$\{0\}\cup\supp\mu_{A,B}\cup\supp\Psi_T\cup\spec T$,
\begin{equation}\label{eq:smoothcomparison}
 d f(0)+\int\EulerQ f\dd\mu_{A,B}-\sum_j f(\lambda_j(T))
 =\frac1{2\pi}\int_\C\Psi_T\Delta f\dd m_2.
\end{equation}
\end{lemma}
\begin{proof}
Choose a compactly supported cutoff equal to one on a neighbourhood of
$\supp\mu_{A,B}$, $\supp\Psi_T$, and the spectrum of $T$, and subtract
$f(0)$.  It is therefore enough to prove the identity for
$f\in C_c^\infty(\R^2)$ with $f(0)=0$.

The one-variable inverse of $\Hs$ is
$Q=t\partial_t(t\partial_t+1)$
\cite[Remark~2, p.~40]{HeinThesis}. If $G(r)=g(ru)$,
integration by parts gives
\[
 \Hs(Qg)(u)=\int_0^1(1-r)\{rG''(r)+2G'(r)\}\dd r=g(u)-g(0).
\]
The Euler operator commutes with ridge composition.  Consequently
Lemma~\ref{lem:ridgegreen} and the full tracial ridge identity
\eqref{eq:measure} imply
\begin{equation}\label{eq:euler-ridge-step}
 \int\EulerQ f\dd\mu_{A,B}
 =\av\tr g_\theta(A\cos\theta+B\sin\theta).
\end{equation}
This use of Fubini is absolute: on the compact support of
$\mu_{A,B}$, \eqref{eq:Qglinear} gives
$|Q g_\theta(v_\theta\cdot z)|\le C|z|$, and
\eqref{eq:radialmoments} with exponent one gives
$\int|z|\dd\mu_{A,B}<\infty$.

Let $X_\theta=A\cos\theta+B\sin\theta$.  Inserting
\eqref{eq:gtheta} and using the spectral theorem,
\begin{align*}
 \tr g_\theta(X_\theta)
 &=\frac1{2\pi}\int\Delta f(w)
 \left\{\log\left|\det(v_\theta\cdot w\,I-X_\theta)\right|
       -d\log|v_\theta\cdot w|\right\}\dd m_2(w).
\end{align*}
Here
$|\det(v_\theta\cdot w\,I-X_\theta)|
 =|\det\Rea(e^{-i\theta}(T-wI))|$.
Every eigenvalue of $X_\theta$ lies in
$[-(\|A\|+\|B\|),\|A\|+\|B\|]$.  Taking the eigenvalues in increasing
order makes $a_j(\theta)=\lambda_j(X_\theta)$ continuous (hence measurable)
in $\theta$.  Lemma~\ref{lem:fubini} therefore applies to these bounded shifts and justifies
the interchange of angular and planar integration by absolute
integrability.  Averaging in $\theta$ therefore gives
\begin{equation}\label{eq:fixed-r}
 \av\tr g_\theta(X_\theta)
 =\frac1{2\pi}\int\Delta f(w)(L_T(w)-d\log|w|)\dd m_2(w).
\end{equation}
By definition $L_T=\Psi_T+\log|\det(T-wI)|$.  Green's formula and
$\log|\det(T-wI)|=\sum_j\log|w-\lambda_j(T)|$ yield
\[
 \frac1{2\pi}\int\Delta f(w)\log|\det(T-wI)|\dd m_2(w)
 =\sum_j f(\lambda_j(T)),
\]
while the same formula with $\log|w|$ gives $f(0)=0$.  Finally,
$\Psi_T\in L^{q}$ for every finite $q$ and $\Delta f$ is bounded and
compactly supported, so its term is absolutely integrable.  Combining
these identities with \eqref{eq:euler-ridge-step} proves
\eqref{eq:smoothcomparison}.  Restoring the subtracted constant gives
the term $df(0)$.
\end{proof}

\subsection{Nonhomogeneous transfer and homogeneous consequences}
For planar $f$ with $f(0)=0$, put
$\Hs f(z)=\int_0^1(1-r)f(rz)\dd r/r$ and define
\begin{equation}\label{eq:Theta}
 \Theta_T(w)=\int_0^1\frac{1-r}{r}\Psi_T(w/r)\dd r.
\end{equation}
This function is nonnegative and supported in $|w|\le\|T\|$.
Minkowski's inequality and planar scaling give
\begin{equation}\label{eq:Theta-bounds}
 \|\Theta_T\|_s\le\|\Psi_T\|_s
       \int_0^1(1-r)r^{2/s-1}\dd r
 =\frac{s^2}{2(s+2)}\|\Psi_T\|_s,\qquad 1\le s<\infty.
\end{equation}

\begin{theorem}\label{thm:nonhomogeneous}
Let $T=A+iB$ be a finite matrix, and let $D$ be an open disk centered
at zero containing
$\supp\mu_{A,B}\cup\supp\Psi_T\cup\spec T$.
Let $f\in W^{2,s}_{\mathrm{loc}}(D)$ be real-valued, with
$1<s<\infty$, and use its continuous representative. If $f(0)=0$, then
\begin{equation}\label{eq:generalexact}
 \sum_j(\Hs f)(\lambda_j(T))-\int f\dd\mu_{A,B}
 =\frac1{2\pi}\int\Theta_T(-\Delta f)\dd m_2.
\end{equation}
All terms are finite; the right side is nonnegative if $\Delta f\le0$.
\end{theorem}
\begin{proof}
For smooth $f$, radial integration gives
$\EulerQ\Hs f=f$ and
$\Delta\Hs f(z)=\int_0^1r(1-r)\Delta f(rz)\dd r$.
Apply Lemma~\ref{lem:smoothcomparison} to $\Hs f$ and change variables
$w=rz$ in the potential term. On a closed disk containing the
supports and their radial contractions, absolute Fubini follows from
$\|\Psi_T\|_1\|\Delta f\|_\infty\int_0^1r(1-r)\dd r<\infty$.
This proves \eqref{eq:generalexact} for smooth local tests.

Choose nested closed disks in $D$ whose interiors contain all the
supports. After multiplying $f$ by an outer cutoff equal to one on
the larger disk, mollify and subtract the value at zero:
$f_\varepsilon=f*\eta_\varepsilon-(f*\eta_\varepsilon)(0)$.
Choose $0<\alpha<\min\{1,2-2/s\}$. Sobolev--Morrey embedding gives
$f_\varepsilon\to f$ in $C^{0,\alpha}$ and $W^{2,s}$ on the inner
disk. Put
$a_\varepsilon=\|f_\varepsilon-f\|_{C^{0,\alpha}}\to0$
there. Since $f_\varepsilon(0)=f(0)=0$, the measure and spectral
errors are bounded by
\[
 a_\varepsilon\int|z|^\alpha\dd\mu_{A,B},
 \qquad
 \frac{a_\varepsilon}{\alpha(\alpha+1)}
        \sum_j|\lambda_j(T)|^\alpha.
\]
They tend to zero by \eqref{eq:radialmoments}. The potential error is
at most
$\|\Theta_T\|_{s'}\|\Delta(f_\varepsilon-f)\|_s/(2\pi)$,
where $s'=s/(s-1)$, and tends to zero by
\eqref{eq:Theta-bounds}. The same estimates establish finiteness of
all three terms for $f$, and passage to the limit proves the theorem.
\end{proof}

\begin{corollary}
\label{cor:nonhomogeneouscertificate}
Let $U$ satisfy the hypotheses of Theorem~\ref{thm:nonhomogeneous}.
For finite families of real coefficients $c_j,u_j,v_j$, let
$\phi_j:\R\to\R$ be measurable, with $\phi_j(0)=0$ and
$\int_{-M}^M|\phi_j(t)/t|\dd t<\infty$ for every $M>0$, and set
$F(x,y)=\sum_jc_j\phi_j(u_jx+v_jy)$. Then
\begin{align}\label{eq:nonhomcertificate}
 \sum_k(\Hs U)(\lambda_k(T))
       -\sum_jc_j\tr(\Hs\phi_j)(u_jA+v_jB)
 &=\int(U-F)\dd\mu_{A,B}\notag\\
 &\quad+\frac1{2\pi}\int\Theta_T(-\Delta U)\dd m_2.
\end{align}
All terms are finite. Both remainders are nonnegative if
$U\ge F$ on $D$ and $\Delta U\le0$.
\end{corollary}
\begin{proof}
The full ridge identity \eqref{eq:measure}, applied to $|\phi_j|$,
proves absolute integrability of each ridge. Its signed sum for
$\phi_j$ gives
$\int F\dd\mu_{A,B}=\sum_jc_j\tr(\Hs\phi_j)(u_jA+v_jB)$.
Combining this with \eqref{eq:generalexact} proves the identity,
and positivity gives the final assertion.
\end{proof}

\begin{corollary}\label{thm:certificate}
Let $U$ be real-valued, $p$-homogeneous, and belong to
$C^1(\C)\cap W^{2,q}_{\mathrm{loc}}(\C)$ for some $q>1$.
For finite families of real coefficients $c_j,u_j,v_j$, set
\[
 F(x,y)=\sum_jc_j|u_jx+v_jy|^p,\qquad
 \mathcal F_p(A,B)=\sum_jc_j\tr|u_jA+v_jB|^p.
\]
Then
\begin{align}\label{eq:certificate}
 \sum_kU(\lambda_k(T))-\mathcal F_p(A,B)
 &=p(p+1)\int(U-F)\dd\mu_{A,B}\notag\\
 &\quad+\frac1{2\pi}\int_\C\Psi_T(-\Delta U)\dd m_2.
\end{align}
Both remainder terms are nonnegative if $U\ge F$ and $\Delta U\le0$.
We call such a function $U$ an admissible superharmonic majorant of $F$.
\end{corollary}
\begin{proof}
Homogeneity gives $\Hs U=U/[p(p+1)]$ and
$\Delta U(rz)=r^{p-2}\Delta U(z)$ almost everywhere.
H\"older's inequality and the compact support of $\Psi_T$ give
$J=\int\Psi_T|\Delta U|\dd m_2<\infty$.
Tonelli's theorem and the substitution $w=rz$ therefore yield
\[
 \int\Theta_T|\Delta U|\dd m_2
 =J\int_0^1(1-r)r^{p-1}\dd r
 =\frac{J}{p(p+1)}.
\]
The same change of variables is consequently valid for the signed
Laplacian integral. Apply Corollary~\ref{cor:nonhomogeneouscertificate}
with $\phi_j(t)=|t|^p$, for which
$\Hs\phi_j=\phi_j/[p(p+1)]$, and multiply by $p(p+1)$.
This proves \eqref{eq:certificate}, including its positivity assertion.
\end{proof}
For a homogeneous $V$ in the same regularity class and a nilpotent
$N=A+iB$, Corollary~\ref{thm:certificate} with $F=0$ specializes to
\begin{equation}\label{eq:homidentity}
 p(p+1)\int_\C V\dd\mu_{A,B}
 =\frac1{2\pi}\int_\C\Psi_N\Delta V\dd m_2.
\end{equation}
Corollary~\ref{thm:certificate} supplies the forward implication for
every finite signed real ridge sum.  Sections~\ref{sec:toeplitz}--\ref{sec:regularity}
prove the converse, establishing the necessity and sufficiency of
the superharmonic-majorant condition for the full class in Theorem~\ref{thm:classification}.

\subsection{The sharp coordinate specialization}\label{subsec:coordinatepath}
The sharp coordinate estimate uses the classical explicit superharmonic
majorant. The obstacle classification treats general signed ridge sums
and identifies their optimal spectral correction.

\begin{corollary}\label{prop:coordinatepath}
For every $1<p<\infty$,
\begin{equation}\label{eq:coordinatepath}
 \sup_{N\ne0\ {\rm finite\ nilpotent}}
       \frac{\|\Ima N\|_p}{\|\Rea N\|_p}
 =\sup_{\mathcal E}\|H_{\mathcal E}\|_{\Sp_p\to\Sp_p}=C_p.
\end{equation}
Both estimates hold uniformly at every matrix amplification.
\end{corollary}
\begin{proof}
Let $U_p$ be the sectorwise superharmonic majorant in
Appendix~\ref{sec:coordinate} and put $V(x,y)=U_p(y,x)$.
Its elementary verification gives
\[
 V(x,y)\ge |y|^p-C_p^p|x|^p,\qquad
 \Delta V\le0,\qquad V(x,0)\le0.
\]
For nilpotent $N=A+iB$, the spectral term in
Corollary~\ref{thm:certificate} is zero, yielding
$\|B\|_p\le C_p\|A\|_p$.
For Hermitian $X$, the matrix $T=X+iH_{\mathcal E}X$ is block upper
triangular with Hermitian diagonal blocks, so its spectrum is real.
The same corollary gives
\[
 \|H_{\mathcal E}X\|_p^p-C_p^p\|X\|_p^p
 \le\sum_{\lambda\in\spec T}V(\lambda)\le0.
\]
For arbitrary $X$, apply this to
$\left(\begin{smallmatrix}0&X\\X^*&0\end{smallmatrix}\right)$
and the amplified flag; both $p$th powers acquire the factor two.
Arbitrary block sizes also give the bound at every amplification.

For sharpness, analytic polynomials $q$ with $q(0)=0$ approach the
classical Pichorides constant \cite{Pichorides,HKV}.
Lemma~\ref{lem:commontoeplitz} transfers their ratios to $q(S_n)$.
These strictly upper-triangular matrices satisfy
$\Ima q(S_n)=H_{\mathcal E}\Rea q(S_n)$ for the standard flag,
giving both lower bounds in \eqref{eq:coordinatepath}.
\end{proof}

\section{Proof of the classification and optimal correction}\label{sec:toeplitz}
\subsection{Finite Toeplitz sections}
\label{subsec:commontoeplitz}
The lower bounds in the two parts use the same finite-section limit.
Its singular-value form handles arbitrary complex symbols; its
self-adjoint form supplies the continuous spectral tests needed for the
nilpotent classification.

\begin{lemma}\label{lem:commontoeplitz}
Let $Q(z)=\sum_{k=-r}^rQ_kz^k$ be an $M_d(\mathbb C)$-valued
trigonometric polynomial, with $d$ and $r$ fixed, and put
\[
 A_n(Q)=(Q_{j-i})_{1\le i,j\le n},\qquad n>2r.
\]
For every $1\le p<\infty$,
\begin{equation}\label{eq:commontoeplitzsingular}
 \lim_{n\to\infty}\frac1n
       \operatorname{Tr}_{nd}|A_n(Q)|^p
 =\int_0^{2\pi}\operatorname{Tr}_d|Q(e^{it})|^p
       \frac{dt}{2\pi}.
\end{equation}
If $Q(e^{it})$ is Hermitian for every $t$ and $g$ is continuous on a
common interval containing the spectra of $A_n(Q)$ and $Q(e^{it})$, then
\begin{equation}\label{eq:commontoeplitzspectral}
 \lim_{n\to\infty}\frac1n\operatorname{Tr}_{nd}g(A_n(Q))
 =\int_0^{2\pi}\operatorname{Tr}_dg(Q(e^{it}))
       \frac{dt}{2\pi}.
\end{equation}
\end{lemma}
\begin{proof}
Let $S_n$ be the cyclic shift with a one at $(i,j)$ when
$j-i\equiv1\pmod n$, and set
\[
 C_n(Q)=\sum_{k=-r}^rS_n^k\otimes Q_k,\qquad
 R=\sum_{k=-r}^r\|Q_k\|.
\]
The difference $E_n=A_n(Q)-C_n(Q)$ is supported in the first and last
$r$ block rows.  Hence
\[
 \operatorname{rank}E_n\le2rd,\qquad
 \|E_n\|\le2R,\qquad
 \|A_n(Q)\|,\ \|C_n(Q)\|\le R.
\]
The last bound for the Toeplitz matrix follows by compressing the
Laurent operator, or directly by the triangle inequality for its
partial-shift terms.  Therefore
\[
 \left(\frac1n\operatorname{Tr}_{nd}|E_n|^p\right)^{1/p}
 \le2R(2rd/n)^{1/p}\longrightarrow0.
\]
The discrete Fourier transform on the first tensor factor makes
$C_n(Q)$ block diagonal, with blocks $Q(\zeta_n^\ell)$,
$0\le\ell<n$, where $\zeta_n=e^{2\pi i/n}$.  The triangle inequality
in the Schatten $p$ norm scaled by $n^{-1/p}$ and the Riemann-sum limit prove
\eqref{eq:commontoeplitzsingular}.  This uses singular values, without
assuming that $Q$ or its Toeplitz sections are normal.

For the second assertion both matrices are Hermitian.  For an integer
$k\ge1$, telescoping the difference of their $k$th powers and using
$\|E_n\|_1\le2R\operatorname{rank}E_n$ gives
\[
 \frac1n\left|\operatorname{Tr}_{nd}
       (A_n(Q)^k-C_n(Q)^k)\right|
 \le\frac{2kR^k\operatorname{rank}E_n}{n}
 \longrightarrow0.
\]
The constant polynomial has equal traces on both sides.  The same
conclusion follows for every polynomial, and uniform polynomial
approximation on the common spectral interval proves the assertion for
$g$.  Fourier block diagonalization and the Riemann-sum limit now yield
\eqref{eq:commontoeplitzspectral}.
\end{proof}

\subsection{Nilpotent Toeplitz realization}
Let $S_n$ have ones on its first superdiagonal and zeros elsewhere.
For $q(z)=\sum_{k=1}^d\alpha_kz^k$ put $N_n=q(S_n)$, which is
strictly upper triangular. For real $a,b$, set
\[
 X_n=a\Rea N_n+b\Ima N_n,
 \qquad h(t)=a\Rea q(e^{it})+b\Ima q(e^{it}).
\]

For the scalar Hermitian symbol $Q(e^{it})=h(t)$, the Toeplitz
section $A_n(Q)$ is exactly $X_n$.
Lemma~\ref{lem:commontoeplitz}, with $d=1$, therefore gives, for
every continuous real function $g$ on a common interval containing
the spectra of $X_n$ and the range of $h$,
\begin{equation}\label{eq:toeplitz}
 \lim_{n\to\infty}\frac1n\tr g(X_n)
 =\int_0^{2\pi}g(h(t))\frac{\dd t}{2\pi}.
\end{equation}

Taking $g(s)=|s|^p$ and summing the finitely many directions gives
\begin{equation}\label{eq:ridgeLimit}
 \frac1n\mathcal F_p(\Rea q(S_n),\Ima q(S_n))
 \longrightarrow\int_0^{2\pi}F(q(e^{it}))\frac{\dd t}{2\pi}.
\end{equation}
Thus (N) implies (P). It also follows that a strict failure of (P)
produces a finite nilpotent counterexample for all sufficiently
large $n$.

\begin{lemma}\label{lem:polar}
For $p>0$ and $h\in L^1(\T;\R)$ put $V(re^{i\theta})=r^ph(\theta)$ for $r>0$
and $\rho=h''+p^2h$ distributionally. For every $\varphi\in C_c^\infty(\C)$,
\[
 \langle\Delta V,\varphi\rangle
 =\int_0^\infty r^{p-1}\langle\rho,\varphi(re^{i\cdot})\rangle\dd r.
\]
If $\rho$ is a finite signed measure, $h$ has a Lipschitz representative
whose extension with $V(0)=0$ is continuous; $\rho\ge0$ makes it subharmonic.
If $p>1$, $h\in W^{2,q}(\T)$, $q>1$, and $q(p-2)+2>0$,
then $V\in C^1(\C)\cap W^{2,q}_{\rm loc}(\C)$.
\end{lemma}
\begin{proof}
Polar integration by parts gives the identity, with vanishing origin boundary terms
$O(r^p)+O(r^{p+1})$. A finite $\rho$ makes $h''$ finite; a periodic
primitive gives $h'\in BV$ and $h$ Lipschitz. The radial measure
$r^{p-1}\dd r\,\rho$ has no origin atom. In the Sobolev case $h\in C^1(\T)$; the gradient is $O(r^{p-1})$
and extends by zero, while the Hessian is bounded by
$Cr^{p-2}(|h|+|h'|+|h''|)$. The radial integral $\int_0^1r^{q(p-2)+1}\dd r$ is finite;
gradient boundary terms $O(\varepsilon^p)$ make it the global weak Hessian.
\end{proof}

\subsection{The homogeneous Poisson envelope}\label{sec:envelope}
For a continuous real function $G$ on $\C$, let
\begin{equation}\label{eq:envelope}
 \mathcal P(-G)(z)=
 \inf_{\substack{h\text{ holomorphic near }\overline\D\\h(0)=z}}
 \int_0^{2\pi}-G(h(e^{it}))\frac{\dd t}{2\pi}.
\end{equation}
The Poletsky--Rosay theorem says that this function is subharmonic
(possibly identically $-\infty$) and is the largest subharmonic
minorant of $-G$. The disc maps extend holomorphically past the closed
unit disc; constant discs are allowed. We use the theorem on $\C$
only; see \cite{Poletsky,Rosay} and the formulation in
\cite[Section~2]{LS}.
The analytic-test/majorant implication in the next lemma is also
Cole's principle \cite[Theorem~2.1]{HKV}, applied to $-G$ with
the test set $E=\{0\}$. We include the envelope argument to specify
the least majorant and the angular regularity used below.

\begin{lemma}
\label{lem:envelope}
Let $G$ be continuous and $p$-homogeneous. Suppose
$\int G(q(e^{it}))\dd t/(2\pi)\le0$ for every polynomial $q$ with
$q(0)=0$. Then
\[
 U_*=-\mathcal P(-G)
\]
is a finite continuous $p$-homogeneous least superharmonic majorant
of $G$. Its angular part is Lipschitz. It is harmonic on the open set
$\{U_*>G\}$.
\end{lemma}
\begin{proof}
Taylor polynomials approximate each disc in \eqref{eq:envelope}
uniformly on $\overline\D$, preserving its center; continuity of $G$
on their common compact range passes the boundary averages to the
limit. Put $v=\mathcal P(-G)$. Every disc centered at zero therefore
has boundary average of $-G$ at least zero, and the zero disc has
average zero.  Hence
\begin{equation}\label{eq:vzero}
 v(0)=0.
\end{equation}
The Poletsky--Rosay theorem therefore gives a subharmonic $v$ which is
not identically $-\infty$ and is the largest subharmonic minorant of
$-G$.

For $r>0$, multiplication by $r$ is a bijection between discs centered
at $z$ and discs centered at $rz$.  Since $G(r\zeta)=r^pG(\zeta)$,
\begin{equation}\label{eq:vhomogeneous}
 v(rz)=r^p v(z)
\end{equation}
first in the extended-real sense.  We next show that no exceptional
$-\infty$ values occur.

Local integrability, Fubini and \eqref{eq:vhomogeneous} give
$h\in L^1(\T)$ with $v(re^{i\theta})=r^ph(\theta)$ almost everywhere.
Nonnegative product tests supported in an annulus in Lemma~\ref{lem:polar} give
\begin{equation}\label{eq:angular-positive}
 \rho=h''+p^2h\ge0.
\end{equation}
Positivity on the compact circle makes $\rho$ a finite measure.
The lemma gives a continuous subharmonic $V$, with $V(0)=0$ and Lipschitz $h$.
Its almost-everywhere agreement with $v$ fixes the upper-semicontinuous
subharmonic representative, so $v=V$ everywhere.

Set $U_*=-v$.  Since $v\le-G$, $U_*\ge G$; and because $v$ is the
largest subharmonic minorant of $-G$, $U_*$ is the least superharmonic
majorant of $G$.  It remains to prove harmonicity on the noncontact
set.  Fix $z_0$ with $U_*(z_0)>G(z_0)$.  By continuity, choose a closed
disc $\overline D(z_0,r)$ contained in this set so small that
\[
 \inf_{\partial D}U_* > \sup_{\overline D}G.
\]
Let $H_D$ be the harmonic extension to $D$ of $U_*|_{\partial D}$.
The minimum principle gives
$H_D\ge\inf_{\partial D}U_*>G$ on $D$, while superharmonicity gives
$H_D\le U_*$. Replace $U_*$ by $H_D$ in $D$. The replacement
$\widetilde U$ is continuous and superharmonic on either side of
$\partial D$; at a boundary point it equals $U_*$ and is everywhere
at most $U_*$, so its circle averages are at most those of $U_*$.
This proves the super-mean-value inequality there as well. Since
$\widetilde U\ge G$, minimality forces $\widetilde U=U_*$.
Thus $U_*$ is harmonic on the noncontact set.
\end{proof}

If $G(-z)=G(z)$, the transformation $h\mapsto-h$ in the envelope
shows $U_*(-z)=U_*(z)$. Notice that this lemma did not use matrices;
the connection to matrices is \eqref{eq:ridgeLimit}.

\subsection{Regularity and completion of the proof}\label{sec:regularity}
\begin{lemma}\label{lem:regularity}
In Lemma~\ref{lem:envelope}, suppose the angular obstacle
$f(\theta)=G(e^{i\theta})$ belongs to $C^1(\T)\cap W^{2,q}(\T)$
for some $q>1$. Then the angular majorant $u$ belongs to
$C^1(\T)\cap W^{2,q}(\T)$, and, almost everywhere,
\begin{equation}\label{eq:contact}
 u''=f''\mathbf1_{\{u=f\}}-p^2u\mathbf1_{\{u>f\}}.
\end{equation}
If additionally $q(p-2)+2>0$, then $U_*\in\mathcal U_p$.
\end{lemma}
\begin{proof}
The preceding lemma gives Lipschitz $u$ and
$u''=-p^2u-\rho$ for a finite positive measure $\rho$.  Thus $u'$ has
a bounded-variation representative and therefore has one-sided limits
at every point.  On $O=\{u>f\}$, harmonicity gives
$u''=-p^2u$, so $\rho$ is supported on the closed contact set
$K=\{u=f\}$.

Fix $\theta\in K$.  Since $u-f\ge0$ and $(u-f)(\theta)=0$, the left
and right difference quotients give
\[
 u'_-(\theta)\le f'(\theta)\le u'_+(\theta).
\]
The atom of the distributional derivative $u''$ at $\theta$ is the
jump of the BV function $u'$, hence
\[
 u'_+(\theta)-u'_-(\theta)=-\rho\{\theta\}\le0.
\]
The two inequalities force
$u'_-(\theta)=u'_+(\theta)=f'(\theta)$ and
$\rho\{\theta\}=0$.  There are no atoms on $O$ either.  A BV function
whose derivative has no atoms has a continuous representative, so
$u'$ is continuous and $u'=f'$ on $K$.

Put $w=u'-f'$, a continuous BV function vanishing on $K$.
The derivative of a continuous BV function charges no level set. Indeed,
choose $\chi\in C_c^1(\R)$ with $\chi'(0)=1$. The continuous BV
chain rule gives
\[
 D[\varepsilon\chi(w/\varepsilon)]
 =\chi'(w/\varepsilon)Dw.
\]
This chain rule follows from the mean-value theorem in the Stieltjes
sums, using continuity of $w$. As $\varepsilon\downarrow0$, the
left side tends to zero distributionally, whereas dominated convergence
against $|Dw|$ makes the right side tend to $Dw|_{\{w=0\}}$.
Thus $Dw|_K=0$. On $O$ we have $Dw=(-p^2u-f'')\dd\theta$, so
$Dw=(-p^2u-f'')\mathbf1_O\dd\theta$ globally.
This proves \eqref{eq:contact} and $u\in C^1(\T)\cap W^{2,q}(\T)$.

Lemma~\ref{lem:polar} gives the asserted planar regularity;
domination and superharmonicity already follow from the envelope.
\end{proof}

For a finite ridge sum, each nonzero angular direction is a multiple
of $|\cos(\theta-\alpha)|^p$. If $1<p<2$, it belongs to $W^{2,q}$
for every $1<q<1/(2-p)$; if $p\ge2$, any finite $q>1$ is allowed.
These choices also satisfy the radial condition in
Lemma~\ref{lem:regularity}.

\begin{proof}[Proof of Theorem~\ref{thm:classification}]
Lemma~\ref{lem:commontoeplitz} proves (N)$\Rightarrow$(P), even if only
the indicated Toeplitz nilpotents are tested. Lemmas
\ref{lem:envelope} and \ref{lem:regularity} prove (P)$\Rightarrow$(S)
with the claimed least majorant and contact equations.
For (S)$\Rightarrow$(N), apply \eqref{eq:certificate}. Both remainder terms on
the right are nonnegative, while every eigenvalue of $N$ is zero
and $U(0)=0$. Thus $\mathcal F_p\le0$.

It remains to prove \eqref{eq:optimalcorrection}. The spectral identity
applied to $U_*$ gives, whenever $\spec T=\{z\}$,
$n^{-1}\mathcal F_p(\Rea T,\Ima T)\le U_*(z)$.
Conversely, apply Lemma~\ref{lem:commontoeplitz} to each translated real
pencil of $zI_n+q(S_n)$ to obtain
\[
 \lim_{n\to\infty}\frac1n\mathcal F_p
    \bigl(\Rea(zI_n+q(S_n)),\Ima(zI_n+q(S_n))\bigr)
 =\int_0^{2\pi}F(z+q(e^{it}))\frac{\dd t}{2\pi}.
\]
By Taylor approximation of discs centered at $z$, the supremum of
these integrals is $-\mathcal P(-F)(z)=U_*(z)$. All the shifted
Toeplitz matrices have spectrum $\{z\}$, which proves both equalities.
If \eqref{eq:universalcorrection} holds, these same tests force
$G(z)\ge U_*(z)$. The converse follows by summing $G\ge U_*$ in the
spectral inequality.
\end{proof}

\begin{remark}
If $p=2$ and $F(x,y)=\alpha x^2+2\beta xy+\gamma y^2$, then for
nilpotent $N=A+iB$, the identity $\tr N^2=0$ gives
$\tr A^2=\tr B^2$ and $\tr AB=0$. Thus
\[
 \mathcal F_2(A,B)=(\alpha+\gamma)\tr A^2.
\]
The nilpotent inequality holds exactly when $\alpha+\gamma\le0$.
This is also exactly $\Delta F\le0$, and $U=F$ is a superharmonic majorant.
This also identifies the harmonic angular mode of order two.
\end{remark}

\begin{remark}
Angular smoothing of an arbitrary majorant preserves superharmonicity
but need not preserve exact domination of the original obstacle.
Adding $\varepsilon|z|^p$ can restore domination but may destroy
superharmonicity. The contact-set regularity argument produces an
admissible majorant at the optimal constant and preserves both
conditions.
\end{remark}

\begin{corollary}
\label{cor:momentcone}\label{cor:sharp}
Let $1<p<\infty$. For fixed real directions $(a_j,b_j)_{j=1}^m$, the closed cones
generated by
\[
 \bigl(\tr|a_j\Rea N+b_j\Ima N|^p\bigr)_{j=1}^m
 \quad\text{and}\quad
 \left(\int_0^{2\pi}|a_j\Rea q(e^{it})+b_j\Ima q(e^{it})|^p
                    \frac{\dd t}{2\pi}\right)_{j=1}^m
\]
coincide, where $N$ ranges over finite nilpotents and $q(0)=0$.
Let $P$ be a finite signed real ridge sum and let $Q\not\equiv0$
be a finite real ridge sum with nonnegative coefficients, with trace
functionals $\mathcal P_p,\mathcal Q_p$. Then
\begin{align}
 \gamma_*
 &:=\sup_{N\ne0\ {\rm nilpotent}}
       \frac{\mathcal P_p(\Rea N,\Ima N)}
            {\mathcal Q_p(\Rea N,\Ima N)}\label{eq:sharpclass}\\
 &=\sup_{\substack{q\ne0\ {\rm polynomial}\\q(0)=0}}
       \frac{\int P(q(e^{it}))\dd t}{\int Q(q(e^{it}))\dd t}
 \in\R.\label{eq:sharpdisc}
\end{align}
The parameter $\gamma_*$ is the least real $\gamma$ for which
$P-\gamma Q$ has an admissible superharmonic majorant, and such a
majorant exists at $\gamma=\gamma_*$.
\end{corollary}
\begin{proof}
The two closed cones have the same polar by
Theorem~\ref{thm:classification}, so finite-dimensional separation
identifies them. Choose a term $\eta|ax+by|^p$ of $Q$, with
$\eta>0$ and $a^2+b^2>0$. Rotate a nilpotent by $a-ib$ and apply
Corollary~\ref{prop:coordinatepath} to obtain, for every real $u,v$,
\[
 \|uA+vB\|_p\le
 \frac{|ua+vb|+C_p|va-ub|}{a^2+b^2}\,\|aA+bB\|_p.
\]
Thus $|\mathcal P_p(A,B)|\le C\mathcal Q_p(A,B)$ for a finite
coefficient-dependent $C$. For this direction, $aA+bB=0$ would make
the rotated nilpotent skew-Hermitian, hence zero. A vanishing scalar
boundary part similarly forces $q=0$ by harmonic uniqueness.
Thus both denominators are positive. Finally, applying the classification to
$P-\gamma Q$ for each real $\gamma$ identifies the two suprema and
supplies a majorant at their common finite value.
\end{proof}

\section{Proof of the unified realization theorem}\label{sec:realizations}
The classification supplies the superharmonic majorants used below.
We first compare the optimal constant in the analytic-disc inequality
with a scalar operator norm on $L^p(\R)$, then use finite-dimensional
realization and sampling to prove the norm identities.

\subsection{Scalar localization and optimal constants}
\begin{lemma}\label{lem:threshold}
For the real pairs and weights in Theorem~\ref{thm:unified},
\begin{equation}\label{eq:thresholdupper}
 \Lambda_p(P)\le\|\mathbf H\|_{\rm scalar}
 \le\left(\sum_jw_jK_p(a_j,b_j)^p\right)^{1/p}<\infty.
\end{equation}
\end{lemma}
\begin{proof}
The scalar norm is finite by the individual real-pencil bounds
\cite{HKV,DGZ}. Let $q\ne0$, $q(0)=0$, and put
$f(t)=\Rea q(e^{it})$, $g(t)=\Ima q(e^{it})$.
They are real trigonometric polynomials with zero mean and
$H_\T f=g$, using multiplier $-i\operatorname{sgn}(k)$.
Choose a nonzero real Schwartz function $\chi$ with Fourier
transform supported in $(-1/2,1/2)$, with convention
$\widehat h(\xi)=\int h(t)e^{-it\xi}\dd t$.
For $L\ge1$, the positive- and negative-frequency components of
$f_L(t)=\chi(t/L)f(t)$ have Fourier supports in their respective
half-lines. Consequently,
\[
 Hf_L(t)=\chi(t/L)g(t).
\]
For continuous $2\pi$-periodic $h$,
\[
 \lim_{L\to\infty}\frac1L\int_\R|\chi(t/L)|^p h(t)\dd t
 =\left(\int_\R|\chi(s)|^p\dd s\right)
       \int_0^{2\pi}h(t)\frac{\dd t}{2\pi}.
\]
Approximate $h$ uniformly by trigonometric polynomials and apply
the Riemann--Lebesgue lemma to nonzero frequencies. Applying the
limit to the input and the weighted output $p$th powers shows
that each disc quotient is at most $\|\mathbf H\|_{\rm scalar}^p$.
The denominator is positive, since $\Rea q\equiv0$ and $q(0)=0$
would imply $q=0$.

Corollary~\ref{cor:sharp}, with $Q=|x|^p$, consequently gives
$\Lambda_p(P)\le\|\mathbf H\|_{\rm scalar}$.
\end{proof}

\subsection{Simultaneous finite realization}
The half-shift diagonalization of the scalar sign matrix is already
given in \cite[Lemma~1]{Zhou}: it is the case $\alpha=1/2$,
$V_0=1$, $V_m=0$ for $m\ne0$ of the model below, up to normalization.
We allow arbitrary Hermitian block coefficients and realize all
shifts using the same nilpotent matrix, as required for the joint
operator-norm identity.
For Hermitian $V_0,\ldots,V_{n-1}\in M_d$, put
$S=\sum_mV_m$, $\bar V=S/n$, $\theta_j=2\pi j/n$, and
$\widehat V(t)=\sum_{m=0}^{n-1}V_me^{-imt}$.
The indices $j,k,\ell$ in the finite matrices below range over
$0,\ldots,n-1$. Define the nilpotent block matrix
\begin{equation}\label{eq:finiteN}
 N_{jk}=\begin{cases}
 2\widehat V(\theta_j-\theta_k)/n,&j<k,\\
 0,&j\ge k.
 \end{cases}
\end{equation}
\begin{proposition}\label{prop:finiteFourier}
Let $A=\Rea N$, $B=\Ima N$, $\theta=\pi\alpha$, $0<\alpha<1$, and
\[
 F_{j\ell}=n^{-1/2}e^{-i\ell\theta_j},\qquad
 (G_\alpha)_{j\ell}=n^{-1/2}e^{-i(\ell+\alpha)\theta_j}.
\]
With tensor products by $I_d$ understood,
\begin{align}
 F^*AF&=\diag_\ell(V_\ell-\bar V),\label{eq:integerdiag}\\
 G_\alpha^*(\cos\theta\,A-\sin\theta\,B)G_\alpha
 &=\diag_\ell\left\{
 \frac{\sin\theta}{n}\sum_m\cot\frac{\pi(\ell+\alpha-m)}n\,V_m
                  -\cos\theta\,\bar V\right\}.\label{eq:twisted}
\end{align}
\end{proposition}
\begin{proof}
Hermitian coefficients give $\widehat V(-t)^*=\widehat V(t)$.
The full block circulant with entries $\widehat V(\theta_j-\theta_k)/n$
is diagonalized by $F$ into $\diag(V_\ell)$; subtracting its diagonal
$I_n\otimes\bar V$ gives \eqref{eq:integerdiag}.

For $t=j-k>0$ the second pencil has coefficient
$c_t=e^{-i\theta}\widehat V(2\pi t/n)/n$; for $t<0$ the factor is
$e^{i\theta}$. Thus $c_{t-n}=e^{2i\theta}c_t$ for $1\le t<n$, and
$c_te^{2\pi i(\ell+\alpha)t/n}$ is periodic. The pencil is therefore
diagonal in the $G_\alpha$ basis. If
$q=e^{2\pi i(\ell+\alpha-m)/n}$, then $q^n=e^{2i\theta}$ and
\[
 e^{-i\theta}\sum_{t=1}^{n-1}q^t
 =\sin\theta\cot\frac{\pi(\ell+\alpha-m)}n-\cos\theta.
\]
Expanding $\widehat V$ proves \eqref{eq:twisted}, with only scalar
Fourier factors passing the matrix coefficients.
\end{proof}
At $\alpha=1/2$ this is the half-integer diagonalization of $-B$.

The same elementary norm identity removes both the mean correction here
and the block diagonal in the Schur realization below:
\begin{equation}\label{eq:dilution}
 \|(Z/n)_{\ell=1}^n\|_{\ell^p(\Sp_p)}
 =\|Z\otimes I_n/n\|_p=n^{1/p-1}\|Z\|_p\longrightarrow0.
\end{equation}
The equality follows from the repeated singular values and the limit
uses $p>1$.

\begin{lemma}\label{lem:finitewindow}
For $a_j=\cos(\pi\alpha_j)$, $b_j=\sin(\pi\alpha_j)$,
$0<\alpha_j<1$, one has
\[
\|\mathbf R\|_{\ell^p(\Z;\Sp_p(\mathcal K))\to
\ell_p^m(\ell^p(\Z;\Sp_p(\mathcal K)))}\le\Lambda_p(P).
\]
\end{lemma}
\begin{proof}
Write $C=\Lambda_p(P)<\infty$, by Lemma~\ref{lem:threshold}. Taking $N^*$ simultaneously changes every
$b_j$ to $-b_j$ without changing $C$.
For a Hermitian sequence $V_0,\ldots,V_{n-1}$, put $S=\sum_mV_m$,
$\bar V=S/n$, and define
\[
 (\mathscr R_{n,\alpha}V)_\ell
 =\frac{\sin(\pi\alpha)}n\sum_{m=0}^{n-1}
          \cot\frac{\pi(\ell+\alpha-m)}n\,V_m.
\]
Proposition~\ref{prop:finiteFourier} uses the same $N$ and $A$ for
every $\alpha_j$. Summing the $p$th powers after the separate output
unitary conjugations therefore yields
\[
 \left(\sum_jw_j\|\mathscr R_{n,\alpha_j}V
                 -a_j(\bar V)_\ell\|_p^p\right)^{1/p}
 \le C\|V-(\bar V)_\ell\|_p.
\]
The norm here includes the finite sequence variable. By
\eqref{eq:dilution} and the triangle inequality,
\begin{equation}\label{eq:finiteerror}
 \left(\sum_jw_j\|\mathscr R_{n,\alpha_j}V\|_p^p\right)^{1/p}
 \le C\|V\|_p+
 \left(C+\Bigl(\sum_jw_j|a_j|^p\Bigr)^{1/p}\right)
                  n^{1/p-1}\|S\|_p.
\end{equation}
For sufficiently large $n$, regard a fixed finitely supported sequence
on $\Z$ as a sequence on $\Z/n\Z$ by assigning zero to unused
residue classes. Its norm and sum are independent of $n$.
On each fixed finite set of output indices, the factors
$n^{-1}\cot(\pi(\ell+\alpha_j-m)/n)$ converge to
$[\pi(\ell+\alpha_j-m)]^{-1}$, uniformly over the input support
and the finitely many $j$. Restrict \eqref{eq:finiteerror}
to that window and let $n\to\infty$; the error vanishes because
$p>1$. Increasing the output window gives the asserted estimate by
monotone convergence.

For arbitrary matrix coefficients use the Hermitian dilation
$V\mapsto\left(\begin{smallmatrix}0&V\\V^*&0\end{smallmatrix}\right)$.
All kernels are real, and the input and every output $p$th power
norm acquire the same factor two. Density of finitely supported
finite-rank sequences and common finite-dimensional compressions
give the bounded extension for every $\mathcal K$.
\end{proof}

\subsection{Sampling}
\begin{lemma}\label{lem:sampling}\label{lem:samplingnorm}
For $f\in C_c^\infty(\R;M_d)$ and $v^h_j=h^{1/p}f(hj)$,
\begin{equation}\label{eq:sampling}
 h^{-1/p}(R_\alpha v^h)_\ell
 \longrightarrow \cos(\pi\alpha)f(x)+\sin(\pi\alpha)Hf(x),
 \qquad x=h\ell,
\end{equation}
uniformly on compact sets of such $x$.
For the shifted pairs of Lemma~\ref{lem:finitewindow},
\begin{equation}\label{eq:continuous-from-discrete}
 \|\mathbf H\|_{L^p(\R;\Sp_p(\mathcal K))\to\ell_p^m(L^p(\R;\Sp_p(\mathcal K)))}
 \le\|\mathbf R\|.
\end{equation}
\end{lemma}
\begin{proof}
The left side of \eqref{eq:sampling} is
$\sin(\pi\alpha)\pi^{-1}\sum_k f(x-hk)/(k+\alpha)$.
Fix $\delta>0$ and $M=\lfloor\delta/h\rfloor$.
The constant term on $|k|\le M$ tends to
$\pi\cot(\pi\alpha)f(x)$ before the prefactor is applied, by the
paired partial fractions of $\pi\cot(\pi z)$.
The remaining inner sum has norm at most
$Ch\sum_{1\le|k|\le M}|k|/|k+\alpha|\le C_\alpha\delta$.
For the outer sum write
\[
 \frac{f(x-hk)}{k+\alpha}
 =h\,\frac{f(x-hk)}{hk+\alpha h}.
\]
Fix also $R>0$.  Because $f$ has compact support, for $|x|\le R$ all
nonzero summands lie in one fixed compact $t$-interval.  On
$|t|\ge\delta/2$ the functions
$(x,t)\mapsto f(x-t)/(t+\alpha h)$ converge uniformly, together with a
uniform Lipschitz bound in $t$, to $f(x-t)/t$.  The uniform
Riemann-sum estimate therefore gives
\begin{equation}\label{eq:outer-riemann}
 \sup_{|x|\le R}\left\|
 h\!\sum_{|hk|>\delta}\frac{f(x-hk)}{hk+\alpha h}
 -\int_{|t|>\delta}\frac{f(x-t)}t\,\dd t\right\|\longrightarrow0.
\end{equation}
Since $M=\lfloor\delta/h\rfloor$, the condition $|k|>M$ is exactly
equivalent to $|hk|>\delta$ for integer $k$, so \eqref{eq:outer-riemann}
is the limit for the actual outer sum.  In the missing symmetric
interval,
\[
 \left\|\int_{|t|<\delta}
 \frac{f(x-t)-f(x)}t\,\dd t\right\|\le C_R\delta
\]
by the mean-value theorem, uniformly for $|x|\le R$; the constant part
has zero symmetric principal value.  Together with the inner-sum bound
above, this proves the asserted uniform convergence by first letting
$h\downarrow0$ with $\delta$ fixed and then letting
$\delta\downarrow0$.

For the joint norm estimate, set $v_k^h=h^{1/p}f(hk)$.
Then $\sum_k\|v_k^h\|_p^p\to\|f\|_{L^p(\Sp_p)}^p$.
The local limit gives uniform convergence on each compact
output interval for all finitely many shifts. Sum the weighted
output $p$th powers, pass to the Riemann integrals on that interval,
and then increase the interval to $\R$. Monotone convergence gives
\eqref{eq:continuous-from-discrete}. Smooth finite-rank functions
are dense, so the estimate extends to the stated spaces.

\end{proof}

\subsection{Schur realization}\label{subsec:schurrealization}
\begin{lemma}\label{prop:schurequivalence}
For arbitrary real pairs and nonnegative weights,
$\sup_{\mathcal E}\|\mathbf T_{\mathcal E}\|=\Lambda_p(P)$.
The corresponding multiplier for the standard rank-one projections
on $\ell^2(\mathbb N)$ has the same norm.
\end{lemma}
\begin{proof}
If $X=X^*$ and $D_{\mathcal E}X=0$, then
$N=2\sum_{r<s}E_rXE_s$ is nilpotent with
$\Rea N=X$ and $\Ima N=H_{\mathcal E}X$.
Equation~\eqref{eq:jointconstant} gives the joint estimate.
For arbitrary zero-block-diagonal $X$, Hermitian dilation and the
flag $E_r\otimes I_2$ give the same bound, because every real pencil
commutes with taking adjoints and all $p$th powers gain a factor two.

For general $X$, let $u=n^{-1/2}(1,\ldots,1)$, $P=uu^*$, and put
\[
 Y_n=X\otimes P-D_{\mathcal E}X\otimes I_n/n,\qquad J\xi=\xi\otimes u.
\]
The flag $\mathcal E_n=(E_r\otimes e_{ss})$, in lexicographic order,
has $D_{\mathcal E_n}Y_n=0$. Direct block summation gives
\begin{equation}\label{eq:compression}
 J^*Y_nJ=X-D_{\mathcal E}X/n,\qquad
 J^*H_{\mathcal E_n}(Y_n)J=H_{\mathcal E}X.
\end{equation}
The within-block signs cancel since $\sum_{r,s}\operatorname{sgn}(r-s)=0$.
Compression is contractive for each output. Summing their $p$th
powers and using
\[
 \|Y_n-X\otimes P\|_p=n^{1/p-1}\|D_{\mathcal E}X\|_p\longrightarrow0
\]
proves the upper bound on all $X$.
Conversely, Schur triangularization realizes every nilpotent as a
strictly upper-triangular matrix for a rank-one flag and supplies
the lower bound. Finite principal sections and density give the
infinite ordered transform.
\end{proof}

\subsection{Canonical amplification and trace normalization}
\label{subsec:commonamplification}\label{subsec:cbproof}
The Schatten amplification converts the tracial matrix estimates in both
parts of the paper to the usual completely bounded estimates.  We record
this conversion once, keeping the canonical operator-space matrix norm
distinct from the tracial $L^p$ norm.

\begin{lemma}\label{lem:commonamplification}
Let $1<p<\infty$, and let $\mathcal Q,\mathcal R_1,\ldots,
\mathcal R_m$ be semifinite tracial von Neumann algebras.  Equip their
$L^p$ spaces and the finite $\ell^p$ direct sum with their canonical
operator-space structures.  For a bounded linear map
\[
 T:L^p(\mathcal Q)\longrightarrow
       \bigoplus_{j=1}^m{}_{\!p}L^p(\mathcal R_j)
\]
one has
\begin{equation}\label{eq:commoncbamplification}
 \|T\|_{\mathrm{cb}}
 =\sup_{n\ge1}
 \|\operatorname{id}_{M_n}\otimes T\|_\mathrm{tr},
\end{equation}
where the norm on the right is between the corresponding tracial spaces
\[
 L^p(M_n\,\overline\otimes\,\mathcal Q)
 \quad\hbox{and}\quad
 \bigoplus_{j=1}^m{}_{\!p}
 L^p(M_n\,\overline\otimes\,\mathcal R_j).
\]
The matrix trace in these spaces is the ordinary unnormalized trace.
Replacing it simultaneously in the domain and range by
$n^{-1}\operatorname{Tr}_n$ does not change any amplification norm.
\end{lemma}
\begin{proof}
The standard Schatten Fubini identification and matrix-norm recovery
formula are
\begin{align}
 \mathcal S_p^n[L^p(\mathcal Q)]
 &=L^p(M_n\,\overline\otimes\,\mathcal Q),
 \label{eq:commonSpFubini}\\
 \|x\|_{M_n(L^p(\mathcal Q))}
 &=\sup_{\substack{a,b\in M_n\\
       \|a\|_{\mathcal S_{2p}^n},\,
       \|b\|_{\mathcal S_{2p}^n}\le1}}
       \|(a\otimes I)x(b\otimes I)\|_
          {L^p(M_n\,\overline\otimes\,\mathcal Q)}.
 \label{eq:commonLpMatrixNorm}
\end{align}
Here the Schatten norms are unnormalized.  These identities, and the
equivalence of complete contraction with contraction at every Schatten
$p$-amplification, are recalled in
\cite[Section~2, equations~(2.8)--(2.9)]{rz:JRS}; see also
\cite[Corollary~1.4 and Lemma~1.7]{Pisier}.
The range is the $L^p$ space of the finite direct-sum algebra
$\bigoplus_j\mathcal R_j$, with the direct-sum trace, so the same
identifications apply there.

For completeness, if all tracial amplification norms are at most $C$,
entrywise amplification commutes with multiplication by the scalar
matrices $a,b$.  Formula~\eqref{eq:commonLpMatrixNorm} therefore gives
\[
 \|(\operatorname{id}_{M_n}\otimes T)x\|_{M_n(F)}
 \le C\|x\|_{M_n(L^p(\mathcal Q))},
 \qquad F=\bigoplus_{j=1}^m{}_{\!p}L^p(\mathcal R_j).
\]
Thus $\|T\|_{\mathrm{cb}}\le C$.  Conversely, the standard Schatten
amplification characterization bounds every
$\operatorname{id}_{\mathcal S_p^n}\otimes T$ by
$\|T\|_{\mathrm{cb}}$; \eqref{eq:commonSpFubini} identifies its norm
with the tracial amplification norm.  This proves
\eqref{eq:commoncbamplification}.  Normalizing the matrix trace multiplies
both the domain and range norms by $n^{-1/p}$, proving the final assertion.
\end{proof}

In particular, a uniform estimate for the amplified flags in the Hilbert
part and a uniform estimate for $M_n(\mathcal B)\subset M_n(\mathcal M)$
in the subdiagonal part give the corresponding completely bounded
estimates by the same lemma.  The sharp tracial estimates themselves
are proved separately in the two parts.

In the present cases the
Fubini identifications for vector-valued noncommutative $L^p$
spaces are
\begin{equation}\label{eq:matchingFubini}
 \Sp_p^d[L^p(\R)]=L^p(\R;\Sp_p^d),\qquad
 \Sp_p^d[\Sp_p(\ell^2)]=\Sp_p(\C^d\otimes\ell^2).
\end{equation}
The norms in the right-hand spaces use the unnormalized trace, consistently
with the rest of the paper. Formula \eqref{eq:commoncbamplification}
compares the corresponding amplification norms of a map.

\begin{proof}[Proof of Theorem~\ref{thm:unified}]
Lemma~\ref{lem:threshold} gives finiteness, and
Corollary~\ref{cor:sharp} gives the variational and superharmonic
majorant statements. For the shifted pairs, scalar embedding, sampling,
finite-window transference, and the same lemma give the following chain of inequalities:
\begin{equation}\label{eq:jointchain}
 \|\mathbf H\|_{\rm scalar}\le\|\mathbf H\|_{\Sp_p}
 \le\|\mathbf R\|_{\Sp_p}\le\Lambda_p(P)
 \le\|\mathbf H\|_{\rm scalar}.
\end{equation}
Applying sampling also with $\mathcal K=\C$ and embedding scalar
sequences shows $\|\mathbf R\|_{\rm scalar}=\Lambda_p(P)$.
Lemma~\ref{prop:schurequivalence} supplies the Schur model.

For general pairs with $b_j\ne0$, multiplying an output by $-1$
if necessary permits the representation
$(a_j,b_j)=\varepsilon_j r_j(\cos(\pi\alpha_j),
\sin(\pi\alpha_j))$, where $\varepsilon_j\in\{-1,1\}$,
$r_j>0$, and $0<\alpha_j<1$.
Absorb $r_j^p$ into $w_j$. The outputs with $b_j=0$ add the
same constant $\sum_{b_j=0}w_j|a_j|^p$ to every norm ratio to the
$p$th power, including the nilpotent and disc ratios. This proves
all real pairs, including the case of zero weights.
Finally $\tr N^2=0$ gives $\tr A^2=\tr B^2$ and $\tr AB=0$,
which proves the stated quadratic value.

For completely bounded norms, the same Fubini identifications hold for
$\ell^p(\Z)$ and finite $\ell^p$ direct sums.
Thus the $d$th Schatten amplification of $\mathbf H$ or $\mathbf R$
is exactly its $\Sp_p^d$-valued version, whose norm is $\Lambda_p(P)$.
The $d$th amplification of $\mathbf T_{\mathcal E}$ is the same
joint pencil for the block flag $(I_d\otimes E_j)$, up to interchange
of tensor factors. Hence its norm is at most $\Lambda_p(P)$, uniformly
in $d$ and the flag. Taking the supremum in the amplification
characterization and using the operator-norm lower bounds proves
the completely bounded norm assertions of the theorem.

For the standard infinite Schur order, finite principal sections
and density at every amplification level give the same upper
bound; nilpotent tests on rank-one flags give the lower bound.
\end{proof}

\begin{proof}[Proof of Corollary~\ref{thm:realpencil}]
For one output, the scalar theorem of \cite{HKV,DGZ} gives
\eqref{eq:Koblique}. Theorem~\ref{thm:unified} and the preceding
infinite-order argument give every stated norm equality.
Scalar reflection, homogeneity, and duality give
\eqref{eq:Kproperties}--\eqref{eq:constant}.
Only the signed diagonal convention remains, and it is proved next.
\end{proof}

\begin{corollary}\label{cor:rectangular}
Use the real pairs and weights of Theorem~\ref{thm:unified}. For real row
levels $\lambda_1,\ldots,\lambda_r$ and column levels $\mu_1,\ldots,\mu_s$,
define the joint rectangular multiplier
\[
 (\mathbf M_{\lambda,\mu}X)_\ell
 =\bigl(w_\ell^{1/p}
 [a_\ell+ib_\ell\operatorname{sgn}(\lambda_j-\mu_k)]X_{jk}\bigr)_{j,k}.
\]
On rectangular Schatten matrices with their canonical corner operator-space
structure, it is the restriction of a single joint flag multiplier on $\C^{r+s}$
to a rectangular corner, under a complete isometry. Hence its
operator and completely bounded norms are at most $\Lambda_p(P)$,
and their suprema over all finite lists and sizes equal $\Lambda_p(P)$.
For one output this gives the bound and uniform sharp constant $K_p(a,b)$.
\end{corollary}
\begin{proof}
Embed $X$ in the off-diagonal corner from the $\mu$-space to the
$\lambda$-space of their direct sum, and order the combined levels,
grouping ties. This embedding preserves singular values at every
Schatten amplification and is completely isometric by
\eqref{eq:commonLpMatrixNorm}. The resulting joint flag preserves the
corner and has exactly the displayed symbols.
Theorem~\ref{thm:unified} gives both upper bounds. Equal square lists
recover the standard rank-one joint orders; their uniform sharpness
in Lemma~\ref{prop:schurequivalence} gives both lower bounds.
\end{proof}

\subsection{Diagonal conventions for the signed Schur multiplier}
\label{subsec:diagonalnormalization}
For $n\ge1$, define Schur multipliers on $\Sp_p^n$ by
\[
 (S_0^{(n)}X)_{jk}=\operatorname{sgn}(j-k)X_{jk},\qquad
 S_+^{(n)}=S_0^{(n)}+D_n,\qquad S_-^{(n)}=S_0^{(n)}-D_n,
\]
where $D_n$ is the diagonal conditional expectation. Thus $S_+^{(n)}$ has symbol
$+1$ on $j\ge k$ and $-1$ on $j<k$, while $S_-^{(n)}$ has diagonal
value $-1$. Notice that $\HSch=iS_0$ on finite principal sections.

\begin{lemma}\label{lem:diagonalnormalization}
Fix $1<p<\infty$. For every $d\ge1$, put
\[
 \|S\|_{p,d}
 =\|S\otimes\id_{\Sp_p^d}\|_{\Sp_p(\C^n\otimes\C^d)
                                      \to\Sp_p(\C^n\otimes\C^d)}
\]
when $S$ acts on $n\times n$ matrices. Then
\begin{equation}\label{eq:diagonalcomparisonfinite}
 \|S_0^{(n)}\|_{p,d}\le\|S_+^{(n)}\|_{p,d}
 =\|S_-^{(n)}\|_{p,d}\le\|S_0^{(2n)}\|_{p,d}.
\end{equation}
Consequently the dimension suprema agree for all three conventions,
both for operator norms and completely bounded norms. The infinite
multipliers $S_0,S_+,S_-$ extend to $\Sp_p(\ell^2)$ and satisfy
\begin{equation}\label{eq:diagonalnormsharp}
 \|S_\varepsilon\|_{\Sp_p\to\Sp_p}
 =\|S_\varepsilon\|_{\mathrm{cb},\Sp_p}=C_p,
 \qquad \varepsilon\in\{0,+,-\}.
\end{equation}
\end{lemma}
\begin{proof}
Take $\lambda_j=2j$ and $\mu_k=2k-1$ in
Corollary~\ref{cor:rectangular}. Since
$\operatorname{sgn}(\lambda_j-\mu_k)$ is $+1$ for $j\ge k$ and $-1$ otherwise,
$iS_+^{(n)}$ is a completely isometric corner restriction of
$iS_0^{(2n)}$. This proves the last inequality in
\eqref{eq:diagonalcomparisonfinite} at every amplification level;
it is the row--column interleaving of \cite[Proposition~6.1]{NR}.

Let $R_ne_j=e_{n+1-j}$ and let $\mathcal R_n$ be conjugation by $R_n$.
Direct comparison of symbols gives
\[
 S_-^{(n)}=-\mathcal R_n S_+^{(n)}\mathcal R_n,
 \qquad S_0^{(n)}=\tfrac12(S_+^{(n)}+S_-^{(n)}).
\]
Conjugation by $R_n\otimes I_d$ is isometric, so the first identity
identifies the two norms and the second gives the remaining inequality
in \eqref{eq:diagonalcomparisonfinite}. Taking the supremum over $n$
proves equality of the dimension suprema at each $d$; then taking the
supremum over $d$ and using \eqref{eq:commoncbamplification} proves the
completely bounded norm assertion. Finite principal sections and density give
the infinite extensions. Since $\HSch=iS_0$, the already proved
zero-diagonal part of Corollary~\ref{cor:HilbertCB} gives the common
value $C_p$, proving \eqref{eq:diagonalnormsharp} and completing that
corollary.
\end{proof}

The norm equalities concern suprema over matrix dimensions.
The norms need not agree in a fixed dimension: $S_0^{(1)}=0$
whereas $S_+^{(1)}=I$. Order reversal and multiplication by $i$ preserve
the norms, reconciling the conventions in \cite[Theorem~6.2]{NR}.

\section{Accretive spectral comparison in finite algebras}\label{rz:sec:spectral}
Throughout this section and the Riesz arguments below, $(\M,\tau)$ and $\B$ are as in Section~\ref{sec:intro}. The letter $H=G+iK$ denotes a bounded accretive element; it does not denote the Hilbert transform in this section. The symbol $\operatorname{Det}_\tau$ denotes the Fuglede--Kadison determinant, whereas $\Delta$ continues to denote the planar Laplacian in the matrix spectral comparisons.

For a bounded $A\in\M$, its decreasing singular-value function is
\[
 \mu_v(A)=\inf\{c\geq0:\tau 1_{(c,\infty)}(|A|)\leq v\},\qquad 0<v<1.
\]
Spectral layer-cake integration gives $\tau f(|A|)=\int_0^1 f(\mu_v(A))\,dv$ for continuous $f$. The equivalent projection formula is
\[
 \mu_v(A)=\inf\{\|Ae\|:e=e^*=e^2\in\M,\ \tau(\Id-e)\leq v\}.
\]
To see the equivalence, use the spectral projection below any level larger than the distribution inverse for one direction. For the other, if $\|Ae\|\leq c$, the spectral projection $f=1_{(c,\infty)}(|A|)$ has $f\wedge e=0$; the polar decomposition of $(\Id-e)f$ gives $\tau f\leq\tau(\Id-e)$. In particular,
\begin{equation}\label{rz:mulipschitz}
 |\mu_v(A)-\mu_v(B)|\leq\|A-B\|.
\end{equation}

\begin{theorem}\label{rz:kernelcomparison}
Let $H=G+iK\in\M$, where $G\ge0$ and $K=K^*$, and put
\[
 g(v)=\mu_v(G),\qquad x(v)=\mu_v(H^*H),\qquad
 z(v)=g(v)+i\sqrt{x(v)-g(v)^2}.
\]
The square root is well defined. Let $w:(0,\infty)\to[0,\infty)$ be
measurable with
\begin{equation}\label{rz:kernelweight}
 \int_0^1\frac{w(u)}u\,du+
 \int_1^\infty\frac{w(u)}{u^2}\,du<\infty.
\end{equation}
For real $a,b$, define, on the closed right half-plane,
\begin{equation}\label{rz:kernelF}
 \varphi(z)=a+bz+\int_0^\infty
       w(u)[z/u-\log(1+z/u)]\,du.
\end{equation}
Define $\varphi(H)$ by the corresponding norm-convergent operator integral,
with constant term $a\Id$. Then
\begin{equation}\label{rz:kerneldefect}
 \int_0^1\Rea \varphi(z(v))\,dv-\Rea\tau \varphi(H)
 =\frac12\int_0^\infty w(u)\mathcal L_u(G,H^*H)\,du\ge0,
\end{equation}
where
\begin{equation}\label{rz:Ldef}
 \mathcal L_u(G,X)=\tau\log(\Id+2G/u+X/u^2)
 -\int_0^1\log(1+2g(v)/u+x(v)/u^2)\,dv\ge0.
\end{equation}
All assertions, including the exact defect identity, hold for
$G\ge0$ without invertibility.
\end{theorem}

\begin{corollary}\label{rz:comparison}
For $H,G,g,x,z$ as in Theorem~\ref{rz:kernelcomparison} and $1<s<2$,
put $\eta_s=s\sin(\pi(s-1))/\pi$.
The principal accretive powers, including boundary values, are
defined in Subsection~\ref{rz:mellin}. Then
\begin{equation}\label{rz:spectralcomparison}
 \Rea\tau H^s\le\int_0^1\Rea z(v)^s\,dv,
\end{equation}
and for every $G\ge0$ the exact defect is
\begin{equation}\label{rz:comparisondefect}
 \int_0^1\Rea z(v)^s\,dv-\Rea\tau H^s
 =\frac{\eta_s}{2}\int_0^\infty
      u^{s-1}\mathcal L_u(G,H^*H)\,du\ge0.
\end{equation}
\end{corollary}

\begin{corollary}\label{rz:budget}
For $G\geq0$, $K=K^*$ in $\M$, and $1<s<2$, set
\[
 H=G+iK,\qquad\alpha=\frac\pi{2s},\qquad r=\sec\alpha,\qquad a=\tan\alpha.
\]
Principal accretive powers, including boundary values, are defined in Subsection~\ref{rz:mellin}. Then
\begin{equation}\label{rz:D}
 \quad\Ds(G,K)=r^s\tau G^s-\tau|H|^s-a\Rea\tau H^s\geq0.\quad
\end{equation}
Equality holds for $H=(1\pm ia)G$ with $G\geq0$.
\end{corollary}

\subsection{Singular-value dominance and logarithmic rearrangement}
\begin{lemma}\label{rz:dominance}
If $H=G+iK$ with $G\geq0$, then
\begin{equation}\label{rz:mu}
 \mu_v(G)\leq\mu_v(H),\qquad 0<v<1.
\end{equation}
\end{lemma}
\begin{proof}
For $c>0$ set $e=1_{(c,\infty)}(G)$ and $f=1_{(c,\infty)}(|H|)$. Their ranges satisfy $e\wedge(\Id-f)=0$. Indeed, a nonzero vector $\xi$ in the intersection would give
\[
 \Rea\langle H\xi,\xi\rangle=\langle G\xi,\xi\rangle>c\|\xi\|^2,
 \qquad \|H\xi\|\leq c\|\xi\|,
\]
a contradiction. Take the polar decomposition $fe=v|fe|$ in $\M$. The initial support is $e$, because $e\wedge(\Id-f)=0$, and the range support is dominated by $f$. Hence
\[
 \tau e=\tau(v^*v)=\tau(vv^*)\leq\tau f.
\] Thus the distribution function of $G$ is dominated by that of $|H|$ at every $c>0$. Taking generalized inverses proves \eqref{rz:mu}.
\end{proof}

\begin{lemma}\label{rz:logsum}
For bounded $A,B\geq0$,
\begin{equation}\label{rz:logineq}
 \tau\log(\Id+A+B)\geq\int_0^1\log(1+\mu_v(A)+\mu_v(B))\,dv.
\end{equation}
\end{lemma}
\begin{proof}
We give a direct proof of the relevant rearrangement inequality; see \cite{rz:FK} for its general formulation. For $A\geq0$ and $0\leq q\leq1$, the Ky Fan formula is
\[
 \int_0^q\mu_v(A)\,dv
 =\sup\{\tau(AQ):0\leq Q\leq\Id,\ \tau Q=q\}.
\]
Indeed, layer-cake integration and positivity give
\[
 \tau(AQ)=\int_0^\infty\tau\bigl(Q1_{(c,\infty)}(A)\bigr)\,dc
 \leq\int_0^\infty\min\{q,\tau 1_{(c,\infty)}(A)\}\,dc
 =\int_0^q\mu_v(A)\,dv.
\]
For the reverse inequality, take the spectral projection above the threshold $\mu_q(A)$ and add, if necessary, a scalar multiple of the level projection to obtain trace $q$. The endpoint cases are immediate. This choice uses positive contractions and works in algebras with atoms.

Put $h(v)=\mu_v(A+B)$ and $\nu(v)=\mu_v(A)+\mu_v(B)$. Applying the formula to the sum gives
\[
 \int_0^q h(v)\,dv\leq\int_0^q\nu(v)\,dv\quad(0\leq q\leq1),
 \qquad\int_0^1h=\int_0^1\nu.
\]
Both functions are decreasing, so for $c\geq0$,
\[
 \int_0^1(h-c)_+\,dv
 =\sup_{0\leq q\leq1}\left(\int_0^q h\,dv-cq\right)
 \leq\int_0^1(\nu-c)_+\,dv.
\]
Using the elementary identity
\[
 t-\log(1+t)=\int_0^\infty\frac{(t-c)_+}{(1+c)^2}\,dc\qquad(t\geq0)
\]
and Tonelli, we obtain
\[
 \int_0^1[h-\log(1+h)]\,dv\leq\int_0^1[\nu-\log(1+\nu)]\,dv.
\]
The linear terms cancel, which proves \eqref{rz:logineq}.
\end{proof}

\subsection{Logarithmic kernels and the exact comparison}
For a bounded accretive $A$ and $u>0$,
\begin{equation}\label{rz:logresolvent}
 \log(\Id+A/u)=\int_0^1(A/u)(\Id+tA/u)^{-1}\,dt,
 \qquad\|(\Id+tA/u)^{-1}\|\le1.
\end{equation}
The logarithm is its principal holomorphic value, since
$\Rea(\Id+A/u)\ge\Id$. Subtracting \eqref{rz:logresolvent} gives
\[
 A/u-\log(\Id+A/u)
 =\int_0^1t(A/u)^2(\Id+tA/u)^{-1}\,dt.
\]
Consequently, if $\|A\|\le M$, then for every $u>0$,
\begin{equation}\label{rz:kernelcompensation}
 \|A/u-\log(\Id+A/u)\|\le
 \min\left\{\frac{2M}{u},\frac{M^2}{2u^2}\right\}.
\end{equation}
The same bound holds for scalar $A$ in the closed right half-plane. Thus \eqref{rz:kernelweight} makes all the
compensated integrals norm convergent, uniformly on bounded sets of
accretive elements. For each fixed $u$, resolvent continuity in
\eqref{rz:logresolvent} and dominated convergence also show that $\varphi(A)$
is norm continuous under uniformly bounded accretive perturbations.

\begin{lemma}\label{rz:exponentialdeterminant}
For every $B\in\M$, $\log\operatorname{Det}_\tau(e^B)=\Rea\tau B$.
\end{lemma}
\begin{proof}
Determinant multiplicativity and the norm Taylor expansion give
\[
 2\log\operatorname{Det}_\tau(e^B)
 =n\tau\log(e^{B^*/n}e^{B/n})
 =\tau(B^*+B)+O(n^{-1}).
\]
Indeed, the product inside the logarithm is
$\Id+(B^*+B)/n+O(n^{-2})$ in norm. Letting $n\to\infty$ proves the
assertion.
\end{proof}
For bounded accretive $H=G+iK$, take $B=\log(\Id+H/u)$, its bounded
principal logarithm. Since $e^B=\Id+H/u$, the lemma gives
\begin{equation}\label{rz:logbridge}
 \Rea\tau\log(\Id+H/u)
 =\tfrac12\tau\log\left(\Id+2G/u+H^*H/u^2\right).
\end{equation}

\begin{proof}[Proof of Theorem~\ref{rz:kernelcomparison}]
Lemma~\ref{rz:dominance} gives $x(v)\ge g(v)^2$ and $|z(v)|\le\|H\|$.
Lemma~\ref{rz:logsum}, applied to $2G/u$ and $H^*H/u^2$, gives
$\mathcal L_u\ge0$. The trace bridge \eqref{rz:logbridge} gives directly
\[
 \int_0^1\Rea[z(v)/u-\log(1+z(v)/u)]\,dv
 -\Rea\tau[H/u-\log(\Id+H/u)]
 =\tfrac12\mathcal L_u(G,H^*H).
\]
The bounds \eqref{rz:kernelcompensation} and \eqref{rz:kernelweight}
justify integrating this identity against $w$ with both compensations
intact. The real affine terms agree, so this proves
\eqref{rz:kerneldefect}, including absolute convergence and $G\ge0$.
\end{proof}

\subsection{The compensated determinant-pair identity}\label{rz:mellin}
Set
\[
 \eta_s=\frac{s\sin(\pi(s-1))}{\pi}>0.
\]
For $\Rea z>0$ the following convergent, compensated formula holds:
\begin{equation}\label{rz:scalarMellin}
 z^s=\eta_s\int_0^\infty u^{s-1}\left[\frac zu-\log\left(1+\frac zu\right)\right]du.
\end{equation}
For positive real $z$, integration by parts gives
\[
 \int_0^\infty u^{s-1}\left[\frac zu-\log(1+z/u)\right]du
 =\frac{z^2}{s}\int_0^\infty\frac{u^{s-2}}{u+z}\,du
 =\frac{\pi z^s}{s\sin(\pi(s-1))}.
\]
The boundary terms vanish since $1<s<2$. On compact subsets of the right half-plane the compensated integrand is uniformly $O(u^{-1}+|\log u|)$ at zero and $O(u^{-2})$ at infinity. It therefore defines a holomorphic function, and analytic continuation proves \eqref{rz:scalarMellin}. The two terms are never integrated separately.

For $\Rea A\ge\delta\Id>0$, choose a Dunford contour enclosing
its spectrum in $\{\Rea z\ge\delta/2\}$. The uniform compensated
scalar bounds and the bounded resolvent on this contour permit
interchange with the contour integral. Hence \eqref{rz:scalarMellin}
holds as a norm-convergent operator integral and gives the principal
holomorphic power $A^s$.
For every bounded accretive $A$, we use the same formula for its
boundary value:
\begin{equation}\label{rz:boundaryMellin}
 A^s=\eta_s\int_0^\infty u^{s-1}
       [A/u-\log(\Id+A/u)]\,du.
\end{equation}
The bounds \eqref{rz:kernelcompensation} show norm convergence and
continuity under uniformly bounded accretive perturbations, because
$u^{s-2}$ at zero and $u^{s-3}$ at infinity are integrable precisely
for $1<s<2$. This identifies \eqref{rz:boundaryMellin} with the
continuous extension of the principal holomorphic power.

\begin{lemma}\label{rz:mellinderivatives}
Let $0<\beta<1$, $b_\beta=\pi/\sin(\pi\beta)$, and $\Rea a>0$.
Then, locally uniformly in $a$, the following integrals converge absolutely:
\begin{align}
 K_\beta(a)&:=\int_0^\infty t^\beta
       [(t+a)^{-1}-t^{-1}]\,dt=-b_\beta a^\beta,
       \label{rz:rayresolvent}\\
 \int_0^\infty t^\beta
       [\log(1+a/t)-a/(t+a)]\,dt
   &=\frac{\beta b_\beta}{1+\beta}a^{1+\beta}.
       \label{rz:raylogprimitive}
\end{align}
\end{lemma}
\begin{proof}
By \eqref{rz:scalarMellin},
$I(a):=\int_0^\infty t^\beta[a/t-\log(1+a/t)]\,dt
=b_\beta a^{1+\beta}/(1+\beta)$.
The differentiated integrand is locally uniformly $O(t^{\beta-1})$
at zero and $O(t^{\beta-2})$ at infinity. Thus differentiation is
valid, and the two displayed integrals are $-I'(a)$ and $aI'(a)-I(a)$.
The logarithmic integrand is $O(t^\beta(1+|\log t|))$ at zero and
$O(t^{\beta-2})$ at infinity, proving the convergence assertion.
\end{proof}

Write $X=H^*H$. Taking the real trace in \eqref{rz:scalarMellin} gives
\begin{equation}\label{rz:detpair}
 \Rea\tau H^s=\eta_s\int_0^\infty u^{s-1}
 \left[\frac{\tau G}{u}-\frac12\tau\log(\Id+2G/u+X/u^2)\right]du.
\end{equation}
This identity retains the noncommutative pair $(G,X)$.

\begin{proof}[Proof of Corollary~\ref{rz:comparison}]
Take $a=b=0$ and $w(u)=\eta_su^{s-1}$ in
Theorem~\ref{rz:kernelcomparison}. Condition~\eqref{rz:kernelweight}
holds for $1<s<2$, and the Mellin formula identifies $\varphi(z)$ with
$z^s$, including the boundary by continuity. Therefore
\eqref{rz:kerneldefect} is exactly \eqref{rz:comparisondefect},
which also gives \eqref{rz:spectralcomparison}.
\end{proof}

\subsection{The scalar sharp inequality and positive certificate}
\begin{lemma}\label{rz:scalar}
For $g\geq0$, $k\in\mathbb R$, and the constants in Corollary~\ref{rz:budget},
\begin{equation}\label{rz:scalarbudget}
 d_s(g,k):=r^s g^s-(g^2+k^2)^{s/2}-a\Rea(g+ik)^s\geq0.
\end{equation}
For $g>0$, equality holds exactly when $|k|=ag$.
\end{lemma}
\begin{proof}
By conjugation symmetry take $k\geq0$. For $g>0$ write $\theta=\arctan(k/g)$. The quantity to be bounded above is
\[
 F(\theta)=\sec^s\theta[1+a\cos(s\theta)],\qquad 0\leq\theta<\pi/2.
\]
Its derivative is
\[
 F'(\theta)=s\sec^{s+1}\theta[\sin\theta-a\sin((s-1)\theta)].
\]
The ratio $\sin\theta/\sin((s-1)\theta)$ is strictly decreasing: its logarithmic derivative is $\cot\theta-(s-1)\cot((s-1)\theta)<0$, since $t\cot t$ strictly decreases on $(0,\pi/2)$. At $\theta=\alpha$ the ratio is $\tan\alpha=a$. Thus $F$ increases up to $\alpha$ and decreases after it, with maximum $F(\alpha)=r^s$. This proves \eqref{rz:scalarbudget} and the equality assertion. Continuity as $g\downarrow0$ covers $g=0$.
\end{proof}

\begin{proof}[Proof of Corollary~\ref{rz:budget}]
Spectral integration gives $\tau G^s=\int_0^1g(v)^s\,dv$ and
$\tau|H|^s=\int_0^1|z(v)|^s\,dv$.
Using \eqref{rz:comparisondefect}, we obtain for every $G\ge0$
the exact certificate
\begin{equation}\label{rz:certificate}
 \Ds(G,K)=\int_0^1d_s(g(v),\sqrt{x(v)-g(v)^2})\,dv
 +\frac{a\eta_s}{2}\int_0^\infty
       u^{s-1}\mathcal L_u(G,H^*H)\,du.
\end{equation}
Both terms are nonnegative by Lemma~\ref{rz:scalar} and
\eqref{rz:Ldef}, proving the bound. For $H=(1\pm ia)G$,
normal functional calculus gives $|H|^s=r^sG^s$ and $\Rea H^s=0$,
which proves attainment.
\end{proof}

\begin{remark}
The certificate compares the decreasing spectra of $G$ and $H^*H$ and retains their noncommutative logarithmic gap. No joint spectral measure for $(G,K)$ is required.
\end{remark}

\subsection{Determinant covariance and its exact residual}
We prove the trace inequality stated in Theorem~\ref{rz:detcovariance},
including its exact nonnegative residual.

\begin{proof}[Proof of Theorem~\ref{rz:detcovariance}]
Lemma~\ref{rz:exponentialdeterminant} and the norm logarithmic series give
$\log\operatorname{Det}_\tau(\Id+\zeta N)
 =\Rea(\zeta\tau N)+O(|\zeta|^2)$ near zero.
Nonnegativity for $\zeta=\pm t,\pm it$ as $t\downarrow0$ forces $\tau N=0$.

Put $\theta=\pi/[2(1+\beta)]$, $\lambda_t=te^{i\theta}$,
$R_t=(D+\lambda_t)^{-1}$, $W_t=R_t^*R_t$, and
$T_t(z)=[R_t(\lambda_t+zN)]^*[R_t(\lambda_t+zN)]$.
The determinant assumption gives $\operatorname{Det}_\tau(\lambda_t+zN)\ge t$.
Multiplicativity then gives $\operatorname{Det}_\tau(T_t(z))>0$.
Since $T_t(z)$ is bounded, $\tau\log^+T_t(z)<\infty$; positivity
of its determinant also gives $\tau\log^-T_t(z)<\infty$.
Thus $\tau\log T_t(z)$ is a finite spectral integral. This argument
uses no bounded inverse of $\lambda_t+zN$. Define
\[
 E_t(z)=\tfrac12\tau[T_t(z)-\Id-\log T_t(z)]\ge0,
 \qquad J_t(z)=\log\operatorname{Det}_\tau(\lambda_t+zN)-\log t\ge0.
\]
The first inequality is the spectral integral of $x-1-\log x\ge0$.
Writing $b_t=\log t-\tau\log|D+\lambda_t|$, multiplicativity gives
$E_t+J_t=\tfrac12\tau(T_t-\Id)-b_t$. Expanding only this sum yields
\begin{equation}\label{rz:covariance_entropy}
 E_t(z)+J_t(z)
 =F_t+\tfrac12|z|^2\tau(N^*W_tN)
       +\Rea(te^{-i\theta}z\tau(W_tN)),
\end{equation}
where $F_t=\tau\log|\Id+D/\lambda_t|-\Rea\tau(R_tD)
+\tfrac12\tau(D^2W_t)$.

Set $C_\beta=b_\beta\sin(\beta\theta)/\sin\theta$ and
$d_\beta=b_\beta\sin((1+\beta)\theta)/\sin\theta$.
For $d>0$, put $a_\pm=de^{\pm i\theta}$ and
$w_t(d)=(t^2+2dt\cos\theta+d^2)^{-1}$.
The two partial fractions of the same primitive \eqref{rz:rayresolvent} give
\begin{align}
 \int_0^\infty t^\beta w_t(d)\,dt
 &=\frac{K_\beta(a_+)-K_\beta(a_-)}{a_--a_+}
   =C_\beta d^{\beta-1},\label{rz:rayfirstaudit}\\
 \int_0^\infty t^{\beta+1}[w_t(d)-t^{-2}]\,dt
 &=\frac{a_-K_\beta(a_-)-a_+K_\beta(a_+)}{a_--a_+}
   =-d_\beta d^\beta.\label{rz:raysecondaudit}
\end{align}
For the second identity the pointwise compensated partial fraction is
$tw_t(d)-t^{-1}
=[a_-((t+a_-)^{-1}-t^{-1})-a_+((t+a_+)^{-1}-t^{-1})]/(a_--a_+)$.
The real part of \eqref{rz:raylogprimitive} at $a=de^{-i\theta}$ gives
\begin{equation}\label{rz:raylogaudit}
 \int_0^\infty t^\beta
 [\log|1+d/(te^{i\theta})|-\Rea\{d/(d+te^{i\theta})\}]\,dt
 =\frac{\beta b_\beta}{1+\beta}\cos((1+\beta)\theta)d^{1+\beta}=0.
\end{equation}
For $d\in[m,L]$, $0<m\le L<\infty$, these integrands are uniformly
$O(t^{\beta-2})$ at infinity, and respectively $O(t^\beta)$,
$O(t^{\beta-1})$, and $O(t^\beta(1+|\log t|))$ at zero.
Thus bounded spectral calculus and trace integration apply to all three
identities while their compensations remain intact.

Put $S=\tau D^{1+\beta}$, $A=\tau(D^\beta N)$, and
$v=\tau(N^*D^{\beta-1}N)$.
Since $\tau N=0$, pointwise
$\tau(W_tN)=\tau((W_t-t^{-2}\Id)N)$.
The operator-valued integrand $t^{\beta+1}(W_t-t^{-2}\Id)$ is
norm integrable: it is $O(t^{\beta-1})$ at zero and
$O(t^{\beta-2})$ at infinity. Equation~\eqref{rz:raysecondaudit}
therefore integrates the linear term with its compensation intact.
The other two identities integrate $F_t$ and the quadratic term in
\eqref{rz:covariance_entropy}, giving the exact identity
\begin{equation}\label{rz:covariance_quadratic}
 \mathcal R(z):=\int_0^\infty t^\beta[E_t(z)+J_t(z)]\,dt
 =\frac{C_\beta}{2}(S+v|z|^2)
       -d_\beta\Rea(e^{-i\theta}zA)\ge0.
\end{equation}
The compensated expansion proves absolute integrability of the sum
for each fixed $z$; its nonnegative summands are consequently integrable
individually. No separate integration of their divergent pieces is used.
If $v=0$, faithfulness and invertibility of $D$ give $N=0$ and $A=0$,
so \eqref{rz:weightedineq} is immediate. If $v>0$, set
$z_* =d_\beta e^{i\theta}\overline A/(C_\beta v)$. Then
\begin{equation}\label{rz:covariance_certificate}
 \sin^2(\beta\theta)Sv-|A|^2
 =\frac{2C_\beta v}{d_\beta^2}
     \int_0^\infty t^\beta[E_t(z_*)+J_t(z_*)]\,dt\ge0.
\end{equation}
Here $C_\beta/d_\beta=\sin(\beta\theta)$ since
$(1+\beta)\theta=\pi/2$. This proves the bound and exact residual;
the upper proof does not use the later optimality argument.
\end{proof}

\begin{corollary}\label{rz:weighted}
For bounded $N\in\B_0$ or $N\in\B_0^*$ and bounded invertible $D>0$,
the covariance bound \eqref{rz:weightedineq} and, when $v>0$, its
certificate \eqref{rz:covariance_certificate} hold for $0<\beta<1$.
\end{corollary}
\begin{proof}
Multiplicativity gives $\Phi(\Id+\zeta N)=\Id$; determinant Jensen in
$\B$ or its adjoint algebra gives $\operatorname{Det}_\tau(\Id+\zeta N)\ge1$.
Apply Theorem~\ref{rz:detcovariance}.
\end{proof}

\begin{corollary}\label{rz:nilpotent_covariance}
Every finite nilpotent matrix satisfies Theorem~\ref{rz:detcovariance},
with normalized matrix trace, for every positive invertible $D$.
\end{corollary}
\begin{proof}
Nilpotence gives
$\operatorname{Det}_{\operatorname{Tr}/n}(\Id+\zeta N)
=|\det(\Id+\zeta N)|^{1/n}=1$.
\end{proof}

\section{Analytic witnesses and the sharp Riesz theorem}\label{rz:sec:witnesses}
We construct analytic test elements whose pairing with the retained
summand is exact. For $2<p<4$, factorization converts determinant
covariance into an $L^2$ angle estimate. For $p>4$, the positive
canonical estimate supplies the test element directly. The analytic
regularization behind that estimate is proved in
Appendix~\ref{rz:analyticbackground} without sharp boundedness.
Continuity and duality complete the retained projection; a
subdiagonal corner construction then gives the strict projection.

\begin{lemma}\label{rz:analytic_testing}
Let $1<p<\infty$, $U\in\B$, and $0\ne W\in H^{p'}$.
For every $Y\in\B_0^*$,
\begin{equation}\label{rz:analytic_test_bound}
 \|U+Y\|_p\ge\frac{\Rea\tau(U^*W)}{\|W\|_{p'}}.
\end{equation}
\end{lemma}
\begin{proof}
Approximate $W$ in $L^{p'}$ by bounded analytic elements.
Multiplicativity of $\Phi$ and trace preservation give
$\tau(Y^*W)=0$, and H\"older's inequality gives the assertion.
This uses no boundedness of the $L^p$ projection.
\end{proof}

\subsection{The positive canonical Hardy estimate}\label{rz:hardy}
Let $P_2$ be the $L^2$ orthogonal projection onto $H^2$, and set $\mathcal T_{\B}=2P_2-\Phi$. Appendix~\ref{rz:prelim} proves that for bounded $g\geq0$,
\[
 h=\mathcal T_{\B}g\in H^2,\qquad\Rea h=g,\qquad\Phi h=\Phi g\geq0.
\]
Appendix~\ref{rz:analyticbackground} proves the bounded analytic regularization used below, independently of the sharp theorem.

\begin{corollary}\label{rz:canonical}
For $1<s<2$ and bounded $g\geq0$,
\begin{equation}\label{rz:canonicalbound}
 \|\mathcal T_{\B}g\|_s\leq\sec\frac\pi{2s}\,\|g\|_s.
\end{equation}
\end{corollary}
\begin{proof}
Let $b\in\B$ be bounded and accretive with $\Phi b\ge0$.
For $u>0$ and $0\le t\le1$, $C=\Id+tb/u$ has $\Rea C\ge\Id$;
$\|\Id-rC\|^2\le1-2r+r^2\|C\|^2<1$ for sufficiently small $r>0$.
A rescaled Neumann series gives $C^{-1}\in\B$, and multiplicativity gives
$\Phi(C^{-1})=(\Id+t\Phi b/u)^{-1}$.
Equation~\eqref{rz:logresolvent} therefore intertwines the compensated
logarithmic kernels under $\Phi$. By \eqref{rz:kernelcompensation}, their
norm-convergent Mellin integrals commute with $\Phi$, including when
$\Phi b$ has a kernel. Thus \eqref{rz:boundaryMellin} gives
$\Phi(b^s)=(\Phi b)^s$ and $\tau b^s=\tau(\Phi b)^s\ge0$.
Corollary~\ref{rz:budget} yields
$\|b\|_s\le\sec(\pi/(2s))\|\Rea b\|_s$.
For $h=\mathcal T_{\B}g$, apply this bound to the bounded analytic
Yosida approximants $b_\varepsilon=h(\Id+\varepsilon h)^{-1}$ from
Lemma~\ref{rz:selfcontained_approx}. Their real parts and expectations
are nonnegative, and convergence in $L^2$, hence in $L^s$, proves
\eqref{rz:canonicalbound}. Their analytic membership is established
independently in Appendix~\ref{rz:analyticbackground}.
\end{proof}

\subsection{The middle exponent interval}\label{rz:middle}
Corollary~\ref{rz:weighted} and analytic factorization supply this interval
without any previously established sharp projection estimate.

\begin{corollary}\label{rz:middleprojection}
For $2\le p\le4$, $U\in H^p$ and $Y\in(H_0^p)^*$,
\[
 \|U+Y\|_p\ge\sin(\pi/p)\,\|U\|_p.
\]
Consequently the retained-diagonal bound \eqref{rz:Pbound} holds.
\end{corollary}
\begin{proof}
Fix $2<p<4$ and put $\beta=(p-2)/2$ and $q=2p/(p-2)$.
First let $U\in\B$ be invertible in $\M$. Set $C=U^*U$ and
$S=\tau C^{p/2}$. Theorem~\ref{rz:standard}(1) gives
$R,R^{-1}\in\B$ with $R^*R=C^\beta$. Put
\[
 X=UR^*,\qquad M=UR^{-1},\qquad D=(RR^*)^{1/\beta}.
\]
For $Z\in\B_0^*$, $N=M^*Z\in\B_0^*$ because $\B_0$ is an ideal.
The polar decomposition $R=VC^{\beta/2}$ gives
\[
 D^\beta N=X^*Z,\qquad
 \tau D^{1+\beta}=S=\|X\|_2^2,\qquad
 MD^{\beta-1}M^*=UC^{-1}U^*=\Id.
\]
Corollary~\ref{rz:weighted} therefore gives
$|\tau(X^*Z)|^2\le\cos^2(\pi/p)S\|Z\|_2^2$.
By density in $(H_0^2)^*$ and orthogonality in \eqref{rz:orthodecomp},
\begin{equation}\label{rz:gate}
 h:=\|P_2X\|_2\ge\sin(\pi/p)\sqrt S>0.
\end{equation}
The bounded $\B$ module property in Appendix~\ref{rz:prelim} gives
$W=(P_2X)R\in H^2\subset H^{p'}$, and $W\ne0$ since $R$ is invertible.
H\"older and $\beta q/2=p/2$ give
$\|W\|_{p'}\le h\|R\|_q=hS^{1/q}$.
Cyclicity gives the exact pairing
$\Rea\tau(U^*W)=\tau(X^*P_2X)=h^2$.
Lemma~\ref{rz:analytic_testing} thus yields, for every $Y\in\B_0^*$,
\[
 \|U+Y\|_p\ge hS^{-1/q}
 \ge\sin(\pi/p)S^{1/2-1/q}
 =\sin(\pi/p)\|U\|_p.
\]
For an arbitrary bounded $U\in\B$, the matrices
\[
 U_\varepsilon=\begin{pmatrix}U&\varepsilon\Id\\\varepsilon\Id&0\end{pmatrix},
 \qquad
 U_\varepsilon^{-1}=\begin{pmatrix}0&\varepsilon^{-1}\Id\\
 \varepsilon^{-1}\Id&-\varepsilon^{-2}U\end{pmatrix}
\]
belong to $M_2(\B)$. Apply the invertible case in the finite subdiagonal
algebra $M_2(\B)\subset M_2(\M)$, with entrywise $\Phi$, normalized trace and
$Y_\varepsilon=\operatorname{diag}(Y,0)\in M_2(\B_0^*)$.
As $\varepsilon\downarrow0$, operator-norm continuity gives the desired
bound for $U,Y$; the common corner factor $2^{-1/p}$ cancels.
The matrix step depends only on the invertible-case proof.
For bounded $U,Y$, continuity of their norms in $p$ gives $p=4$;
$p=2$ follows from \eqref{rz:orthodecomp}. Finally approximate the
summands separately by $\B$ and $\B_0^*$ in $L^p$ to obtain the stated
analytic distance and hence \eqref{rz:Pbound}.
\end{proof}

\subsection{Completion of the all-exponent projection theorem}\label{rz:completion}
\subsubsection{Exponents above four}
\begin{proof}[Proof of Theorem~\ref{rz:allp} for $p>4$]
For bounded $U\in\B$, $U\ne0$, set
\[
 s=\frac p{p-2}\in(1,2),\qquad
 S=\|U\|_p^p,\qquad g=(UU^*)^{(p-2)/2},\quad h=\mathcal T_{\B}g,\quad W=hU.
\]
Here $g$ is bounded and positive. Since $1/p'=1/s+1/p$, Corollary~\ref{rz:canonical} and H\"older give
\begin{equation}\label{rz:witnessnorm}
 \|W\|_{p'}\leq\sec\frac\pi{2s}\,\|U\|_p^{p-1}.
\end{equation}
Since $h\in H^2$, the bounded $\B$ module property in
Appendix~\ref{rz:prelim} and $H^2\subset H^{p'}$ give $W\in H^{p'}$.
On the other hand, $1/s+2/p=1$ and trace cyclicity give the exact pairing
\begin{equation}\label{rz:witnesspair}
 \Rea\tau U^*W=\tau\bigl((UU^*)\Rea h\bigr)=\tau\bigl((UU^*)g\bigr)=S.
\end{equation}
Lemma~\ref{rz:analytic_testing} consequently gives
\[
 \|U+Y\|_p\geq\frac S{\|W\|_{p'}}
 \geq\cos\frac\pi{2s}\,\|U\|_p
 =\sin\frac\pi p\,\|U\|_p.
\]
Core density extends this estimate to $L^p$.
\end{proof}

\subsubsection{Dual exponents}
For $1<p<2$ let $q=p'>2$. The preceding arguments bound $P$ on $L^q$. On $\mathcal C$ the trace adjoint of $P$ is itself by \eqref{rz:orthodecomp}. Its Banach adjoint therefore defines the bounded extension on $L^p$, agrees with $P$ on the core, and has norm at most $\csc(\pi/q)=\csc(\pi/p)$. This proves the retained-diagonal assertion of Theorem~\ref{rz:allp}.

\subsubsection*{The strict projection and uniform sharpness}
In $M_2(\M)$, with normalized matrix trace, consider
\[
 \widetilde\B=
 \begin{pmatrix}\B&\B_0\\\B&\B\end{pmatrix},
 \qquad
 \widetilde\Phi\begin{pmatrix}x&y\\z&w\end{pmatrix}
 =\begin{pmatrix}\Phi x&0\\0&\Phi w\end{pmatrix}.
\]
It is a weak-star closed unital subdiagonal algebra: $\B_0$ is an ideal of $\B$, the diagonal intersection is $\mathcal D\oplus\mathcal D$, and $\widetilde\Phi$ is multiplicative on it. Weak-star density and trace preservation follow entrywise. Exel's finite-trace maximality theorem \cite{rz:Exel} makes it maximal. Its retained-diagonal projection on the upper-right corner is $P_0$. Apply the already universal bound to that corner; the factor $2^{-1/p}$ in both corner norms cancels. This proves the strict bound. For bounded $U_0\in\B_0$ and
$Y_0\in\B^*$, it gives
$\|U_0+Y_0\|_p\ge\sin(\pi/p)\|U_0\|_p$, since
$P_0(U_0+Y_0)=U_0$. Approximate the two summands separately by
$\B$ and $\B_0^*$ for the retained distance, and by $\B_0$ and
$\B^*$ for the strict distance. Continuity of the $L^p$ norm extends
both estimates to $H^p+(H_0^p)^*$ and $H_0^p+(H^p)^*$,
respectively.

\subsubsection*{An explicit sharpness construction}
Take $\M=L^\infty(\mathbb T)$ with normalized Haar trace and $\B=H^\infty(\mathbb T)$. For $0<b<1/p$ put
\[
 h_b(z)=\left(\frac{1+z}{1-z}\right)^b,\qquad c_b=\cos(\pi b).
\]
The principal power is analytic in the disk and $h_b(0)=1$. Its boundary values have modulus $|\cot(t/2)|^b$ and argument $\pm\pi b/2$. Hence, pointwise almost everywhere,
\[
 |h_b-c_b\overline{h_b}|=\sin(\pi b)|h_b|.
\]
Set $F_b=h_b-c_b(\overline{h_b}-1)$. For fixed $b$, use first the bounded analytic functions $h_b(\rho z)$, $\rho<1$, and their corresponding core decompositions. As $\rho\uparrow1$ they converge in $L^p$ to the displayed functions: a bound by $C|t|^{-b}$ near zero supplies dominated convergence since $bp<1$. Thus $PF_b=h_b$. The reverse triangle inequality gives
\[
 \left|\frac{\|F_b\|_p}{\|h_b\|_p}-\sin(\pi b)\right|
 \leq\frac{|c_b|}{\|h_b\|_p}.
\]
As $b\uparrow1/p$, $\|h_b\|_p\to\infty$, since $|\cot(t/2)|\sim2/|t|$ near zero. The ratio therefore tends to $\sin(\pi/p)$, proving $\|P\|\geq\csc(\pi/p)$. On the circle, $P_0=M_zPM_{\bar z}$, where both multiplication operators are isometries; hence $\|P_0\|=\|P\|$. This proves both uniform sharpness assertions and completes Theorem~\ref{rz:allp}.

\begin{proof}[Optimality of the coefficient in Theorem~\ref{rz:detcovariance}]
Fix $0<\beta<1$ and put $p=2(1+\beta)$, so that the squared coefficient
in \eqref{rz:weightedineq} is $\cos^2(\pi/p)$.
If it could be replaced universally by $c<\cos^2(\pi/p)$, necessarily
$c\ge0$ by taking $D=\Id$ and a nonzero nilpotent $N$.
The invertible-case proof of Corollary~\ref{rz:middleprojection} would
then give $h^2\ge(1-c)S$ in \eqref{rz:gate}. Its exact pairing, matrix
argument and separate closure passage would consequently give
$\|U+Y\|_p\ge\sqrt{1-c}\,\|U\|_p$ for all the analytic pairs there.
Since $\sqrt{1-c}>\sin(\pi/p)$, this contradicts the circle sequence
$U=h_b$, $Y=-c_b(\overline{h_b}-1)$ constructed above, whose norm
ratio tends to $\sin(\pi/p)$. Thus the coefficient is optimal.
\end{proof}

\section{Weighted angular realization and cone duality}\label{sec:angular}
Tracial joint spectral measures need not have finite total mass.
Indeed, for a nonzero nilpotent matrix $N$, one has $A=\Rea N\ne0$, and
\[
 \int|x|^q\dd\mu_{A,B}
 =\frac{\tr|A|^q}{q(q+1)}\longrightarrow\infty
 \quad(q\downarrow0).
\]
The compact support shows that a finite total mass would contradict
this divergence. We therefore use finite weighted pushforward measures on the circle.
Throughout this section, $C(\T)$ denotes the real-valued continuous
functions on $\T$.

\begin{definition}
For a nilpotent $N=A+iB$, define the finite positive measure $\nu_N$
on $\T$ by
\begin{equation}\label{eq:nu}
 \int_\T g\dd\nu_N
 =p(p+1)\int_{\C\setminus\{0\}}
       |z|^p g(\arg z)\dd\mu_{A,B}(z),\qquad g\in C(\T).
\end{equation}
For a polynomial $q$ with $q(0)=0$, define
\begin{equation}\label{eq:sigma}
 \int_\T g\dd\sigma_q
 =\int_0^{2\pi}|q(e^{it})|^p g(\arg q(e^{it}))\frac{\dd t}{2\pi},
\end{equation}
with integrand zero when $q(e^{it})=0$. Put
\[
 \mathscr S_p=\{u\in C(\T):\ \Lp u:=u''+p^2u\le0
                      \text{ distributionally}\},
\]
and define its polar cone restricted to finite positive measures,
using the nonpositive-pairing convention,
\begin{equation}\label{eq:J}
 \mathscr J_p=\{\nu\ge0:\ \int u\dd\nu\le0
                                  \text{ for every }u\in\mathscr S_p\}.
\end{equation}
\end{definition}

\begin{theorem}\label{thm:density}
In the weak-* topology of finite measures paired with $C(\T)$,
\begin{equation}\label{eq:coneequality}
 \overline{\operatorname{cone}\{\nu_N:N\text{ finite nilpotent}\}}^{\,w^*}
 =\overline{\operatorname{cone}\{\sigma_q:q(0)=0\}}^{\,w^*}
=\mathscr J_p.
\end{equation}
The family $\{\nu_N\}$ is itself a convex cone, so its conic hull
can be omitted. Every measure in $\mathscr J_p$ is consequently a
weak-* limit of a sequence of weighted angular measures of single
finite nilpotent matrices.
For each fixed polynomial $q$ with $q(0)=0$,
\begin{equation}\label{eq:measurelimit}
 \frac1n\nu_{q(S_n)}\longrightarrow\sigma_q
 \qquad\text{weak-*}.
\end{equation}
Equivalently,
\begin{equation}\label{eq:positivepotentialcone}
 \mathscr J_p=\{\nu\ge0:\ \nu=\Lp\omega,\quad
                      \omega\in C(\T),\quad\omega\ge0\}.
\end{equation}
Every such potential is Lipschitz. If $m=\nu(\T)$, then
\begin{equation}\label{eq:potentialbounds}
 \int_\T\omega\dd\theta=\frac m{p^2},\qquad
 \|\omega'\|_\infty\le m,\qquad
 \|\omega\|_\infty\le\frac m{2\pi p^2}+\pi m.
\end{equation}
Moreover, $(\omega,\nu)$ is approximated by nonnegative trigonometric
polynomials $(\omega_k,v_k)$ satisfying $v_k=\Lp\omega_k$,
$\int v_k\dd\theta=m$, $\omega_k\to\omega$ uniformly, and
$v_k\dd\theta\to\nu$ weak-*.
The realization uses weighted angular measures and allows growing
matrix dimensions.
\end{theorem}

\begin{proof}
We first prove the approximation assertion for continuous angular tests.
Put $N_n=q(S_n)=A_n+iB_n$, $\mu_n=n^{-1}\mu_{A_n,B_n}$, and
$\eta_n=|z|^2\mu_n$. Their supports lie in a common compact square.
Define the finite measure
\begin{equation}\label{eq:radiallift}
 \eta=\int_0^{2\pi}\int_0^1r(1-r)|q(e^{it})|^2
             \delta_{rq(e^{it})}\dd r\,\frac{\dd t}{2\pi}.
\end{equation}
For every polynomial $P$, the homogeneous parts of $|z|^2P$ are
spanned by powers of real linear forms. Equation~\eqref{eq:polymoment}
and Lemma~\ref{lem:commontoeplitz} give $\int P\dd\eta_n\to\int P\dd\eta$.
Uniform polynomial approximation therefore gives $\eta_n\to\eta$ weakly.
Choose $0<\alpha<\min\{p,2\}$. The ridge moments and the uniform
operator norm bound give $\sup_n\int|z|^{\alpha-2}\dd\eta_n<\infty$, while
\[
 \int|z|^{\alpha-2}\dd\eta
 =\frac1{\alpha(\alpha+1)}
       \int_0^{2\pi}|q(e^{it})|^\alpha\frac{\dd t}{2\pi}<\infty.
\]
For both measures, the origin tail of $|z|^{p-2}g(\arg z)$ is thus
$O(\|g\|_\infty\delta^{p-\alpha})$.
Take a continuous radial cutoff $0\le\chi_\delta\le1$, zero on
$[0,\delta]$ and one on $[2\delta,\infty)$.
Test weak convergence with $\chi_\delta(|z|)|z|^{p-2}g(\arg z)$ and
then remove the cutoff. The radial integral
$\int_0^1(1-r)r^{p-1}\dd r=1/[p(p+1)]$ proves \eqref{eq:measurelimit}.

Let $\mathcal C_N,\mathcal C_P$ be the closed cones in
\eqref{eq:coneequality}, and let $\mathcal V_p$ denote the
cone of measures with nonnegative potentials on the right of \eqref{eq:positivepotentialcone}.
We prove the inclusions
\[
 \mathcal C_P\subset\mathcal C_N\subset\mathcal V_p
 \subset\mathscr J_p\subset\mathcal C_P.
\]
The first inclusion is \eqref{eq:measurelimit}. For a finite
nilpotent define
\begin{equation}\label{eq:W}
 \omega_N^{(p)}(\theta)=\frac1{2\pi}\int_0^\infty
                  \Psi_N(re^{i\theta})r^{p-1}\dd r\ge0.
\end{equation}
It is integrable by compact support and the logarithmic bound at zero.
Testing \eqref{eq:homidentity} with smooth angular profiles gives
\begin{equation}\label{eq:angularpotential}
 \nu_N=\Lp\omega_N^{(p)}.
\end{equation}
Such homogeneous tests have the required local Sobolev regularity
by choosing $q>1$ close enough to one.

If $\omega\in L^1(\T)$ is nonnegative and $\Lp\omega=\nu\ge0$,
testing against $1$ gives $p^2\int\omega\dd\theta=m=\nu(\T)$.
Hence $\omega''=\nu-p^2\omega\dd\theta$ has total variation at most
$2m$. The derivative $\omega'$ has a periodic BV representative of
mean zero. Its essential oscillation is at most half its total
variation: the two arcs joining any two Lebesgue points each
contribute at least their difference. Thus
$\|\omega'\|_\infty\le\operatorname{osc}(\omega')\le m$.
Its primitive is Lipschitz, and the shorter circular arc gives
$\operatorname{osc}(\omega)\le\pi m$.
Comparison with the mean gives the last bound in
\eqref{eq:potentialbounds}. In particular \eqref{eq:W} has a
nonnegative continuous representative, and every generator belongs
to $\mathcal V_p$.
These bounds also show that $\mathcal V_p$ is weak-* closed.
Indeed, let a net $\nu_\gamma\in\mathcal V_p$ converge weak-* to
$\nu$, and choose nonnegative continuous potentials $\omega_\gamma$
with $\Lp\omega_\gamma=\nu_\gamma$. The limit $\nu$ is positive.
Testing against $1$ shows that, after restricting to a tail, the
masses are uniformly bounded. The estimates
\eqref{eq:potentialbounds} then make the potentials uniformly
bounded and equi-Lipschitz. By Arzel\`a--Ascoli, a subnet converges
uniformly to a nonnegative continuous function $\omega$.
For every $\varphi\in C^\infty(\T)$, passage to the limit gives
\[
 \int_\T\varphi\,\dd\nu
 =\int_\T\omega\,\Lp\varphi\,\dd\theta.
\]
Thus $\nu=\Lp\omega\in\mathcal V_p$. Since $\mathcal V_p$ is a
convex cone containing every generator, this proves
$\mathcal C_N\subset\mathcal V_p$.

For $\nu=\Lp\omega\in\mathcal V_p$ and $u\in\mathscr S_p$, smooth
angular convolutions satisfy
\[
 \int u\dd\nu
 =\lim_{\varepsilon\downarrow0}
       \int\omega\Lp u_\varepsilon\dd\theta\le0.
\]
Thus $\mathcal V_p\subset\mathscr J_p$.

For the reverse inclusion, suppose a finite positive measure
$\nu\in\mathscr J_p$ lies outside $\mathcal C_P$. Hahn--Banach
separation for the weak-* closed convex cone gives $g\in C(\T)$
with
\[
 \int g\dd\sigma_q\le0\quad\text{for all }q(0)=0,
 \qquad \int g\dd\nu>0.
\]
Set $G(re^{i\theta})=r^pg(\theta)$ and $G(0)=0$.
Lemma~\ref{lem:envelope} gives a continuous homogeneous
superharmonic majorant with angular part $u\in\mathscr S_p$.
Since $g\le u$ and $\nu\ge0$, one obtains
$\int g\dd\nu\le\int u\dd\nu\le0$, a contradiction.
Thus $\mathscr J_p\subset\mathcal C_P$, which proves both descriptions of the cone.

For the final realization assertion, the tracial measures are additive
under direct sums and transform by pushforward under positive scalar
dilation, by uniqueness in Hein\"avaara's representation
\cite{Hein}. It follows that
\[
 \nu_{N\oplus M}=\nu_N+\nu_M,\qquad
 \nu_{rN}=r^p\nu_N\quad(r\ge0).
\]
Thus the first generating family is a convex cone. Explicitly, for
a finite combination $\sum_{\ell=1}^L t_\ell\sigma_{q_\ell}$,
$t_\ell\ge0$, one may use the single nilpotent matrix
\[
 M_n=\bigoplus_{\ell=1}^L
       (t_\ell/n)^{1/p}q_\ell(S_n).
\]
Its angular measure converges to that combination by
\eqref{eq:measurelimit}. Weak-* topology on mass-bounded positive
measures on the circle is metrizable. Approximating an arbitrary
member of the closed cone and using a diagonal argument gives
the required sequence; mass bounds are obtained by including the test $g=1$.

Finally let $K_k$ be the nonnegative Fej\'er kernels normalized by
$\int K_k\dd\theta=1$. Put $\omega_k=K_k*\omega$ and
$v_k=K_k*\nu$. Positivity, the mass identity, and
$v_k=\Lp\omega_k$ are preserved; the asserted convergences are the
approximate-identity properties. At an integer exponent $p\ge2$, every $\nu\in\mathscr J_p$
satisfies $\int_\T e^{ip\theta}\,\dd\nu(\theta)=0$, since both
signs of $\cos(p\theta)$ and $\sin(p\theta)$ belong to $\mathscr S_p$.
Any two potentials for the same measure differ by
$a\cos(p\theta)+b\sin(p\theta)$, with $a,b\in\R$.
The bounds above hold for every nonnegative potential, so no
normalization of this kernel component is needed in the
compactness argument.
\end{proof}

\section{Sharp consequences}\label{sec:joint}
We first pass the finite trace certificate to compact operators and
obtain the exact comparison for a prescribed spectral set, followed
by its real diagonal perturbation consequence. At exponent four,
the directional moment disk gives an explicit signed classification
and joint norm formula. We then realize the analytic-distance
theorem as sharp triangular and Fourier projection bounds.

\subsection{Compact spectral corrections}\label{sec:consequences}
Compressing a quasinilpotent operator need not preserve quasinilpotence.
The spectral correction must be passed to the limit.

\begin{theorem}\label{thm:spectralmatching}
Let $1<p<\infty$ and $T_n\to T$ in $\Sp_p$, with all operators
compact. List their nonzero eigenvalues with algebraic multiplicities
and append countably many zeros. For two such lists define
\[
 d_p(\Lambda(S),\Lambda(T))
 =\inf_{\sigma}\left(\sum_{j\ge1}
       |\lambda_j(S)-\lambda_{\sigma(j)}(T)|^p\right)^{1/p},
\]
where $\sigma$ ranges over bijections of $\mathbb N$.
Then $d_p(\Lambda(T_n),\Lambda(T))\to0$.
\end{theorem}
\begin{proof}
For $S\in\Sp_p$, select finitely many generalized eigenspaces with
$0<|\lambda_j(S)|\le\delta$ and write
$b^p=\sum_j|\lambda_j(S)|^p$. In a Schur basis on their invariant
sum $E$, take $D=\operatorname{diag}(\overline\lambda_j|\lambda_j|^{p-2})$
and extend it by zero. Invariance gives
$\tr(DS)=\tr_E(D S|_E)=b^p$, while
$\|D\|\le\delta^{p-1}$ and $\|D\|_{p'}=b^{p-1}$.
For every finite-rank $F$, trace H\"older therefore gives
\[
 b^p\le a+b^{p-1}\varepsilon,
 \qquad a=\delta^{p-1}\|F\|_1,\quad\varepsilon=\|S-F\|_p.
\]
If $b>\varepsilon$, then
$(b-\varepsilon)^p\le b^{p-1}(b-\varepsilon)\le a$;
otherwise $b\le\varepsilon$. Exhausting the small spectrum yields
\begin{equation}\label{eq:smalltail}
 b_\delta(S):=
 \left(\sum_{0<|\lambda_j(S)|\le\delta}|\lambda_j(S)|^p\right)^{1/p}
 \le\|S-F\|_p+\delta^{1/p'}\|F\|_1^{1/p}.
\end{equation}
Taking $F=0$ and $\delta\ge\|S\|$ also proves
$\sum_j|\lambda_j(S)|^p\le\|S\|_p^p$.
Apply \eqref{eq:smalltail} to $S=T_n$, let $\delta\downarrow0$ with
$F$ fixed, and then approximate $T$ by finite-rank $F$. Thus
\begin{equation}\label{eq:uniformtail}
 \lim_{\delta\downarrow0}\limsup_n b_\delta(T_n)=0.
\end{equation}

Choose $|z|=\delta$ avoiding $\spec T$. Norm convergence and
Riesz projections \cite[Chapter~IV]{Kato} match the finite eigenvalue
multisets outside this circle, including algebraic multiplicities.
Let their cardinality be $m_\delta$ and maximum matching distance
$\eta_{n,\delta}\to0$, with $\eta_{n,\delta}=0$ if $m_\delta=0$.
Match the remaining padded lists arbitrarily. Then
\[
 d_p(\Lambda(T_n),\Lambda(T))^p
 \le m_\delta\eta_{n,\delta}^p
       +2^{p-1}\bigl(b_\delta(T_n)^p+b_\delta(T)^p\bigr).
\]
First let $n\to\infty$, then $\delta\downarrow0$ through admissible
circles. Both tails vanish, proving the theorem.
\end{proof}

\begin{corollary}\label{lem:spectralcontinuity}
Under Theorem~\ref{thm:spectralmatching}, every continuous $h$ on
a common spectral disk with $|h(z)|\le M|z|^p$ satisfies
\begin{equation}\label{eq:spectralcontinuity}
 \sum_jh(\lambda_j(T_n))\longrightarrow\sum_jh(\lambda_j(T)).
\end{equation}
The series converge absolutely. The convergence is uniform over any
family of continuous tests $\mathcal H$ with the same growth bound and equicontinuity
on each closed annulus away from zero.
\end{corollary}
\begin{proof}
Absolute convergence follows from the spectral $\ell^p$ bound.
For a common disk $|z|\le R$, let $\omega_{\mathcal H,\delta}$ be
a common modulus on $\delta\le|z|\le R$. The matching above gives
\begin{align}\label{eq:spectralErrorBudget}
 \sup_{h\in\mathcal H}
 \left|\sum_jh(\lambda_j(T_n))-\sum_jh(\lambda_j(T))\right|
 &\le m_\delta\omega_{\mathcal H,\delta}(\eta_{n,\delta})\notag\\
 &\quad+M\bigl(b_\delta(T_n)^p+b_\delta(T)^p\bigr).
\end{align}
Let $n\to\infty$, then $\delta\downarrow0$ and use
\eqref{eq:uniformtail}. A singleton family gives the first assertion.
\end{proof}
\begin{corollary}\label{cor:compactcertificate}
For every $T=A+iB\in\Sp_p$ and every superharmonic majorant $U\in\mathcal U_p$ with $U\ge F$,
\begin{equation}\label{eq:compactcertificate}
 \mathcal F_p(A,B)\le\sum_kU(\lambda_k(T)),
\end{equation}
with an absolutely convergent series. When $F$ satisfies
Theorem~\ref{thm:classification}, one may take its least majorant
$U_*$; it is pointwise optimal already among finite-matrix spectral
corrections by \eqref{eq:optimalcorrection}. In particular
$\mathcal F_p(A,B)\le0$ if $T$ is quasinilpotent.
\end{corollary}
\begin{proof}
Choose finite-rank orthogonal projections $P_N$ converging strongly
to $I$ on a separable reducing subspace containing the ranges of
$T$ and $T^*$. Finite-rank approximation gives
$P_NTP_N\to T$ in $\Sp_p$. Apply Corollary~\ref{thm:certificate} to
these compressions and Corollary~\ref{lem:spectralcontinuity} to $U$.
The Hermitian pencils converge in $\Sp_p$ as well.
\end{proof}
Schatten $p$-norm convergence is essential: viewed on $\ell^2$
by zero extension, $T_N=N^{-1/p}I_N$ tends to zero
in operator norm and has bounded $\Sp_p$ norm, but its eigenvalue
$p$-mass stays one.

Let $U_{p,a,b}$ be the least majorant of
$F_{p,a,b}(x,y)=|ax+by|^p-K_p(a,b)^p|x|^p$, supplied by the
classification and Corollary~\ref{thm:realpencil}.
Corollary~\ref{thm:certificate} and Corollary~\ref{cor:compactcertificate}
give, for finite matrices and compact $T\in\Sp_p$ respectively,
\begin{equation}\label{eq:obliquespectral}
 \tr|aA+bB|^p-K_p(a,b)^p\tr|A|^p
 \le\sum_kU_{p,a,b}(\lambda_k(T)).
\end{equation}
For finite matrices the exact difference is \eqref{eq:certificate}
with $F=F_{p,a,b}$. For quasinilpotent $T$ the sum vanishes, proving
the spectral form of Corollary~\ref{cor:MM}.

The classical coordinate superharmonic majorant $U_p$ recorded in
Appendix~\ref{sec:coordinate} gives the more symmetric specialization
\begin{equation}\label{eq:coordinate-remainder}
 \tr|A|^p-C_p^p\tr|B|^p\le\sum_kU_p(\lambda_k(T)).
\end{equation}
The scalar majorant is verified in Appendix~\ref{sec:coordinate};
its matrix and compact-operator consequences follow as in
Section~\ref{subsec:coordinatepath}. For nilpotent matrices the
spectral sum vanishes, whereas the two nonnegative integrals in
\eqref{eq:certificate} need not vanish.

\begin{proof}[Proof of Corollary~\ref{cor:MM}]
The classical qualitative theorem of Matsaev \cite{Matsaev61} gives
Schatten membership of the other self-adjoint part from either one;
see also \cite[Chapter~III]{GK70} and \cite{MS}.
The optimal-majorant assertion of Theorem~\ref{thm:unified}, with
one output and Corollary~\ref{thm:realpencil}, supplies a majorant of
$|ax+by|^p-K_p(a,b)^p|x|^p$ that vanishes at zero.
Corollary~\ref{cor:compactcertificate} applied to $T\in\Sp_p$
therefore gives \eqref{eq:obliqueMM}. For the second formulation
apply the first one to $iT=-B+iA$ and use
$K_p(0,1)=C_p$.  Sharpness follows from the finite-nilpotent supremum
in Corollary~\ref{thm:realpencil}.
\end{proof}

\begin{proof}[Proof of Corollary~\ref{cor:spectralsector}]
For $K\ge C_p$, Proposition~\ref{prop:coordinate} gives the admissible
majorant $U_{p,K}(x,y)=U_p(x,(K/C_p)y)$ of $|x|^p-K^p|y|^p$, nonpositive
on $\Sigma_{p,K}$. If $L_\Omega<\infty$, take $K=L_\Omega$ in
Corollary~\ref{cor:compactcertificate}; its absolutely convergent
spectral sum is nonpositive, proving \eqref{eq:sectorMatsaev}.
Scalar matrices give the lower bound $\sup_\Omega r$.
For the other lower bound, fix $z_0\in\Omega$ and a nonzero finite
nilpotent $N$. The matrices $z_0I+tN$ have spectrum $\{z_0\}$, and
as $t\to\infty$ their real-to-imaginary norm ratios tend to
$\|\Rea N\|_p/\|\Ima N\|_p$. Corollary~\ref{prop:coordinatepath},
applied also after rotation by $i$, makes these limits arbitrarily
close to $C_p$. This proves \eqref{eq:spectralsetconstant}, including
the absence of a finite bound when the scalar supremum is infinite.
The stated spectral-set criterion follows immediately.
\end{proof}

\subsection{Real diagonal perturbations}
\begin{corollary}\label{cor:diagonalload}
For every $\kappa\in\R$,
\begin{equation}\label{eq:diagonalload}
 \sup_{\mathcal E}\|H_{\mathcal E}+\kappa D_{\mathcal E}\|
 =\sup_{\mathcal E}\|H_{\mathcal E}+\kappa D_{\mathcal E}\|_{\rm cb}
 =\max\{C_p,|\kappa|\}.
\end{equation}
On $\Sp_p(\ell^2)$, the multiplier $\HSch+\kappa D$ has the same
operator norm and completely bounded norm, where $D$ is the
diagonal conditional expectation. More generally, for a finite
flag and real $\kappa_j$,
\begin{equation}\label{eq:variableload}
 \left\|H_{\mathcal E}X+\sum_j\kappa_jE_jXE_j\right\|_p
 \le\max\{C_p,\max_j|\kappa_j|\}\,\|X\|_p,
\end{equation}
with the same bound at every amplification level.
For every bounded real sequence $\boldsymbol\kappa=(\kappa_j)$,
define $D_{\boldsymbol\kappa}X=\sum_j\kappa_j e_{jj}Xe_{jj}$ on
$\Sp_p(\ell^2(\mathbb N))$. Then
\begin{equation}\label{eq:infinitediagonalload}
 \|\HSch+D_{\boldsymbol\kappa}\|
 =\|\HSch+D_{\boldsymbol\kappa}\|_{\rm cb}
 =\max\{C_p,\|\boldsymbol\kappa\|_\infty\}.
\end{equation}
\end{corollary}
\begin{proof}
Put $K=\max\{C_p,\max_j|\kappa_j|\}$. For $X=X^*$, set
$Y=H_{\mathcal E}X+\sum_j\kappa_jE_jXE_j$.
The block upper-triangular matrix $i(X+iY)=-Y+iX$ has diagonal
blocks $(-\kappa_j+i)E_jXE_j$, so its spectrum lies in $\Sigma_{p,K}$.
Corollary~\ref{cor:spectralsector} gives $\|Y\|_p\le K\|X\|_p$.
Hermitian dilation gives the same bound for arbitrary $X$, since
all $\kappa_j$ are real. Applying the argument to $(I_d\otimes E_j)$
proves \eqref{eq:variableload} at every amplification level.

For a bounded sequence $\boldsymbol\kappa$, the diagonal series
converges in $\Sp_p$ and contractivity of $D$ gives
$\|D_{\boldsymbol\kappa}X\|_p\le\|\boldsymbol\kappa\|_\infty\|X\|_p$.
Finite principal sections and density therefore give the upper
bound in \eqref{eq:infinitediagonalload} at every amplification level.
The tests $X=e_{jj}$ give the lower bound $\|\boldsymbol\kappa\|_\infty$.
The zero-diagonal matrices $X=\Rea q(S_n)$ from
Corollary~\ref{prop:coordinatepath} give the lower bound $C_p$.
These prove the infinite operator norm, and
\eqref{eq:commoncbamplification} gives the same complete norm.
For constant $\kappa$, the same two tests over finite flags prove
\eqref{eq:diagonalload}; the infinite assertion is the constant-sequence
case of \eqref{eq:infinitediagonalload}.
\end{proof}

This real diagonal perturbation preserves self-adjointness. The signed conventions
$S_0\pm D$ from Lemma~\ref{lem:diagonalnormalization} instead become
$H_{\mathcal E}\pm iD_{\mathcal E}$ after multiplication by $i$;
their normalization comparison is a separate assertion.

\subsection{Fourth-power moment geometry}
\begin{theorem}\label{thm:fourmomentdisk}
For every nonzero finite nilpotent $N=A+iB$, its directional
fourth moments have the form
\begin{equation}\label{eq:fourstate}
 \tr|A\cos\theta+B\sin\theta|^4
 =M(1+u\cos2\theta+v\sin2\theta),\qquad M>0.
\end{equation}
With $\rho=2\sqrt2/3$, the closure of all these normalized moment
pairs is precisely $\{(u,v):u^2+v^2\le\rho^2\}$.
For arbitrary real $a,b$, homogeneity gives
\begin{equation}\label{eq:fourdirection}
 \frac1M\tr|aA+bB|^4
 =(a^2+b^2)^2+(a^2+b^2)[(a^2-b^2)u+2abv].
\end{equation}
\end{theorem}
\begin{proof}
For nonzero nilpotent $N=A+iB$, write
$E_N(\theta)=\tr|A\cos\theta+B\sin\theta|^4$.
Expanding via $A\cos\theta+B\sin\theta
=(e^{-i\theta}N+e^{i\theta}N^*)/2$, the fourth harmonics vanish
because $\tr N^4=0$. Thus, for some $M>0$,
\[
 E_N(\theta)=M(1+u\cos2\theta+v\sin2\theta).
\]
Every directional part $A\cos\theta+B\sin\theta$ is nonzero:
otherwise a rotation of $N$ would be skew-Hermitian and nilpotent,
hence zero. Thus $r=(u^2+v^2)^{1/2}<1$.
Apply the sharp coordinate inequality to every rotation of $N$.
The maximum ratio of orthogonal directional fourth moments is
$(1+r)/(1-r)$, whence
\[
 r\le\frac{C_4^4-1}{C_4^4+1}=\frac{2\sqrt2}{3}=:\rho.
\]
Conversely sharpness of $C_4=1+\sqrt2$ among finite nilpotents
gives a sequence of moment pairs $(u,v)$ approaching the boundary
circle.
Rotating the nilpotents realizes every boundary direction in the
closure. Direct sums, with positive fourth-power rescaling, take
convex combinations of these moment pairs. Their closure
is therefore exactly the disk $u^2+v^2\le\rho^2$.

\end{proof}

\subsection{A complete fourth-power criterion}
The moment disk turns the signed classification at $p=4$ into an
explicit second-order cone condition.

\begin{corollary}
\label{cor:quarticcone}
For finitely many real $c_j,a_j,b_j$, put
\begin{align}
 \Gamma_0&=\sum_jc_j(a_j^2+b_j^2)^2,\notag\\
 \Gamma_1&=\sum_jc_j(a_j^2+b_j^2)(a_j^2-b_j^2),\notag\\
 \Gamma_2&=2\sum_jc_ja_jb_j(a_j^2+b_j^2).
 \label{eq:quarticcoefficients}
\end{align}
Then
\begin{equation}\label{eq:quartictrace}
 \sum_jc_j\tr|a_jA+b_jB|^4\le0
 \quad\text{for every finite nilpotent }A+iB
\end{equation}
if and only if
\begin{equation}\label{eq:quarticcone}
 \Gamma_0+\frac{2\sqrt2}{3}
                \sqrt{\Gamma_1^2+\Gamma_2^2}\le0.
\end{equation}
Equivalently, the coefficient triple belongs to a scaled negative
Lorentz cone. This criterion also characterizes the existence of
an admissible superharmonic majorant of
$\sum_jc_j|a_jx+b_jy|^4$, and the inequality then holds for every
compact quasinilpotent $A+iB\in\Sp_4$.
\end{corollary}
\begin{proof}
For $N=A+iB\ne0$, use the normalization \eqref{eq:fourstate}.
Theorem~\ref{thm:fourmomentdisk} gives the closed moment disk
$u^2+v^2\le8/9$. Summing the
directional fourth moments with their signed coefficients gives
\[
 \frac1M\sum_jc_j\tr|a_jA+b_jB|^4
 =\Gamma_0+\Gamma_1u+\Gamma_2v.
\]
The supremum over the closed disk is the left side of
\eqref{eq:quarticcone}. If it is positive, a nearby realized
moment pair gives a finite nilpotent counterexample; if it is
nonpositive, every pair satisfies \eqref{eq:quartictrace}.
The zero matrix is immediate. Theorem~\ref{thm:classification}
and Corollary~\ref{cor:compactcertificate} give the majorant and
compact-operator assertions.
\end{proof}
The coordinate value $C_4$ used in this derivation is classical
\cite[Theorem~6.2]{NR}. The corollary makes the entire signed
fourth-power classification explicit; it requires no additional
estimate beyond the moment disk.

\begin{corollary}\label{cor:quarticquotients}
Let $P,Q$ be finite signed fourth-power ridge sums, with coefficient
triples $(\Gamma_0,\Gamma_1,\Gamma_2)$ and $(Q_0,Q_1,Q_2)$ as in
\eqref{eq:quarticcoefficients}. Assume
$Q_0>\rho\sqrt{Q_1^2+Q_2^2}$, and put
\begin{align*}
 a&=Q_0^2-\rho^2(Q_1^2+Q_2^2),\\
 b&=Q_0\Gamma_0-\rho^2(Q_1\Gamma_1+Q_2\Gamma_2),\\
 c&=\Gamma_0^2-\rho^2(\Gamma_1^2+\Gamma_2^2).
\end{align*}
The optimal real coefficient in
$\mathcal P_4(A,B)\le\lambda\mathcal Q_4(A,B)$ over finite nilpotents is
\begin{equation}\label{eq:quarticquotientclosed}
 \lambda_* =\frac{b+\sqrt{b^2-ac}}a.
\end{equation}
Equivalently, it is the least real number satisfying
\begin{equation}\label{eq:quarticquotient}
 \lambda Q_0-\Gamma_0\ge\rho
 \sqrt{(\lambda Q_1-\Gamma_1)^2+(\lambda Q_2-\Gamma_2)^2}.
\end{equation}
The same coefficient bounds every compact quasinilpotent
$A+iB\in\Sp_4$. A nonzero $Q$ with nonnegative ridge coefficients
always satisfies the stated hypothesis.
\end{corollary}
\begin{proof}
The hypothesis is exactly positivity of the denominator on the
closed moment disk. Corollary~\ref{cor:quarticcone}, applied to
$P-\lambda Q$, gives \eqref{eq:quarticquotient}; its left side minus
its right side is strictly increasing from $-\infty$ to $+\infty$.
Squaring the boundary equation gives $a\lambda^2-2b\lambda+c=0$.
At $\lambda_0=\Gamma_0/Q_0$ this quadratic is nonpositive, so
$b^2-ac\ge0$ and its larger root is at least $\lambda_0$.
That root therefore satisfies the unsquared equation and is its
unique solution. In the proportional case the discriminant vanishes
and the formula gives $\lambda_* =\lambda_0$ directly.
The compact assertion follows from Corollary~\ref{cor:quarticcone};
nonnegative ridge coefficients give $Q_0>0$ and
$\sqrt{Q_1^2+Q_2^2}\le Q_0$, proving the final assertion.
\end{proof}

For example, $Q(x,y)=x^4+y^4-(x+y)^4/5$ has coefficient triple
$(6/5,0,-4/5)$ and satisfies the denominator condition.
Its trace is strictly positive on every nonzero finite nilpotent,
although $Q(1,1)=-6/5$.

\subsection{The full joint formula at \texorpdfstring{$p=4$}{p=4}}
\begin{corollary}\label{cor:jointfour}\label{prop:four}
For the real pairs and weights in Theorem~\ref{thm:unified}, put
\[
 A_0=\sum_jw_ja_j^2(a_j^2+b_j^2),\quad
 B_0=\sum_jw_jb_j^2(a_j^2+b_j^2),\quad
 C_0=\sum_jw_ja_jb_j(a_j^2+b_j^2).
\]
Then every joint model in that theorem at $p=4$ has norm
\begin{equation}\label{eq:jointfour}
 \Lambda_4(P)^4=A_0+17B_0+4\sqrt2\sqrt{9B_0^2+C_0^2}.
\end{equation}
Moreover
\begin{equation}\label{eq:compatibility}
 \Lambda_4(P)^4=\sum_jw_jK_4(a_j,b_j)^4
\end{equation}
if and only if all indices with $w_j>0$ and $b_j\ne0$ have the
same ratio $a_j/b_j$. Zero weights and pure identity outputs impose
no condition; the condition is vacuous with at most one such index.

For shifted pairs, let $W=\sum_jw_j$, $Z=\sum_jw_je^{2\pi i\alpha_j}$,
and $D=9W-8\Rea Z$. Then
\begin{equation}\label{eq:jointshiftfour}
 \Lambda_4(P)^4=D+\sqrt{D^2-9W^2+8|Z|^2}.
\end{equation}
In particular, two distinct shifts in $(0,1)$ with positive weights
give strict improvement over the sum of the separate fourth powers.
For one output this specializes the classical scalar formula
\cite{HKV} to
\begin{equation}\label{eq:Kfour}
 K_4(a,b)^2=K_{4/3}(a,b)^2
 =\sqrt{a^2+b^2}\left(\sqrt{a^2+9b^2}+2\sqrt2\,|b|\right).
\end{equation}
In particular, at $p=4$ and $p=4/3$ on both $\ell^p(\Z)$ and $\ell^p(\Z;\Sp_p)$,
\begin{equation}\label{eq:shiftfour}
 \|R_\alpha\|^2=\sqrt{1+8\sin^2(\pi\alpha)}
                          +2\sqrt2\sin(\pi\alpha).
\end{equation}
\end{corollary}
\begin{proof}
Apply Corollary~\ref{cor:quarticquotients} with $Q(x,y)=|x|^4$.
Then $(Q_0,Q_1,Q_2)=(1,1,0)$ and
$(\Gamma_0,\Gamma_1,\Gamma_2)=(A_0+B_0,A_0-B_0,2C_0)$, giving
\[
 a=\frac19,\qquad b=\frac{A_0+17B_0}{9},\qquad
 b^2-ac=\frac{32}{81}(9B_0^2+C_0^2).
\]
Formula~\eqref{eq:quarticquotientclosed} gives \eqref{eq:jointfour},
including $B_0=0$, and Theorem~\ref{thm:unified} transfers it to
every joint operator model.

For the equality criterion put
$B_j=w_jb_j^2(a_j^2+b_j^2)$ and
$C_j=w_ja_jb_j(a_j^2+b_j^2)$.
Applying the single-output case of the formula yields the exact gap
\begin{equation}\label{eq:jointgap}
 \sum_jw_jK_4(a_j,b_j)^4-\Lambda_4(P)^4
 =4\sqrt2\left(\sum_j\|(3B_j,C_j)\|_2
                    -\Bigl\|\sum_j(3B_j,C_j)\Bigr\|_2\right).
\end{equation}
Every nonzero vector has first coordinate positive. Equality in
the Euclidean triangle inequality holds exactly when these vectors
are positively collinear, equivalently when their ratios
$C_j/B_j=a_j/b_j$ agree. For shifted pairs,
$A_0=(W+\Rea Z)/2$, $B_0=(W-\Rea Z)/2$, $C_0=\Ima Z/2$,
which proves \eqref{eq:jointshiftfour}. Finally $\cot(\pi\alpha)$
is strictly decreasing on $(0,1)$, proving the strictness assertion.
The single-pencil formula factors as \eqref{eq:Kfour}; scalar
duality gives the $4/3$ value, and substitution gives
\eqref{eq:shiftfour}. This duality argument applies to a single
pencil. For several pencils, the adjoint maps an $\ell^{p'}$
direct sum into a single space, so it does not give the same joint
norm identity at $p'$.
\end{proof}

For example, two unit-weight shifts $\alpha_1=1/4$, $\alpha_2=3/4$
give
\[
 \Lambda_4(P)^4=18+12\sqrt2<18+8\sqrt5
   =\|R_{1/4}\|^4+\|R_{3/4}\|^4.
\]
The compatibility condition concerns common approximate extremizers.
Indeed, a sequence saturating the joint sum must saturate every
positive-weight summand; conversely proportional nonidentity pencils
share a single-pencil extremizing sequence.

Comparisons of arbitrary real directions and products of joint
families acting in independent coordinates are recorded in
Appendix~\ref{sec:furtherrealizations}.

\begin{remark}\label{rem:realboundary}
The planar majorant applies to real linear combinations $aA+bB$
of self-adjoint operators. The norm formula does not extend to
arbitrary complex coefficients $a,b$.
For example, on $\Sp_2$ let $X=e_{12}$. Then
$H_{\mathcal E}X=-iX$, so
$\|(I+iH_{\mathcal E})X\|_2=2\|X\|_2$.
Substituting $(a,b)=(1,i)$ in the absolute-value angular
ratio in \eqref{eq:Koblique} instead gives $\sqrt2$ at $p=2$.
This explicit failure rules out that complex extension.
In particular, the sharp real-combination results, including the
completely bounded equalities of Corollary~\ref{cor:HilbertCB}, do not
establish a sharp Riesz or triangular projection bound.
\end{remark}

\subsection{Riesz projections and matrix amplification}\label{rz:applicationssection}
We apply Theorem~\ref{rz:allp} to matrix and Fourier realizations.
The amplification bound is used at every matrix level before it is
converted to a canonical completely bounded estimate.

\begin{corollary}\label{rz:amplification}
For every algebra in Theorem~\ref{rz:allp}, every $1<p<\infty$, and every $n\geq1$,
\[
 \|\operatorname{Id}_{M_n}\otimes P\|_{L^p(M_n\overline\otimes\M)\to L^p(M_n\overline\otimes\M)},
 \quad\|\operatorname{Id}_{M_n}\otimes P_0\|_{L^p(M_n\overline\otimes\M)\to L^p(M_n\overline\otimes\M)}
 \leq\csc(\pi/p).
\]
For the canonical operator space structure on $L^p(\M)$, also
\[
 \|P\|_{\mathrm{cb}},\ \|P_0\|_{\mathrm{cb}}\leq\csc(\pi/p).
\]
These are uniform upper bounds; a particular subdiagonal algebra need not attain them.
\end{corollary}
\begin{proof}
The algebra $M_n(\B)$ is subdiagonal in $M_n(\M)$, with diagonal $M_n(\mathcal D)$ and expectation $\operatorname{Id}_{M_n}\otimes\Phi$. Multiplicativity, trace preservation, and weak-star density hold entrywise; Theorem~\ref{rz:standard}(3) gives maximality. Its retained and strict projections are the stated amplifications, so Theorem~\ref{rz:allp} proves the tracial bounds. The normalized matrix trace may be replaced by the ordinary trace, since this scales both norms by the same factor.

Lemma~\ref{lem:commonamplification} now converts these uniform tracial amplification bounds to the stated canonical completely bounded bounds. This application uses no additional sharp estimate and keeps the tracial and operator-space matrix norms distinct.
\end{proof}

\subsubsection{Triangular truncations}
For $n\geq1$, let $T_n((x_{ij}))=(1_{\{i\leq j\}}x_{ij})$ be retained-diagonal upper triangular truncation; $S_n^p$ denotes the Schatten $p$-class on $\mathbb C^n$.
\begin{corollary}\label{rz:triangularnorm}
Let $1<p<\infty$ and let $(\M,\tau)$ be any finite von Neumann algebra
with a faithful normal tracial state. Then
\begin{equation}\label{rz:triangularnormeq}
\begin{split}
 \sup_{n\geq1}\|T_n:S_n^p\longrightarrow S_n^p\|
 &=\csc(\pi/p),\\
 \sup_{n\geq1}\|T_n\otimes\operatorname{Id}_{\M}:
 L^p(M_n\,\overline\otimes\,\M)
 \longrightarrow L^p(M_n\,\overline\otimes\,\M)\|
 &=\csc(\pi/p).
\end{split}
\end{equation}
The tensor-product trace is $\operatorname{Tr}_n/n\otimes\tau$.
These equalities concern the supremum over matrix sizes; equality
is not asserted at a fixed size.
\end{corollary}

\begin{proof}[Proof of Corollary~\ref{rz:triangularnorm}]
The upper-triangular algebra over $\M$ is a subdiagonal algebra in
$M_n\,\overline\otimes\,\M$, with diagonal algebra
$\bigoplus_{j=1}^n\M$ and expectation that deletes the off-diagonal
entries. Its maximality follows from Theorem~\ref{rz:standard}(3),
and its Riesz projection is $T_n\otimes\operatorname{Id}_{\M}$.
Theorem~\ref{rz:allp} therefore gives both upper bounds.
The use of normalized rather than ordinary matrix trace does not
change a projection norm.

For the lower bound, we give a finite Toeplitz argument. Let
$f(z)=\sum_{k=-m}^{m}a_kz^k$ be a scalar trigonometric polynomial,
and, for $n>2m$, set
\[
 A_n(f)=(a_{j-i})_{1\leq i,j\leq n},\qquad
 \|A\|_{p,n}=\bigl(n^{-1}\operatorname{Tr}_n|A|^p\bigr)^{1/p}.
\]
If $P_+f=\sum_{k=0}^{m}a_kz^k$, then
$T_nA_n(f)=A_n(P_+f)$.
Apply the singular-value assertion of Lemma~\ref{lem:commontoeplitz} with $d=1$ to $f$ and $P_+f$. It gives
\[
\|A_n(f)\|_{p,n}\longrightarrow\|f\|_p,\qquad \|T_nA_n(f)\|_{p,n}\longrightarrow\|P_+f\|_p.
\]
For every nonzero $f$ we therefore have
\[
 \sup_n\|T_n\|\geq\frac{\|P_+f\|_p}{\|f\|_p}.
\]
Trigonometric polynomials are dense in $L^p(\mathbb T)$.
The boundedness and explicit scalar sharpness already proved above
show that the supremum of the last ratio is $\csc(\pi/p)$.
Finally, the embedding $A\mapsto A\otimes\Id$ is isometric for
the normalized $L^p$ norms and intertwines the two triangular
truncations. It supplies the same lower bound for every fixed
$(\M,\tau)$ and proves \eqref{rz:triangularnormeq}.
\end{proof}

\begin{remark}
This corollary is the exact infinite-size norm of triangular
truncation, including finite-algebra coefficients. The
Fourier--Schur transference relation proved by Neuwirth and Ricard
\cite[Proposition 6.3]{NR} places it in the corresponding completely
bounded scalar Fourier-multiplier framework; the proof here does
not require that transference result.
\end{remark}
\subsubsection{Fourier projections with finite-algebra coefficients}
Let $m$ be normalized Haar measure on $\mathbb T$, and let
$\mathcal N=L^\infty(\mathbb T)\,\overline\otimes\,\M$, with trace
$\widetilde\tau=m\otimes\tau$.
For an $\M$-valued trigonometric polynomial define
\[
 P_+^{\M}\left(\sum_{k\in\mathbb Z}z^k x_k\right)
 =\sum_{k\geq0}z^k x_k.
\]
Here the sum is finite and the coefficients lie in $\M$; these polynomials are dense in $L^p(\mathcal N)$.
\begin{corollary}\label{rz:Fourier_exact}
For every nonzero finite tracial von Neumann algebra $(\M,\tau)$ and
$1<p<\infty$,
\[
 \|P_+^{\M}\|_{L^p(\mathcal N)\to L^p(\mathcal N)}
 =\csc\frac\pi p.
\]
Equivalently, this is the norm of
$P_+\otimes\operatorname{Id}_{L^p(\M)}$ on the Bochner space
$L^p(\mathbb T;L^p(\M))$.
Moreover, for every $n\geq1$,
\begin{equation}\label{rz:Fourier_matrix_exact}
 \|\operatorname{Id}_{M_n}\otimes P_+^{\M}\|_
 {L^p(M_n\,\overline\otimes\,\mathcal N)\to
 L^p(M_n\,\overline\otimes\,\mathcal N)}
 =\csc\frac\pi p.
\end{equation}
With the canonical operator space structure on $L^p(\mathcal N)$,
\[
 \|P_+^{\M}\|_{\mathrm{cb}}=\csc\frac\pi p.
\]
The same equalities hold for the strict Fourier projection
$P_{>0}^{\M}$, which retains precisely the positive frequencies.
\end{corollary}

\begin{proof}[Proof of Corollary~\ref{rz:Fourier_exact}]
For $F\in\mathcal N$, its Fourier coefficient
$\widehat F(k)=(m\otimes\operatorname{Id}_{\M})(z^{-k}F)$ lies in $\M$.
Consider the weak-star closed algebra
\[
 \mathcal A=\{F\in\mathcal N:\widehat F(k)=0\text{ for }k<0\}.
\]
This is the weak-star closure of the analytic polynomials
$\sum_{k\geq0}z^k x_k$.
Indeed, the Fourier coefficient maps are normal, giving one inclusion;
for the other, the Fej\'er means of an element of $\mathcal A$ are
uniformly bounded analytic polynomials converging ultraweakly to it.
Separate ultraweak continuity of multiplication and these polynomial
approximations show that $\mathcal A$ is a unital algebra and that
\[
 \widehat{FG}(0)=\widehat F(0)\widehat G(0),\qquad F,G\in\mathcal A.
\]
To justify the last identity explicitly, first approximate $G$ while
holding an analytic polynomial $F$ fixed, then approximate $F$;
for polynomials no two nonnegative frequencies add to zero except
$0+0$.
Fourier uniqueness gives
$\mathcal A\cap\mathcal A^*=\Id\otimes\M$.
The map $\Psi F=\Id\otimes\widehat F(0)$ is a normal
$\widetilde\tau$-preserving conditional expectation onto this diagonal,
and is multiplicative on $\mathcal A$.
Finally, $\mathcal A+\mathcal A^*$ contains all $\M$-valued
trigonometric polynomials, hence is weak-star dense in $\mathcal N$.
Thus $\mathcal A$ is subdiagonal and is maximal by
Theorem~\ref{rz:standard}(3).
Its retained-diagonal projection agrees with $P_+^{\M}$ on
trigonometric polynomials. Theorem~\ref{rz:allp} gives the upper bound.
On the isometric scalar subspace $f\mapsto f\Id$ it is the scalar
Riesz projection, whose sharp lower bound was proved above.
Hence its norm is exactly $\csc(\pi/p)$.
The identity
\[
 \|F\|_{L^p(\mathcal N)}^p
 =\int_{\mathbb T}\|F(z)\|_{L^p(\M)}^p\,dm(z)
\]
first holds for simple functions and extends by completion, yielding
the stated Bochner-space interpretation.
Corollary~\ref{rz:amplification}, applied to
$\mathcal A\subset\mathcal N$, gives the upper bound at every
tracial matrix level and for the complete norm. The scalar embedding
$f\mapsto f\Id$ gives the matching lower bound at each level,
proving \eqref{rz:Fourier_matrix_exact} and the complete norm equality. Finally,
$P_{>0}^{\M}=M_zP_+^{\M}M_{\bar z}$.
Multiplication by $z$ or $\bar z$ is an isometry at every tracial matrix level and a complete isometry by \eqref{eq:commonLpMatrixNorm}. This proves all strict-projection assertions.
\end{proof}

\section{Limitations of the two-point criterion}\label{sec:boundary}
The scalar two-term pencil is governed by an angular quotient involving
two points separated by $\pi/p$. For a general homogeneous obstacle, consider the necessary
conditions
\begin{equation}\label{eq:pairtest}
 f(\theta)+f(\theta+\pi/p)\le0\qquad\text{for every }\theta.
\end{equation}
These tests are necessary: extend
$\sin(p(\theta-\alpha))$ on $[\alpha,\alpha+\pi/p]$ by zero
periodically. It is nonnegative and its $\Lp$ is
$p(\delta_\alpha+\delta_{\alpha+\pi/p})$; pairing with a
superharmonic angular majorant proves \eqref{eq:pairtest}.
They are not sufficient for general finite real ridge sums.
The two-point criterion of \cite[Theorem~3.7]{HKV} requires additional
sector hypotheses on the homogeneous obstacle. The example below
shows why such restrictions cannot be dropped in the present class.

\begin{proposition}
\label{prop:counter}
Put $p=8/3$ and
\begin{equation}\label{eq:eightF}
 F(x+iy)=\sum_{j=0}^7(-1)^j
 \left|x\cos\frac{j\pi}{8}+y\sin\frac{j\pi}{8}\right|^p.
\end{equation}
Its angular restriction $f$ satisfies
$f(\theta+3\pi/8)=-f(\theta)$, so every test
\eqref{eq:pairtest} is an equality. Nevertheless $F$ has no
continuous homogeneous superharmonic majorant, and its nilpotent
trace inequality fails. More explicitly, for all sufficiently small
$\varepsilon>0$, the polynomial
\[
 q_\varepsilon(z)=z-\varepsilon z^9
\]
has $\int F(q_\varepsilon(e^{it}))\dd t/(2\pi)>0$.
Thus $N_n=S_n-\varepsilon S_n^9$ is a counterexample for all
sufficiently large $n$.
\end{proposition}
\begin{proof}
Shifting the eight directions by $3\pi/8$ permutes them with a sign
change, using the $\pi$-periodicity of $|\cos|^p$, and
$\int f\dd\theta=0$. To construct the analytic counterexample, define
\[
 d_k=\int_0^{2\pi}|\cos t|^p\cos(2kt)\frac{\dd t}{2\pi}.
\]
Integration by parts on $[-\pi/2,\pi/2]$ gives
\[
 (p+2k)d_k=(p-2k+2)d_{k-1}.
\]
For $p=8/3$ this yields $d_4=d_0/364>0$. Therefore
\[
 \int_0^{2\pi}f(t)\cos(8t)\frac{\dd t}{2\pi}
 =8d_4=\frac{2d_0}{91}>0.
\]
Since $F$ is $C^1$ away from zero, differentiation under the compact
boundary integral is legitimate. At $\varepsilon=0$, the radial
variation of $q_\varepsilon(e^{it})$ is $-\cos8t$ and its angular
variation is $-\sin8t$. Hence
\begin{align*}
 \left.\frac{\dd}{\dd\varepsilon}\right|_{0}
 \int F(q_\varepsilon(e^{it}))\frac{\dd t}{2\pi}
 &=-\int\{pf(t)\cos8t+f'(t)\sin8t\}\frac{\dd t}{2\pi}\\
 &=(8-p)\int f(t)\cos8t\frac{\dd t}{2\pi}
 =\frac{32d_0}{273}>0.
\end{align*}
The integral is zero at $\varepsilon=0$, so it is positive for
small positive $\varepsilon$. A continuous homogeneous superharmonic
majorant $U$ would instead give $\int F(q_\varepsilon)\dd t/(2\pi)
\le U(0)=0$, a contradiction. Equation~\eqref{eq:ridgeLimit} supplies
the finite nilpotent counterexamples.
\end{proof}

\appendix
\section{An explicit superharmonic majorant}\label{sec:coordinate}
The construction uses the classical method of joining an obstacle
to homogeneous harmonic functions on sectors; see the superharmonic
proof of Ess\'en \cite{Essen} and the homogeneous-minorant treatment
in \cite[Sections~3--4]{HKV}. We record an explicit coordinate
majorant, including its nonpositivity on the cone used for real
diagonal perturbations. Put $F_p(x,y)=|x|^p-C_p^p|y|^p$.
Complex powers below use the branch continuous on the closed right
half-plane, normalized by $1^p=1$, with value zero at the origin.
For $p>2$, set $\alpha=\pi/(2p)$, $c=\cot\alpha=C_p$,
$L=\cos^{p-1}\alpha/\sin\alpha$, and define
\begin{equation}\label{eq:Upgreater}
 U_p(x,y)=
 \begin{cases}
 L\Rea(|x|+iy)^p,&|y|\le|x|/c,\\
 F_p(x,y),&|y|\ge|x|/c.
 \end{cases}
\end{equation}
For $1<p<2$, set $\beta=\pi/(2p)$, $c=\tan\beta=C_p$,
$M=\sin^{p-1}\beta/\cos\beta$, and define
\begin{equation}\label{eq:Upless}
 U_p(x,y)=
 \begin{cases}
 F_p(x,y),&|y|\le|x|/c,\\
 -M\Rea(|y|+ix)^p,&|y|\ge|x|/c.
 \end{cases}
\end{equation}
At $p=2$, take $U_2=x^2-y^2$.

\begin{proposition}\label{prop:coordinate}
The function $U_p$ is an admissible superharmonic majorant of $F_p$, even in each
coordinate. Moreover $\partial_y^2U_p\le0$ distributionally and
$U_p(z)\le0$ exactly when $|\Rea z|\le C_p|\Ima z|$.
Consequently, for every real $K\ge C_p$,
\[
 U_{p,K}(x,y)=U_p(x,(K/C_p)y)
\]
is an admissible superharmonic majorant of $|x|^p-K^p|y|^p$,
nonpositive exactly on $\Sigma_{p,K}$.
\end{proposition}
\begin{proof}
For $p\ne2$, write $\vartheta=\pi/(2p)$, $b=\cot\vartheta$,
$L=\cos^{p-1}\vartheta/\sin\vartheta$, and
$E_p(t)=L\Rea(1+it)^p-1+b^pt^p$ on $0\le t\le\tan\vartheta$.
For $\varphi=\arctan t$,
\begin{align*}
 q(\varphi)&=\frac{\sin((p-1)\varphi)}{\sin^{p-1}\varphi},&
 q'(\varphi)&=-\frac{(p-1)\sin((p-2)\varphi)}{\sin^p\varphi},\\
 E_p'(t)&=pt^{p-1}(b^p-Lq(\varphi)),&
 q(\vartheta)&=b^p/L.
\end{align*}
Since $|(p-2)\varphi|<\pi/2$,
$\operatorname{sgn}(p-2)E_p'\le0$, and the endpoint identities
$E_p(\tan\vartheta)=E_p'(\tan\vartheta)=0$ imply
$\operatorname{sgn}(p-2)E_p\ge0$.
On the harmonic sector the difference $U_p-F_p$ is
$|x|^pE_p(|y|/|x|)$ for $p>2$, and
$-b^{-p}|y|^pE_p(|x|/|y|)$ for $p<2$.
Thus $U_p\ge F_p$, with limiting axis values understood.

The endpoint identities give $h_p=U_p|_{\T}\in C^1(\T)$ with no angular atoms;
its piecewise second derivative is $O(1+|\theta-\theta_0|^{p-2})$ at the axes.
Hence $h_p\in W^{2,q}(\T)$ for $1<q<1/(2-p)$ when $p<2$,
and any finite $q>1$ otherwise. These choices satisfy $q(p-2)+2>0$;
Lemma~\ref{lem:polar} supplies the planar regularity.
On harmonic sectors $\Delta U_p=0$. Elsewhere
\[
 \Delta U_p=p(p-1)(|x|^{p-2}-c^p|y|^{p-2})\le0:
\]
for $p>2$ use $|y|\ge|x|/c$, and for $p<2$ use $|y|\le|x|/c$
with the negative exponent. In either case the second term is at
least $c^2|x|^{p-2}$ and $c\ge1$.
At $p=2$, $U_2$ is harmonic and $\partial_y^2U_2=-2$. The sign description follows from the
cosine on the harmonic sector and from $F_p$ on its complement.
On contact sectors, $\partial_y^2U_p=-C_p^pp(p-1)|y|^{p-2}\le0$.
On harmonic sectors the same derivative is
$-Lp(p-1)\Rea z^{p-2}$ for $p>2$, with $z=|x|+iy$, and
$-Mp(p-1)\Rea z^{p-2}$ for $p<2$, with $z=|y|+ix$.
In both cases $|\arg z|\le\pi/(2p)$ and
$|(p-2)\arg z|<\pi/2$, so the derivative is nonpositive.
The established $C^1$ and Sobolev regularity excludes singular
interface terms. For $t=K/C_p\ge1$, almost everywhere,
\[
 \Delta[U_p(x,ty)]
 =(\Delta U_p)(x,ty)+(t^2-1)(\partial_y^2U_p)(x,ty)\le0.
\]
Linear change of variables preserves the regularity, and the
majorization and sign assertions follow directly from those of $U_p$.
\end{proof}
Consequently \eqref{eq:coordinate-remainder} has exactly the two remainder terms
in \eqref{eq:certificate}, with no further operator argument.
Its spectral correction can be dropped when $\spec T\subset\Sigma_p$,
or whenever the signed eigenvalue sum is nonpositive.
The Laplacian vanishes on a full sector, so positivity of $\Psi_T$
alone gives no uniform stability lower bound.

\section{Weyl comparison and normalization}\label{sec:boundaries}
The spectral transfer theorem in Section~\ref{sec:measure} also identifies
a compactly supported Weyl distribution and quantifies its difference
from the eigenvalue measure. We record those consequences and the
normalizations used in their potential-theoretic interpretation.

\begin{example}\label{ex:scalarmeasure}
For scalar $T=a+ib$ one can take
$\mu_{a,b}=\int_0^1(1-r)r^{-1}\delta_{(ra,rb)}\dd r$ on
$\R^2\setminus\{0\}$. Its positive radial moments are finite, whereas
its mass is infinite if $(a,b)\ne(0,0)$.
For $T=1$, reflection of this measure preserves every absolute directional
moment but changes $\int x\dd\mu$ from $1/2$ to $-1/2$.
Thus the full ridge identity, not just power moments, is necessary.
\end{example}

The distribution below is the scalar trace of Anderson's Weyl
functional calculus \cite{Anderson69,Anderson70}, characterized by
its Fourier transform $\tr e^{i(\xi_1A+\xi_2B)}$. Its relation to
tracial joint spectral measures is discussed in
\cite[Sections~0.1.1--0.1.3]{HeinThesis}. The corollary records
its radial reconstruction from $\mu_{A,B}$ and its comparison with
the eigenvalue measure through $\Psi_T$.

\begin{corollary}\label{prop:Weyl}
With $\EulerQ=\Euler(\Euler+1)$,
\begin{equation}\label{eq:EulerW}
 \langle \mathfrak W_T,g\rangle=d g(0)+\int\EulerQ g\dd\mu_{A,B}
\end{equation}
defines a real compactly supported distribution of order at most two,
with mass $d$ and directional spectral marginals; in particular
$\langle \mathfrak W_T,e^{i(\xi_1x+\xi_2y)}\rangle
=\tr e^{i(\xi_1A+\xi_2B)}$.
For smooth real $g$ near all the supports,
\begin{align}
 \langle \mathfrak W_T,g\rangle-\sum_jg(\lambda_j(T))
 &=\frac1{2\pi}\int\Psi_T\Delta g\dd m_2,\label{eq:weylcomparison}\\
 \Delta\Psi_T&=2\pi\left(\mathfrak W_T-\sum_j\delta_{\lambda_j(T)}\right),
 \qquad(\log|\cdot|)*\mathfrak W_T=L_T.\label{eq:distributionalidentity}
\end{align}
Furthermore,
\begin{equation}\label{eq:psimass}
 \int_\C\Psi_T\dd m_2
 =\frac\pi2\left(\|T\|_2^2-\sum_j|\lambda_j(T)|^2\right).
\end{equation}
Thus $\Psi_T=0$ almost everywhere if and only if $T$ is normal.
\end{corollary}
\begin{proof}
Since $\EulerQ g=z^{\mathsf T}D^2g\,z+2z\cdot\nabla g$,
the first two radial moments give the order and support of $\mathfrak W_T$.
The right-inverse calculation in Lemma~\ref{lem:smoothcomparison}
and the ridge identity give its mass and directional marginals.
Equation~\eqref{eq:weylcomparison} is exactly
\eqref{eq:smoothcomparison}. Its distributional form gives the first
identity in \eqref{eq:distributionalidentity}; convolution with
$\log|\cdot|$ and compact support of $\Psi_T$ give the second.

For $g(z)=|z|^2$, one has $\EulerQ g=6g$ and
$\int\EulerQ g\dd\mu_{A,B}=\tr(A^2+B^2)=\|T\|_2^2$
by \eqref{eq:moments}. Since $\Delta g=4$,
\eqref{eq:smoothcomparison} proves \eqref{eq:psimass}.
Schur triangularization identifies this spectral difference with the
squared Hilbert--Schmidt norm of the strictly upper-triangular part;
together with $\Psi_T\ge0$, this proves the normality assertion.
\end{proof}

For a finite matrix, write
\begin{equation}\label{eq:henricideparture}
 \delta_2(T)^2=\|T\|_2^2-\sum_j|\lambda_j(T)|^2.
\end{equation}
The number $\delta_2(T)$ is Henrici's departure from normality
\cite{Henrici}. Thus \eqref{eq:psimass} represents its square as
the mass of a nonnegative potential supported in the numerical range.

\begin{corollary}
\label{cor:weylquant}
For a finite matrix $T$ and every real-valued
$g\in C^\infty(\C)$,
\begin{equation}\label{eq:weylquant}
 \left|\langle\mathfrak W_T,g\rangle
                 -\sum_jg(\lambda_j(T))\right|
 \le\frac{\delta_2(T)^2}{4}
          \sup_{w\in W(T)}|\Delta g(w)|.
\end{equation}
The numerical factor $1/4$ is optimal.
\end{corollary}
\begin{proof}
By \eqref{eq:weylcomparison}, nonnegativity of $\Psi_T$, and its
support property \eqref{eq:numericalrangesupport}, the left side
is at most
$(2\pi)^{-1}\sup_{W(T)}|\Delta g|\int\Psi_T\dd m_2$.
Equation~\eqref{eq:psimass} gives \eqref{eq:weylquant}.
For $g(z)=|z|^2$, one has $\Delta g=4$ and the difference is
$\delta_2(T)^2$. Any nonnormal matrix therefore proves optimality.
\end{proof}
The comparison concerns a distribution and the eigenvalue measure;
it does not assert positivity of $\mathfrak W_T$.

\begin{corollary}\label{cor:radialquant}
For a finite matrix $T=A+iB$, put
$K_T=\{rw:0\le r\le1,\ w\in\overline{W(T)}\}$.
Let $D$ be as in
Theorem~\ref{thm:nonhomogeneous}. Every real-valued smooth $f$ on $D$
with $f(0)=0$ satisfies
\begin{equation}\label{eq:radialquant}
 \left|\sum_j(\Hs f)(\lambda_j(T))-\int f\dd\mu_{A,B}\right|
 \le\frac{\delta_2(T)^2}{24}
           \sup_{w\in K_T}|\Delta f(w)|.
\end{equation}
The numerical factor $1/24$ is optimal.
\end{corollary}
\begin{proof}
Since $D$ is centered at zero and contains $\supp\mu_{A,B}$,
\eqref{eq:tracialradius} implies $K_T\subset D$.
Apply Corollary~\ref{cor:weylquant} to $g=\Hs f$, with an outer
cutoff if necessary. Since $g(0)=0$ and $\EulerQ g=f$,
\eqref{eq:EulerW} identifies its defect with minus the radial defect.
The identity $\Delta g(w)=\int_0^1r(1-r)\Delta f(rw)\dd r$ gives
$\sup_{W(T)}|\Delta g|\le\frac16\sup_{K_T}|\Delta f|$.
Thus $1/24=(1/4)(1/6)$; the quadratic test $f(z)=|z|^2$ gives
equality for every nonnormal $T$ and proves optimality.
\end{proof}

\begin{example}\label{ex:nilpotentpotential}
Let $N=\left(\begin{smallmatrix}0&2\\0&0\end{smallmatrix}\right)$.
For $E=N-wI$ the matrix $(w/\bar w)E^{-1}E^*$
has determinant one and trace $2-4/|w|^2$. The Jensen formula therefore
gives
\[
 \Psi_N(w)=
 \begin{cases}2\operatorname{arcosh}(1/|w|),&0<|w|<1,\\
 0,&|w|\ge1.\end{cases}
\]
Its mass is $2\pi$ by \eqref{eq:psimass}, although the eigenvalue
correction is zero. Also
$\langle \mathfrak W_N,e^{i\xi\cdot z}\rangle=2\cos|\xi|$, so
\[
 \langle \mathfrak W_N,1\rangle=2,\qquad
 \langle \mathfrak W_N,|z|^2\rangle=4,\qquad
 \langle \mathfrak W_N,|z|^4\rangle=16/3.
\]
Hence $\langle \mathfrak W_N,(1-|z|^2)^2\rangle=-2/3$.
A nonnegative cutoff equal to one on the support shows that $\mathfrak W_N$
is not positive. This nonpositivity test is the Pauli-matrix example
in \cite[Section~0.1.2, Example~2]{HeinThesis}, after simultaneous
unitary conjugation and division of its quartic polynomial by six.
The calculation above also displays its nonzero positive defect
potential. Finally $\|N\|_2=2$ but
$\|\Rea N\|_2=\|\Ima N\|_2=\sqrt2$:
the whole operator cannot replace its real part in \eqref{eq:main}.
\end{example}

\begin{remark}\label{rem:brown}
For this comparison only, normalize the trace by $\tau=d^{-1}\tr$.
The Fuglede--Kadison determinant \cite{FK} of an invertible matrix $S$
is $\operatorname{Det}_\tau(S)=\exp\tau(\log|S|)=|\det S|^{1/d}$.
The Brown measure of $T$ \cite{Brown} is the normalized eigenvalue measure
$\nu_T^{\mathrm B}=d^{-1}\sum_j\delta_{\lambda_j(T)}$. For
$w\notin\spec T$,
\[
 \log\operatorname{Det}_\tau(T-wI)
 =\int\log|w-z|\dd\nu_T^{\mathrm B}(z).
\]
By \eqref{eq:distributionalidentity}, the normalized defect satisfies
\begin{equation}\label{eq:browncomparison}
 \frac1d\Psi_T=(\log|\cdot|)*
       \left(\frac1d\mathfrak W_T-\nu_T^{\mathrm B}\right)\ge0
       \quad\text{almost everywhere}.
\end{equation}
Thus the logarithmic potential of the trace Weyl distribution dominates
that of the Brown measure, although the Weyl distribution itself can
fail positivity. This gives an operator-algebraic interpretation of
the finite-dimensional spectral defect. Matsaev--Sodin already pointed
to Brown's determinant theory as a setting for extensions of spectral
inequalities \cite{MS}. Extending this real-directional mechanism to
a finite von Neumann algebra would require a compatible positive ridge
representation and the corresponding integrability and comparison
identities; it is a further problem beyond the real-directional matrix and compact
Schatten results proved here.
\end{remark}

\section{Hardy-space preliminaries and structural inputs}\label{rz:prelim}
The conditional expectation extends contractively to every $L^p$; its trace adjoint is itself. The multiplicativity of $\Phi$ and $\tau\circ\Phi=\tau$ show that $\B$ is orthogonal to $\B_0^*$ in $L^2$: if $U\in\B$ and $b\in\B_0$, then $\tau bU=\tau\Phi(b)\Phi(U)=0$. Consequently the decomposition of $\mathcal C$ is unique.

Moreover, $\B^*=\B_0^*+\mathcal D$, so $\mathcal C=\B+\B^*$. This core is norm dense in each $L^p$, $1\leq p<\infty$. Indeed, if a dual functional annihilates it, that functional is given by trace pairing with an element of $L^{p'}\subset L^1$ (or an element of $\M$ for $p=1$), hence is normal on $\M$ and vanishes by weak-star density. Hahn--Banach proves the claim. Taking closures gives the orthogonal decompositions
\begin{equation}\label{rz:orthodecomp}
 L^2(\M)=H^2\oplus(H_0^2)^*,\qquad
 H^2=H_0^2\oplus L^2(\mathcal D).
\end{equation}
Thus $P_2$ is the orthogonal projection onto $H^2$, agrees with $P$ on $\mathcal C$, and is self-adjoint for the trace pairing. If $g=g^*\in L^2$, write $g=a+a^*+d$, where $a\in H_0^2$ and $d=\Phi g=d^*$. Then
\begin{equation}\label{rz:canonicalidentities}
 \mathcal T_{\B}g:=(2P_2-\Phi)g=2a+d\in H^2,
 \qquad \Rea \mathcal T_{\B}g=g,\qquad\Phi \mathcal T_{\B}g=\Phi g.
\end{equation}
Multiplication on either side by a bounded element of $\B$ preserves $H^p$, by approximation and H\"older.

Products of one $H_0^2$ element and one $H^2$ element have zero trace, by bounded analytic approximation and H\"older. We next prove the bounded intersection identity. Put
\[
 \mathcal A=H^2\cap\M
 =\{x\in\M:\tau(bx)=0\text{ for every }b\in\B_0\},
\]
where the second description follows from \eqref{rz:orthodecomp}.
These normal equations show that $\mathcal A$ is ultraweakly closed.
It is a unital algebra containing $\B$: for $x,y\in\mathcal A$, take
$y_j\in\B$ with $y_j\to y$ in $L^2$. The right $\B$ module property
puts $xy_j\in H^2$, and boundedness of $x$ gives $xy_j\to xy$ in $L^2$.
If $x\in\mathcal A\cap\mathcal A^*$, the orthogonal decomposition
forces $x\in L^2(\mathcal D)$, hence $x=\Phi x\in\mathcal D$.
For $x,y\in\mathcal A$, put $x_0=x-\Phi x\in H_0^2$. For $d\in\mathcal D$,
\[
 \tau(d^*x_0y)=\tau(x_0yd^*)=0,
\]
since $yd^*\in H^2$. Thus $\Phi(x_0y)=0$, and
$\Phi(xy)=\Phi x\,\Phi y$.
The inclusion $\B\subset\mathcal A$ supplies weak-star density of
$\mathcal A+\mathcal A^*$. Therefore $\mathcal A$ is a subdiagonal
algebra containing $\B$, and maximality gives $\mathcal A=\B$.
Finally, if $x\in H^q\cap\M$ with $q\geq1$, then
$\tau(bx)=0$ for every $b\in\B_0$, by $L^1$ convergence of its analytic
approximants. Since $x\in L^2$, orthogonality in
\eqref{rz:orthodecomp} gives $x\in H^2$.
We have proved
\begin{equation}\label{rz:selfcontained_intersection}
 H^q\cap\M=\B,\qquad 1\leq q<\infty.
\end{equation}

For clarity we state the published structural inputs used in the proof. These are results about finite subdiagonal algebras, rather than assumptions about the particular ambient algebra.
\begin{theorem}\label{rz:standard}
\begin{enumerate}
\item If $w\in\M$ is positive and boundedly invertible, there exists $R\in\B$ with $R^{-1}\in\B$ and $R^*R=w$.
\item Define the Fuglede--Kadison determinant by
\[
 \operatorname{Det}_\tau(x)=\lim_{\varepsilon\downarrow0}
 \exp\bigl(\tfrac12\tau\log(x^*x+\varepsilon\Id)\bigr).
\]
It is multiplicative on bounded operators, and $\operatorname{Det}_\tau(\Phi b)\leq\operatorname{Det}_\tau(b)$ for $b\in\B$. The same assertion holds in the adjoint algebra $\B^*$.
\item Every finite trace-preserving subdiagonal algebra is maximal.
\end{enumerate}
\end{theorem}
The positive factorization in (1) is Arveson's
Corollary~4.2.4(ii) \cite{rz:Arveson}.
For (2), determinant multiplicativity and its analytic extension
to noninvertible bounded elements are recalled in
\cite[Section~3]{rz:deLaHarpe}; see also \cite{FK}.
The general Jensen inequality, including noninvertible $b\in\B$,
follows directly from \cite[Corollary~3.3]{rz:BlecherLabuschagne},
which applies to all $h\in H^1$.
Labuschagne's Szeg\H{o} theorem and Jensen formula
\cite{rz:Labuschagne}, together with
\cite[Theorem~4.4.3]{rz:Arveson}, also give the bounded Jensen
inequality. Applying it to the adjoint algebra gives the assertion
for $\B^*$. Assertion (3) is Exel's Theorem~7 \cite{rz:Exel}.
When $\operatorname{Det}_\tau(x)>0$,
$\tau\log(x^*x)=2\log\operatorname{Det}_\tau(x)$ is a finite
spectral integral, even if $x$ is not boundedly invertible.
We use no inverse of $x$ in that case.

\section{Analytic regularization and its nonsharp background}\label{rz:analyticbackground}
The operator $\mathcal T_{\B}$ and its real-part identities are given in \eqref{rz:canonicalidentities}. Write
$\mathcal H_{\B}=-i(\mathcal T_{\B}-\Id)$ for conjugation. The three orthogonal components in \eqref{rz:orthodecomp} give
\[
 \mathcal H_{\B}^2=-\Id+\Phi,\qquad \Phi\mathcal H_{\B}=0,
 \qquad\|\mathcal H_{\B}x\|_2\leq\|x\|_2.
\]
For self-adjoint $g$, $\mathcal H_{\B}g$ is self-adjoint and $\mathcal T_{\B}g=g+i\mathcal H_{\B}g$. The following auxiliary estimates justify analytic regularization
independently of the sharp theorem. The moment estimate first puts
the powers of the canonical element in $H^1$. Yosida approximation
then identifies their exponential series with the contraction
semigroup, establishing analytic resolvent membership before any
sharp $L^p$ estimate is used.

\subsection{A nonsharp moment estimate used only for regularization}
\begin{lemma}\label{rz:selfcontained_moment}
For dyadic $q=2^n$, $n\geq1$, the conjugation satisfies
\[
 \|\mathcal H_{\B}u\|_q\leq(q-1)\|u\|_q
 \qquad(u=u^*\in L^q).
\]
Consequently, for bounded $g=g^*\in\M$ and $h=\mathcal T_{\B}g$,
$h\in H^q$ at every dyadic exponent, and for every integer $k\geq1$,
\begin{equation}\label{rz:selfcontained_growth}
 h^k\in H^1,\qquad
 \|h^k\|_1\leq(2k\|g\|_\infty)^k.
\end{equation}
\end{lemma}
\begin{proof}
The $q=2$ bound was proved above. Assume a bound with constant $C_q$
on self-adjoint core elements. For $u=u^*\in\mathcal C$, let
$v=\mathcal H_{\B}u$, $d=\Phi u$ and $f=u+iv=\mathcal T_{\B}u\in\B$.
Then $\Phi f=d$ is self-adjoint. Since
\[
 f^2=u^2-v^2+i(uv+vu),\qquad \Phi(f^2)=d^2,
\]
conjugation of its real part yields
$\mathcal H_{\B}(u^2-v^2)=uv+vu$. Applying $\mathcal H_{\B}$ once more gives
\begin{equation}\label{rz:selfcontained_Cotlar}
 v^2=u^2+\mathcal H_{\B}(uv+vu)-d^2.
\end{equation}
Both $u^2-v^2$ and $uv+vu$ belong to the bounded core, as the real and
imaginary parts of the analytic element $f^2$.
Contractivity of conditional expectation in $L^{2q}$,
\eqref{rz:selfcontained_Cotlar}, and H\"older give
\[
 \|v\|_{2q}^2
 \leq2\|u\|_{2q}^2+2C_q\|u\|_{2q}\|v\|_{2q}.
\]
Solving the quadratic gives $C_{2q}\leq C_q+\sqrt{C_q^2+2}$.
If $C_q\leq q-1$ and $q\geq2$, then
\[
 C_{2q}\leq q-1+\sqrt{(q-1)^2+2}\leq2q-1.
\]
Induction proves the stated bounds on the self-adjoint core.
Core density extends them to self-adjoint $L^q$, consistently with
$\mathcal H_{\B}$ on $L^2$ because $q\geq2$.
Taking self-adjoint core approximants to $g$ in $L^q$ proves
$h=\mathcal T_{\B}g\in H^q$, with $\|h\|_q\leq q\|g\|_q$.
Choose a dyadic $q\geq k$ with $q\leq2k$, allowing $q=2$ for $k=1$.
Then $\|h\|_k\leq2k\|g\|_\infty$, and H\"older gives the norm estimate
in \eqref{rz:selfcontained_growth}. Products of bounded analytic
approximants converging to $h$ in $L^q$ converge to $h^k$ in
$L^{q/k}\subset L^1$, proving $h^k\in H^1$.
\end{proof}

\subsection{Bounded analytic accretive approximants: a proof}
We use the standard maximally accretive semigroup theorem in its usual
operator form: a maximally accretive operator $A$ generates a contraction
semigroup $e^{-tA}$, its resolvent is the Laplace integral of that
semigroup, and its bounded Yosida approximants
$A_n=A(\Id+A/n)^{-1}$ have semigroups converging strongly to it.
These statements are the Lumer--Phillips theorem, the Laplace resolvent formula, and the Yosida approximation theorem; see \cite[Theorems II.3.15, II.1.10 and Section III.4.10]{rz:EngelNagel}. They are independent of Hardy or
Riesz bounds. The following proof establishes all the analytic
membership and approximation assertions needed here.

\begin{lemma}\label{rz:selfcontained_approx}
Let $g\in\M$ be bounded and positive and put $h=\mathcal T_{\B}g$, $d=\Phi g$.
For every $\varepsilon>0$,
\[
 R_\varepsilon=(\Id+\varepsilon h)^{-1}\in\B,
 \qquad \|R_\varepsilon\|_\infty\leq1,
\]
and
\[
 b_\varepsilon=hR_\varepsilon\in\B,
 \qquad\Rea b_\varepsilon\ge0,
 \qquad\Phi b_\varepsilon=d(\Id+\varepsilon d)^{-1}\ge0.
\]
Moreover $b_\varepsilon\to h$ in $L^2$, and hence in $L^s$ for
$1<s<2$, while $\Rea b_\varepsilon\to g$ in the same spaces.
\end{lemma}
\begin{proof}
Put $K=\mathcal H_{\B}g$, viewed as a self-adjoint affiliated
operator, and realize $h=g+iK$ on $\operatorname{Dom}K$.
Since $g$ is bounded, $h$ is closed and $h^*=g-iK$ on the same
domain. Both operators are accretive. For
$\lambda>\|g\|_\infty$, the factorization
\[
 \lambda+h=
 [\Id+g(\lambda+iK)^{-1}](\lambda+iK)
\]
is onto by a Neumann series, since
$\|g(\lambda+iK)^{-1}\|\le\|g\|_\infty/\lambda<1$.
Thus $h$, and likewise $h^*$, is maximally accretive.
Affiliation and uniqueness of the generated semigroup imply that
its resolvents and contraction semigroup commute with every
unitary in the commutant and hence belong to $\M$.
In particular, $\|(\Id+\varepsilon h)^{-1}\|\le1$.

We prove that its contraction semigroup $S_t=e^{-th}$ lies in $\B$.
By \eqref{rz:selfcontained_growth}, for
$0<t<(2e\|g\|_\infty)^{-1}$ the formal series
\[
 \sum_{k=0}^\infty\frac{(-t)^kh^k}{k!}
\]
converges in $H^1$: the estimate $k!\geq(k/e)^k$ makes its norm
majorant geometric. If $g=0$, the assertion is immediate instead.
To identify the series, let $R_n=(\Id+h/n)^{-1}$ and $h_n=hR_n$.
The bounded Yosida semigroups $e^{-th_n}$ converge strongly to $S_t$
and are contractions. Resolvents commute with $h$ on its domain, so
in the algebra of measurable operators
\[
 h_n^k=h^kR_n^k,
 \qquad\|h_n^k\|_1\leq\|h^k\|_1.
\]
Both $h$ and $h^*$ are maximally accretive, so $R_n\to\Id$
strongly and $R_n^*\to\Id$ strongly. Therefore bounded
multiplier convergence gives $h_n^k\to h^k$ in $L^1$ for every fixed
$k$. This multiplier convergence follows directly by truncating a
fixed $L^1$ element to bounded operators and using the finite trace.
Dominated convergence of the exponential series therefore gives
convergence in $L^1$ to the displayed $H^1$ series. Strong convergence
of the uniformly bounded semigroups, applied to the identity vector
in the standard $L^2$ representation, also gives convergence in
$L^2$ and hence $L^1$ to $S_t$. Consequently $S_t\in H^1\cap\M=\B$
for small positive $t$. The semigroup law then yields $S_t\in\B$ for
all $t>0$.

The resolvent formula
\[
 R_\varepsilon=\int_0^\infty e^{-t}S_{\varepsilon t}\,dt
\]
now shows $R_\varepsilon\in\B$, since $\B$ is ultraweakly closed and
the integrand has norm at most $e^{-t}$.
The identity $b_\varepsilon=hR_\varepsilon=(\Id-R_\varepsilon)/\varepsilon$
shows that $b_\varepsilon$ is bounded and analytic.
Multiplicativity extends to a bounded analytic factor times an $H^2$
factor by $L^2$ approximation. Applying it to
$R_\varepsilon(\Id+\varepsilon h)=\Id$ gives
\[
 \Phi R_\varepsilon=(\Id+\varepsilon d)^{-1},\qquad
 \Phi b_\varepsilon=d(\Id+\varepsilon d)^{-1}\ge0.
\]
Using $(\Id+\varepsilon h)R_\varepsilon=\Id$ and $\Rea h=g$ gives the bounded quadratic-form identity
\[
 \Rea b_\varepsilon
 =R_\varepsilon^*gR_\varepsilon+
    \varepsilon b_\varepsilon^*b_\varepsilon\ge0.
\]
All terms are bounded since $g$ and $b_\varepsilon$ are bounded.
Finally $R_\varepsilon,R_\varepsilon^*\to\Id$ strongly. Uniform boundedness
and truncation of $h\in L^2$ give $b_\varepsilon=hR_\varepsilon\to h$
in $L^2$, and taking real parts gives the remaining convergence assertions.
\end{proof}

\section{Rotations and independent coordinates}\label{sec:furtherrealizations}
The real-pencil and joint realization theorems have further
consequences obtained by rotation and by iteration over independent
variables. We record the full statements, including one-component
Schatten membership and the canonical completely bounded norms.

\begin{corollary}\label{cor:twodirections}
Let $(c,d)\ne(0,0)$. The optimal constant, uniform over finite
nilpotent $T=A+iB$, in
$\|aA+bB\|_p\le L\|cA+dB\|_p$ is
\begin{equation}\label{eq:twodirections}
 L=K_p\left(\frac{ac+bd}{c^2+d^2},
                 \frac{bc-ad}{c^2+d^2}\right).
\end{equation}
The same bound holds for compact quasinilpotent $T\in\Sp_p$.
It also holds if $T$ is compact and quasinilpotent and only
$cA+dB\in\Sp_p$ is assumed.
\end{corollary}
\begin{proof}
Put $r=\sqrt{c^2+d^2}$ and rotate $T$ so that its self-adjoint parts are
$A'=(cA+dB)/r$ and $B'=(-dA+cB)/r$.
Then $aA+bB=((ac+bd)/r)A'+((bc-ad)/r)B'$.
Corollary~\ref{thm:realpencil} and homogeneity prove the formula.
The rotation bijects the class of finite nilpotents, so sharpness
transfers in both directions. If $cA+dB=0$ for a nilpotent $T$,
the rotated operator is skew-Hermitian and nilpotent and hence zero.
The weaker membership assumption follows from the qualitative
Matsaev theorem \cite{Matsaev61,GK70,MS}, applied to $i$ times
the rotated operator.
\end{proof}

\begin{corollary}\label{cor:products}
For $1\le j\le d$, let $\Omega_j$ be $\R$ with Lebesgue measure or
$\Z$ with counting measure. Let $\mathbf S_j=(S_{j,k})_{k=1}^{m_j}$
be respectively a joint Hilbert-pencil family or a shifted discrete
family from Theorem~\ref{thm:unified}, with common norm $L_j$.
Put $\Omega=\prod_j\Omega_j$, $m=\prod_jm_j$, and
\[
 \mathbf S=\left(\prod_{j=1}^d S_{j,k_j}^{(j)}\right)_{k_1,\ldots,k_d}.
\]
For every nonzero Hilbert space $\mathcal K$, this map from
$L^p(\Omega;\Sp_p(\mathcal K))$ to
$\ell_p^m(L^p(\Omega;\Sp_p(\mathcal K)))$ has norm $\prod_jL_j$.
Its scalar and canonical completely bounded scalar norms have the
same value. For unweighted single-output families,
$L_j=K_p(a_j,b_j)$ on $\R$ and
$L_j=K_p(\cos(\pi\alpha_j),\sin(\pi\alpha_j))$ on $\Z$.
\end{corollary}
\begin{proof}
Apply each one-coordinate estimate to the corresponding slices,
and sum the other coordinates and output indices by Fubini--Tonelli.
Iteration gives the upper product bound, equally for mixed discrete
and continuous coordinates. Products of scalar approximate extremizers,
multiplied by a fixed rank-one projection, give the lower bound because
both input and joint output norms factor. A zero factor is immediate.
The same estimates hold for every $\Sp_p^n$ coefficient space, so
Lemma~\ref{lem:commonamplification} gives the complete norm equality.
\end{proof}

\section*{Declaration of AI use}
The author used OpenAI models as assistive tools in preparing this
manuscript. Mathematical arguments, proofs, and computations in the
manuscript were verified and finalized by the author.

\raggedbottom


\begin{thebibliography}{99}
\small\raggedright
\setlength{\itemsep}{0pt}
\bibitem{BKRT} R. Ba\~nuelos and M. Kwa\'snicki,
\emph{The $\ell^p$ norm of the Riesz--Titchmarsh transform for even integer $p$},
J. Lond. Math. Soc. \textbf{109} (2024), no.~4, e12888.
\href{https://doi.org/10.1112/jlms.12888}{doi:10.1112/jlms.12888}.
\bibitem{Hein} O. Hein\"avaara,
\emph{Tracial joint spectral measures},
Invent. Math. \textbf{239} (2025), 505--526.
\href{https://doi.org/10.1007/s00222-024-01309-6}{doi:10.1007/s00222-024-01309-6}.
\bibitem{HeinThesis} O. Hein\"avaara,
\emph{Tracial joint spectral measures},
Ph.D. thesis, Princeton University, 2024.
\href{https://arks.princeton.edu/ark:/88435/dsp01z316q4969}{Princeton repository}.
\bibitem{Anderson69} R. F. V. Anderson,
\emph{The Weyl functional calculus},
J. Functional Analysis \textbf{4} (1969), 240--267.
\href{https://doi.org/10.1016/0022-1236(69)90013-5}{doi:10.1016/0022-1236(69)90013-5}.
\bibitem{Anderson70} R. F. V. Anderson,
\emph{On the Weyl functional calculus},
J. Functional Analysis \textbf{6} (1970), 110--115.
\href{https://doi.org/10.1016/0022-1236(70)90050-9}{doi:10.1016/0022-1236(70)90050-9}.
\bibitem{FK} B. Fuglede and R. V. Kadison,
\emph{Determinant theory in finite factors},
Ann. of Math. (2) \textbf{55} (1952), no.~3, 520--530.
\href{https://doi.org/10.2307/1969645}{doi:10.2307/1969645}.
\bibitem{Brown} L. G. Brown,
\emph{Lidskii's theorem in the type II case},
in \emph{Geometric Methods in Operator Algebras} (Kyoto, 1983),
Pitman Res. Notes Math. Ser., vol.~123, Longman Scientific \& Technical,
Harlow, 1986, pp.~1--35.
\bibitem{Henrici} P. Henrici,
\emph{Bounds for iterates, inverses, spectral variation and fields of values
of non-normal matrices},
Numer. Math. \textbf{4} (1962), 24--40.
\href{https://doi.org/10.1007/BF01386294}{doi:10.1007/BF01386294}.
\bibitem{HKV} B. Hollenbeck, N. J. Kalton, and I. E. Verbitsky,
\emph{Best constants for some operators associated with the Fourier and Hilbert transforms},
Studia Math. \textbf{157} (2003), 237--278.
\href{https://doi.org/10.4064/sm157-3-2}{doi:10.4064/sm157-3-2}.
\bibitem{DGZ} Y. Ding, L. Grafakos, and K. Zhu,
\emph{On the norm of the operator $aI+bH$ on $L^p(\mathbb R)$},
Bull. Korean Math. Soc. \textbf{55} (2018), no.~4, 1209--1219.
\href{https://doi.org/10.4134/BKMS.b170676}{doi:10.4134/BKMS.b170676}.
\bibitem{LS} F. L\'arusson and R. Sigurdsson,
\emph{Plurisubharmonicity of envelopes of disc functionals on manifolds},
J. Reine Angew. Math. \textbf{555} (2003), 27--38.
\href{https://doi.org/10.1515/crll.2003.013}{doi:10.1515/crll.2003.013}.
\bibitem{Laeng} E. Laeng,
\emph{Remarks on the Hilbert transform and on some families of multiplier
operators related to it},
Collect. Math. \textbf{58} (2007), no.~1, 25--44.
\bibitem{Zhou} W. Zhou,
\emph{Properties and applications of a conjugate transform on Schatten classes},
J. Inequal. Appl. \textbf{2022} (2022), article 126.
\href{https://doi.org/10.1186/s13660-022-02863-4}{doi:10.1186/s13660-022-02863-4}.
\bibitem{Pichorides} S. K. Pichorides,
\emph{On the best values of the constants in the theorems of M.~Riesz,
Zygmund and Kolmogorov},
Studia Math. \textbf{44} (1972), 165--179.
\href{https://doi.org/10.4064/sm-44-2-165-179}{doi:10.4064/sm-44-2-165-179}.
\bibitem{Essen} M. Ess\'en,
\emph{A superharmonic proof of the M.~Riesz conjugate function theorem},
Ark. Mat. \textbf{22} (1984), 241--249.
\href{https://doi.org/10.1007/BF02384381}{doi:10.1007/BF02384381}.
\bibitem{Poletsky} E. A. Poletsky,
\emph{Holomorphic currents},
Indiana Univ. Math. J. \textbf{42} (1993), 85--144.
\bibitem{Rosay} J.-P. Rosay,
\emph{Poletsky theory of disks on holomorphic manifolds},
Indiana Univ. Math. J. \textbf{52} (2003), 157--169.
\bibitem{NR} S. Neuwirth and \'{E}. Ricard,
\emph{Transfer of Fourier multipliers into Schur multipliers and sumsets in a discrete group},
Canad. J. Math. \textbf{63} (2011), no.~5, 1161--1187.
\href{https://doi.org/10.4153/CJM-2011-053-9}{doi:10.4153/CJM-2011-053-9}.
\bibitem{Pisier} G. Pisier,
\emph{Non-commutative vector valued $L_p$-spaces and completely $p$-summing maps},
Ast\'erisque \textbf{247} (1998).
\bibitem{Kato} T. Kato,
\emph{Perturbation Theory for Linear Operators},
Classics in Mathematics, Springer, Berlin, 1995.
\bibitem{Matsaev61} V. I. Matsaev,
\emph{Volterra operators obtained from self-adjoint operators by perturbation},
Dokl. Akad. Nauk SSSR \textbf{139} (1961), no.~4, 810--813;
English translation, Soviet Math. Dokl. \textbf{2} (1961), 1013--1016.
\href{https://www.mathnet.ru/eng/dan25330}{Math-Net.Ru record}.
\bibitem{GK68} I. Ts. Gohberg and N. Ya. Krupnik,
\emph{Norm of the Hilbert transformation in the $L_p$ space},
Funct. Anal. Appl. \textbf{2} (1968), no.~2, 180--181.
\href{https://doi.org/10.1007/BF01075955}{doi:10.1007/BF01075955}.
\bibitem{GK70} I. C. Gohberg and M. G. Krein,
\emph{Theory and Applications of Volterra Operators in Hilbert Space},
Translations of Mathematical Monographs, vol.~24,
American Mathematical Society, Providence, RI, 1970.
\href{https://doi.org/10.1090/mmono/024}{doi:10.1090/mmono/024}.
\bibitem{MS} V. Matsaev and M. Sodin,
\emph{Entire functions and compact operators with $S_p$-imaginary component},
in \emph{Entire Functions in Modern Analysis: Boris Levin Memorial Conference},
Israel Math. Conf. Proc., vol.~15, Bar-Ilan University, Ramat Gan,
2001, pp.~243--260.
\bibitem{MOS} V. Matsaev, I. Ostrovskii, and M. Sodin,
\emph{The Hilbert transform, rearrangements, and logarithmic determinants},
Banach Center Publ. \textbf{56} (2002), 95--105.
\href{https://doi.org/10.4064/bc56-0-8}{doi:10.4064/bc56-0-8}.

\bibitem{rz:Arveson} W. B. Arveson, Analyticity in operator algebras, \emph{Amer. J. Math.} 89 (1967), 578--642. \href{https://doi.org/10.2307/2373237}{doi:10.2307/2373237}.

\bibitem{rz:Exel} R. Exel, Maximal subdiagonal algebras, \emph{Amer. J. Math.} 110 (1988), 775--782. \href{https://doi.org/10.2307/2374650}{doi:10.2307/2374650}.

\bibitem{rz:FK} T. Fack and H. Kosaki, Generalized $s$-numbers of $\tau$-measurable operators, \emph{Pacific J. Math.} 123 (1986), 269--300. \href{https://doi.org/10.2140/pjm.1986.123.269}{doi:10.2140/pjm.1986.123.269}.

\bibitem{rz:deLaHarpe} P. de la Harpe, Fuglede--Kadison determinant: theme and variations, \emph{Proc. Natl. Acad. Sci. USA} 110 (2013), 15864--15877. \href{https://doi.org/10.1073/pnas.1202059110}{doi:10.1073/pnas.1202059110}.

\bibitem{rz:EngelNagel} K.-J. Engel and R. Nagel, \emph{One-Parameter Semigroups for Linear Evolution Equations}, Graduate Texts in Mathematics 194, Springer, New York, 2000. \href{https://doi.org/10.1007/b97696}{doi:10.1007/b97696}.

\bibitem{rz:Labuschagne} L. E. Labuschagne, A noncommutative Szeg\H{o} theorem for subdiagonal subalgebras of von Neumann algebras, \emph{Proc. Amer. Math. Soc.} 133 (2005), 3643--3646. \href{https://doi.org/10.1090/S0002-9939-05-08064-0}{doi:10.1090/S0002-9939-05-08064-0}.

\bibitem{rz:BlecherLabuschagne} D. P. Blecher and L. E. Labuschagne,
Applications of the Fuglede--Kadison determinant:
Szeg\H{o}'s theorem and outers for noncommutative $H^p$,
\emph{Trans. Amer. Math. Soc.} \textbf{360} (2008), no.~11,
6131--6147.
\href{https://doi.org/10.1090/S0002-9947-08-04506-6}{doi:10.1090/S0002-9947-08-04506-6}.

\bibitem{rz:Rand} N. Randrianantoanina, Hilbert transform associated with finite maximal subdiagonal algebras, \emph{J. Austral. Math. Soc. Ser. A} 65 (1998), 388--404. \href{https://doi.org/10.1017/S1446788700035953}{doi:10.1017/S1446788700035953}.

\bibitem{rz:HV} B. Hollenbeck and I. E. Verbitsky, Best constants for the Riesz projection, \emph{J. Funct. Anal.} 175 (2000), 370--392. \href{https://doi.org/10.1006/jfan.2000.3616}{doi:10.1006/jfan.2000.3616}.

\bibitem{rz:JRS} M. Junge, Z.-J. Ruan and D. Sherman, A classification for 2-isometries of noncommutative $L^p$-spaces, \emph{Israel J. Math.} 150 (2005), 285--314. \href{https://doi.org/10.1007/BF02762384}{doi:10.1007/BF02762384}.
\end{thebibliography}
\end{document}